\documentclass[preprint,10pt]{elsarticle}
\newcommand{\dataFunction}{\MC{T}}
\newcommand{\forward}{{\boldsymbol q}}
\newcommand{\stforward}{\overline{\forward}}
\newcommand{\MC}{\mathcal}

\newcommand{\expectation}[1]{\mathrm{E}[#1]}
\newcommand{\tindx}{n}

\newcommand{\snapshotArg}[1]{\boldsymbol S(#1)}
\newcommand{\svdLeft}{\boldsymbol U}
\newcommand{\svdLeftVecArg}[1]{\boldsymbol u_{#1}}
\newcommand{\svdMid}{\boldsymbol \Sigma}
\newcommand{\svdRight}{\boldsymbol V^T}

\newcommand{\spatialvar}{\mathsf x}
\newcommand{\spatialvarTwo}{\mathsf y}

\newcommand{\physDomain}{\Omega}

\newcommand{\depth}{d}

\newcommand{\gpCovarianceAll}{\bar{\boldsymbol \Sigma}^n(\parametervecTest)}
\newcommand{\gpCovarianceTrain}{{\boldsymbol \Sigma_{\text{train}}^n}}
\newcommand{\gpCovarianceTest}{{\boldsymbol
\Sigma_{\text{test}}^n}(\parametervecTest)}

\newcommand{\gpNoise}{{\lambda}}
\newcommand{\gpLength}{{h}}

\newcommand{\velocityLF}{ \boldsymbol {f}_\text{LF} }
\newcommand{\velocityFOM}{ \boldsymbol f}
\newcommand{\parametervec}{\boldsymbol \mu}
\newcommand{\parametervecTest}{\boldsymbol \mu}
\newcommand{\parametervecTrainArg}[1]{\boldsymbol \mu_{#1}}
\newcommand{\parameter}{\mu}

\newcommand{\param}{\parametervec}
\newcommand{\nparams}{N_{\parametervec}}
\newcommand{\parameterdomain}{\MC{D}}
\newcommand{\parameterdomainROM}{\parameterdomain_\text{ROM}}
\newcommand{\parameterdomainDum}{\overline{\MC{D}}}

\newcommand{\parameterdomainTrain}{\parameterdomain_\text{train}}

\newcommand{\parameterdomainVal}{\parameterdomain_\text{val}}
\newcommand{\parameterdomainTest}{\parameterdomain_\text{test}}

\newcommand{\nfom}{N}
\newcommand{\statevecIntercept}{ {\boldsymbol x}_\text{ref} }

\newcommand{\statevec}{ \boldsymbol x }
\newcommand{\statevecHFM}{\statevec}
\newcommand{\genstatevecROM}{\hat{\boldsymbol {x}}}
\newcommand{\statevecROM}{\genstatevecROM}

\newcommand{\statevecLF}{\statevec_\text{LF}}
\newcommand{\statevecLFInit}{\statevec_\text{LF,0}}
\newcommand{\timeArg}[1]{t^{#1}}
\usepackage{mathtools}
\newcommand{\defeq}{\vcentcolon=}

\newcommand{\matshapea}{
\tikz[baseline=0.0ex]\draw (0,0) rectangle (20pt,2pt);}

\newcommand{\matshapebf}{
\tikz[baseline=0pt]\draw (0pt,0pt) rectangle (20pt,20pt);}
\newcommand{\matshapec}{
\tikz[baseline=0.0ex]\draw (0,0pt) rectangle (2pt,20pt);}

\newcommand{\basismat}{\boldsymbol {\Phi}}
\newcommand{\basismatVecArg}[1]{\boldsymbol {\phi}_{#1}}
\newcommand{\kroneckerVecArg}[1]{\boldsymbol {e}_{#1}}
\newcommand{\testmat}{\boldsymbol {\Psi}}
\newcommand{\dimrom}{K}
\newcommand{\dimLF}{{N_\text{LF}}}
\newcommand{\prolongate}{\mathbf{p}}

\newcommand{\sampMat}{\mathbf{P}}
\newcommand{\sampMatvec}[1]{\mathbf{p}_{#1}}

\newcommand{\nsample}{n_s}

\newcommand{\ntrainROM}{N_{\text{train,ROM}}}
\newcommand{\nskip}{N_{\text{skip}}}

\newcommand{\nparamTrain}{N_s}

\newcommand{\featurevec}{\boldsymbol{\rho}}
\newcommand{\feature}{\rho}
\newcommand{\featurenmu}{\boldsymbol \feature^\tindx(\parametervec)}

\newcommand{\approxregressionfunction}{ {\hat{f}}}

\newcommand{\latentstate}{\boldsymbol h}

\newcommand{\latentstatefnmu}{\latentstate^\tindx(\parametervec)}

\newcommand{\latentstatefn}{\latentstate^\tindx}

\newcommand{\latentstatefnmuArg}[1]{\latentstate^{#1}(\parametervec)}
\newcommand{\latentstatefnArg}[1]{\latentstate^{#1}}
\newcommand{\latentstatezero}{\latentstate_0}

\newcommand{\weights}{\boldsymbol{\theta}}
\newcommand{\weightsf}{\boldsymbol{\theta}_f}
\newcommand{\weightsg}{\boldsymbol{\theta}_{\boldsymbol g}}
\newcommand{\weightsmat}{\boldsymbol{\Theta}}

\newcommand{\regTermf}{\kappa_f}
\newcommand{\regTermg}{\kappa_{\boldsymbol g}}

\newcommand{\arxarga}{p}
\newcommand{\arxargb}{r}
\newcommand{\arxargs}{(\arxarga,\arxargb)}

\newcommand{\ARXfweightsOne}{\weights_1}

\newcommand{\ARXfweightsTwo}{\theta_2}

\newcommand{\cellstate}{\boldsymbol c}
\newcommand{\latentstateh}{\tilde{\latentstate}}

\newcommand{\bias}{\boldsymbol b}
\newcommand{\scalarbias}{ b}

\newcommand{\nneurons}{p}
\newcommand{\activation}{\boldsymbol z}
\newcommand{\activationf}{\boldsymbol \zeta}
\newcommand{\activationfscalar}{\zeta}
\newcommand{\activationRecurrent}{\boldsymbol r}

\newcommand{\weightsActivation}{\boldsymbol W}
\newcommand{\weightsActivationVec}{\boldsymbol w}
\newcommand{\nlayers}{\depth}

\newcommand{\npca}{n_r}

\newcommand{\dualweightedresidual}{d}

\newcommand{\paramtrainindex}{\MC{D}_{\text{train}}}

\newcommand{\trainSet}{\MC{T}_{\text{train}}}
\newcommand{\trainSetKNN}{\MC{T}_{\text{kNN}}}
\newcommand{\testSet}{\MC{T}_{\text{test}}}
\newcommand{\testSetNoise}{\MC{T}_{\text{test},\noise}}
\newcommand{\trainSetNoise}{\MC{T}_{\text{train},\noise}}

\newcommand{\valSet}{\MC{T}_{\text{val}}}
\newcommand{\residualTrainSet}{\MC{T}_{\residual,\text{train}}}
\newcommand{\residualSnapshotMatrix}{\boldsymbol S_{\residual}}

\newcommand{\paramtrainset}{\MC{D}_{\text{train}}}
\newcommand{\paramtrainsetNoise}{\MC{D}_{\text{train},\noise}}
\newcommand{\paramtestsetNoise}{\MC{D}_{\text{test},\noise}}

\newcommand{\noise}{\epsilon}
\newcommand{\noiseDummy}{\epsilon}
\newcommand{\noisenArg}[1]{\epsilon^{#1}}

\newcommand{\approxnoisefunction}{\hat \epsilon}

\newcommand{\residual}{\boldsymbol r}
\newcommand{\residualLF}{\boldsymbol r_\text{LF}}

\newcommand\norm[1]{\left\lVert#1\right\rVert_2}

\newcommand{\methodName}{Time-Series Machine-Learning Error Modeling}
\newcommand{\methodAcronym}{T-MLEM}

\newcommand{\RR}[1]{\mathbb R^{#1}}
\newcommand{\RRstar}[1]{\mathbb R_\star^{#1}}
\newcommand{\ntimesteps}{N_t}

\newcommand{\kArg}[1]{k^{#1}}
\newcommand{\stateDummy}{\boldsymbol w}

\newcommand{\stateDummyLF}{\stateDummy_\text{LF}}
\newcommand{\stateApprox}{\tilde\statevec}
\newcommand{\qoi}{s}

\newcommand{\qoiArg}[1]{s^{#1}}
\newcommand{\qoiFunc}{g}
\newcommand{\qoiFuncArgs}[3]{g(#1;#2,#3)}
\newcommand{\qoiApprox}{\tilde s}

\newcommand{\qoiApproxArg}[1]{\qoiApprox^{#1}}
\newcommand{\qoiErrorArg}[1]{\delta_{\qoi}^{#1}}
\newcommand{\qoiError}{\delta_{\qoi}}
\newcommand{\stateError}{\delta_{\statevec}}
\newcommand{\stateErrorArg}[1]{\delta_{\statevec}^{#1}}
\newcommand{\state}{\statevec}
\newcommand{\tol}{\varepsilon}
\newcommand{\range}[1]{\text{Ran}(#1)}
\newcommand{\identity}{\mathbf{I}}

\newcommand{\lipschitz}{\kappa}
\newcommand{\lipschitzQoI}{\lipschitz_\qoiFunc}

\newcommand{\hconstant}{h}  
\newcommand{\gammaconstantj}{\gamma_j}
\newcommand{\velo}{\velocityFOM}
\newcommand{\bm}[1]{\mathbf{#1}}
\newcommand{\nstate}{N}
\newcommand{\abs}[1]{\left|#1\right|}

\newcommand{\stateROMArg}[1]{\stateApprox^{#1}}
\newcommand{\resnFOM}{\boldsymbol{r}^n}        
\newcommand{\resnROMParamsArg}[1]{{\boldsymbol{r}}^n\left(#1;\stateApprox^{n-1}\paramDep,\ldots,\stateApprox^{n-\kArg{n}}\paramDep,\parametervec\right)}       

\newcommand{\zero}{\bm{0}}
\newcommand{\paramDep}{(\parametervec)}
\newcommand{\timeIndexSet}{\mathbb T}

\newcommand{\timeIndexSetTrainNoise}{\mathbb T_{\text{train},\noise}}
\newcommand{\timeIndexSetTestNoise}{\mathbb T_{\text{test},\noise}}

\newcommand{\timeIndexSetDum}{\overline{\mathbb T}}
\newcommand{\timeIndexSetTrain}{{\mathbb T}_{\text{train}}}
\newcommand{\timeIndexSetVal}{{\mathbb T}_{\text{val}}}

\newcommand{\timeIndexSetTest}{{\mathbb T}_{\text{test}}}

\newcommand{\timeIndexMappingNo}{\tau}
\newcommand{\timeIndexMapping}[1]{\timeIndexMappingNo(#1)}
\newcommand{\ntimestepsCoarse}{\bar\ntimesteps}
\newcommand{\card}[1]{|#1|}
\newcommand{\nfeatures}{N_{\feature}}
\newcommand{\nlatentf}{N_{\latentstate}}

\newcommand{\resApprox}{\tilde\residual}
\newcommand{\resApproxArg}[1]{\tilde\residual^{#1}(\parametervec)}
\newcommand{\resApproxArgParam}[2]{\tilde\residual^{#1}(#2)}

\newcommand{\resApproxArgNo}[1]{\tilde\residual^{#1}}
\newcommand{\resPCAApproxArg}[1]{\hat\residual^{#1}(\parametervec)}
\newcommand{\resGappyApproxArg}[1]{\hat\residual_g^{#1}(\parametervec)}

\newcommand{\errorGen}{\delta}
\newcommand{\meanErrorGen}{\overline{\errorGen}}

\newcommand{\errorGennmu}{\errorGen^n(\parametervec)}
\newcommand{\errorGennmuArg}[1]{\errorGen^{#1}(\parametervec)}

\newcommand{\errorGennTrainArg}[1]{\boldsymbol \errorGen^{#1}_{\text{train}}}
\newcommand{\errorGennAllArg}[1]{\bar {\boldsymbol \errorGen}^{#1}}
\newcommand{\errorGennArg}[1]{\errorGen^{#1}}

\newcommand{\errorModelGen}{\hat \delta}
\newcommand{\errorModelGenn}{\errorModelGen^n}
\newcommand{\errorModelGennmu}{\errorModelGen^n(\parametervec)}

\newcommand{\errorApproxGennmuArg}[1]{\hat\errorGen^{#1}(\parametervec)}
\newcommand{\errorApproxMeanGen}{\hat\errorGen_f}
\newcommand{\errorApproxMeanGenn}{\hat\errorGen^n_f}
\newcommand{\errorApproxMeanGennmunArg}[1]{\hat\errorGen_f^{#1}(\parametervec)}
\newcommand{\errorApproxMeanGennnArg}[1]{\hat\errorGen_f^{#1}}

\newcommand{\errorApproxMeanGennmu}{\hat\errorGen^n_f(\parametervec)}
\newcommand{\errorApproxRandGenn}{\hat{\errorGen}_{\noise}^n}
\newcommand{\errorApproxRandGennmu}{\hat{\errorGen}_{\noise}^n(\parametervec)}
\newcommand{\errorApproxRandGennmunArg}[1]{\hat{\errorGen}_{\noise}^{#1}(\parametervec)}
\newcommand{\errorApproxRandGen}{\hat{\errorGen}_{\noise}}
\newcommand{\errorApproxRandGenDummy}{\hat{\errorGen}_{\noise}}

\newcommand{\recursionf}{\mathbf{g}}

\newcommand{\featuref}{\featurevec}

\newcommand{\featurefnmuArg}[1]{\featuref^{#1}(\param)}

\newcommand{\featurefnArg}[1]{\featuref^{#1}}

\newcommand{\spacetimeOp}[1]{\bar{#1}}
\newcommand{\dual}{\boldsymbol y}
\newcommand{\stdual}{\spacetimeOp{\dual}}

\newcommand{\stresidual}{\spacetimeOp{\residual}}
\newcommand{\ststate}{\spacetimeOp{\state}}
\newcommand{\stt}{\spacetimeOp{t}}
\newcommand{\stqoiFunc}{\spacetimeOp{ \qoiFunc}}
\newcommand{\ststateApprox}{\spacetimeOp{\stateApprox}}
\newcommand{\stqoi}{\spacetimeOp{s}}
\newcommand{\stqoiApprox}{\spacetimeOp{\qoiApprox}}

\newcommand{\vecdummy}{\boldsymbol y}

\newcommand{\unroll}{\text{vec}}
\newcommand{\ststateDummy}{\spacetimeOp{\stateDummy}}

\newcommand{\dummyArg}{\boldsymbol w}

\newcommand{\indexSpace}{\mathcal{H}}
\newcommand{\paramDepoptional}{(\parametervec)}

\newcommand{\normalDistribution}{\MC{N}}
\newcommand{\laplaceDistribution}{\MC{L}}
\newcommand{\std}{\sigma}
\newcommand{\randomvars}[1]{\mathbb V^{#1}}

\usepackage{enumitem}
\newlist{Objective}{enumerate}{2}
\setlist[Objective]{label={Objective \arabic*},align=left}

\newlist{Step}{enumerate}{2}
\setlist[Step]{label={Step \arabic*},align=left}

\newlist{Consideration}{enumerate}{2}
\setlist[Consideration]{label={Consideration \arabic*},align=left}

\newlist{Feature}{enumerate}{2}
\setlist[Feature]{label={Feature \arabic*},align=left}

\newlist{Category}{enumerate}{2}
\setlist[Category]{label={Category \arabic*},align=left}

\usepackage{fullpage}
\usepackage{color}
\usepackage{amsmath}
\usepackage{amsfonts}
\usepackage{amsthm}
\usepackage{mathtools} 
\theoremstyle{definition}

\newtheorem{proposition}{Proposition}[section]
\newtheorem{theorem}{Theorem}[section]

\newtheorem{remark}[theorem]{Remark}
\usepackage[titletoc,title]{appendix}
\usepackage{float}
\usepackage{graphicx} 
\usepackage{epstopdf} 
\usepackage{caption}
\usepackage{subcaption}
\usepackage{comment}
\usepackage{hyperref}

\usepackage{tikz}
\usetikzlibrary{positioning, fit, arrows.meta, shapes}

\usetikzlibrary{
  arrows.meta, 
  fit, 
  positioning, 
}
\tikzset{
  neuron/.style={ 
    circle,draw,thick, 
    inner sep=0pt, 
    minimum size=3.5em, 
    node distance=1ex and 2em, 
  },
  group/.style={ 
    rectangle,draw,thick, 
    inner sep=0pt, 
  },
  io/.style={ 
    neuron, 
    fill=gray!15, 
  },
  conn/.style={ 
    -{Straight Barb[angle=60:2pt 3]}, 
    thick, 
  },
}
\usepackage{listings}
\usepackage{algorithm,algorithmic}
\usepackage{pifont}
\usepackage{float}
\usepackage[english]{babel} 
\usepackage{comment}
\definecolor{mygreen}{rgb}{0,0.6,0}
\definecolor{mygray}{rgb}{0.5,0.5,0.5}
\definecolor{mymauve}{rgb}{0.58,0,0.82}
\usepackage[font=small,labelfont=bf,
   justification=justified,format=plain]{caption}

\begin{document}

\begin{frontmatter}
\title{Time-series machine-learning error models for approximate solutions to
	parameterized dynamical systems}

\author[a]{Eric J.\ Parish}
\ead{ejparis@sandia.gov}
\author[a]{Kevin T.\ Carlberg}
\ead{ktcarlb@sandia.gov} 

\address[a]{Sandia National Laboratories, 7011 East Ave, Livermore, CA 94550}


\begin{abstract}
This work proposes a machine-learning framework for modeling the error
	incurred by approximate solutions to parameterized dynamical systems. In
	particular, we extend the machine-learning error models (MLEM) framework
	proposed in Ref.~\cite{freno_carlberg_ml} to dynamical systems. The proposed
	\methodName\ (\methodAcronym) method constructs a regression model that maps
	features---which comprise error indicators that are derived from standard
	\textit{a posteriori} error-quantification techniques---to a random variable for
	the approximate-solution error at each time instance. The proposed framework
	considers a wide range of candidate features, regression methods, and
	additive noise models.
	We consider primarily recursive regression
	techniques developed for time-series modeling, including both classical
	time-series models (e.g., autoregressive models) and recurrent neural
	networks (RNNs), but also analyze standard non-recursive regression
	techniques
	(e.g., feed-forward neural networks) for comparative purposes.
	Numerical experiments conducted on multiple benchmark problems
	illustrate that the long short-term memory (LSTM) neural network, which is a
	type of RNN, outperforms other methods and yields substantial improvements
	in error predictions over traditional approaches.
\end{abstract}
\end{frontmatter}

\def\layersep{2.5cm}

\section{Introduction}

Myriad applications in engineering and science require many simulations of a
(parameterized) dynamical-system model.  Examples of such ``many-query''
problems include uncertainty propagation, design optimization, and statistical
inference.  When the cost of a single dynamical-system simulation is
computationally expensive---as is the case for partial-differential-equation
(PDE) models characterized by a fine spatiotemporal discretization---such
many-query problems are intractable.  In these cases, analysts often replace
the original (high-fidelity) dynamical-system model with a surrogate model
that can generate low-cost \textit{approximate solutions}, which can in turn
make the many-query problem computationally tractable. Examples of surrogate
models include \textit{inexact solutions} arising from early termination of an
iterative method, \textit{lower-fidelity models} characterized by simplifying
physical assumptions (e.g., linearized dynamics, inviscid flow, elastic
deformation) or a coarse discretization, and projection-based
\textit{reduced-order models (ROMs)} (e.g., Galerkin projection with a proper
orthogonal decomposition basis).

The use of an approximate solution introduces error; as such, it is critical to quantify this error and
properly account for it in the analysis underpinning the many-query
problem. For this purpose, researchers have developed a wide range of
methods for 
\textit{a posteriori}
error quantification.
These efforts, which date back to the pioneering work 
of Babu\v{s}ka and Rheinboldt for finite element analysis~\cite{babuska_rheinboldt_error_1978a,babuska_rheinboldt_error_1978b,ainsworth_oden_fem_error_review},
have resulted in a techniques 
 that can be categorized as (1) \textit{a posteriori} error bounds, (2) error
 indicators, and (3) error models.

\textit{A posteriori error bounds} place bounds on
the (normed) state or quantity-of-interest (QoI) error arising from
the use of an approximate solution. These methods were originally developed in 
the finite-element~\cite{paraschivoiu_bounds97,paraschivoiu_bounds98,maday_2000} and
model-reduction communities. In the reduced-basis community,
Refs.~\cite{reduced_basis_rovas,reliable_prediction_rb}
derived \textit{a
posteriori} QoI error bounds for the reduced-basis method applied to linear
elliptic and parabolic PDEs. 
Refs.~\cite{grepl_patera_apost,rovas_apost} extended these results to
time-dependent linear parabolic PDEs. In the context of model reduction for nonlinear dynamical systems,
Ref.~\cite{homescu_romlinbounds} derived error bounds for
the linearized Galerkin ROM, while
Ref.~\cite{carlberg_lspg_v_galerkin} derived error bounds
for Galerkin and least-squares Petrov-Galerkin (LSPG) ROMs.
These error bounds are 
\textit{residual-based}, meaning that they depend
on the (dual) norm of the high-fidelity-model residual evaluated at the
approximate
solution. Such bounds are appealing because they can provide guaranteed,
rigorous
bounds on the errors of interest, thus providing a critical tool for
certified predictions. However, these bounds often suffer from 
lack of sharpness, i.e., they are often orders-of-magnitude
larger than the approximate-solution error itself. This is
exacerbated in dynamical systems, where error bounds for many types of  
approximate solutions grow
exponentially in time (see, e.g., Ref.~\cite{carlberg_lspg_v_galerkin}).
In addition, these
bounds typically require (difficult-to-compute) estimates or bounds of operator
quantities such as Lipschitz and inf--sup constants. These drawbacks often limit
the utility of \textit{a posteriori} error bounds in practical applications.

Alternatively, \textit{error indicators} are computable quantities that are 
\textit{indicative} of the approximate-solution
error. Error indicators typically do not provide unbiased
estimates of the error, nor do they rigorously bound the error;
instead, they are often correlated with the error, correspond to a low-order
approximation of the error, and/or
appear as terms in error bounds.
The two most common 
error indicators are the (dual) norm of the high-fidelity-model residual
evaluated at the approximate solution (i.e., the residual norm) and the
dual-weighted residual. The residual norm is informative of the
error, as it typically appears in \textit{a posteriori} error bounds and is
usually strongly correlated with the error; it is
often used within
greedy methods for snapshot collection 
\cite{bui2008parametric,bui2008model,hinze2012residual,
wu2015adaptive,yano2018lp} and when using ROMs for PDE-constrained
optimization within a trust-region framework
\cite{zahr2015progressive,zahrThesis,zahr2018efficient}.
In contrast, dual-weighted residuals are derived from a
Taylor-series approximation of the error, and are typically used for
goal-oriented error estimation and mesh adaptation~\cite{lu_phd_thesis,ackmann_errorestimation,VENDITTI2000204,
VENDITTI200240}. 
Due to their practical utility, error indicators have been largely successful
in quantifying and controlling errors through mesh adaptation for
static problems (i.e., problems without time evolution). However, their
success has been more 
limited for dynamical systems due to the fact that errors for such
systems exhibit dependence on \textit{non-local} quantities, i.e.,
the approximate-solution error at a given time instance depends on the past
time history of the system.
Thus, the residual norm 
at the current time instance is no longer directly indicative of the
approximate-solution error; the error additionally depends on non-local
quantities. This can be clearly seen from \textit{a posteriori} error bounds,
as the bound for the error at a given time instance depends on the residual
norm at previous time instances. Similarly, dual-weighted-residual error estimation requires
computing solutions to either the forward-in-time linearized sensitivity
equations or the backward-in-time linearized adjoint equations.
Not only does this incur significant
computational and storage costs, it can also be linearly unstable and produce
exploding gradients when applied to chaotic systems (although this can be mitigated by
approaches such as least-squares
shadowing~\cite{blonigan_thesis,shimizu_aiaa18}).
Further, it is often practically challenging to implement adjoint-based
methods, e.g., due to the requirement of residual-Jacobian
transposes, which are not always easily exposed.

\textit{Error models} seek to model the state and QoI errors
directly via regression techniques.  The most popular such approach is the
so-called ``model-discrepancy'' method of Kennedy and O'Hagan~\cite{ken_ohag} (and related approaches~\cite{huang_kriging,eldred_correction,march_optimize,gano_kriging}). The
model-discrepancy approach computes a Gaussian process (GP) that provides a
mapping from any finite number of points in parameter space to a Gaussian
random vector representing the approximate-solution error at those points. As
such, this technique can be considered a regression approach, where the
\textit{regression model} is provided by a GP, the \textit{features} are the
system parameters, and the \textit{response} is the approximate-solution
error.
Recently, researchers have pursued improvements over this technique
by
(1) considering a wider class of modern machine-learning regression methods
than simply Gaussian processes (which do not scale well to high-dimensional
feature spaces), and (2) considering more informative features than the system
parameters alone.  First, the reduced-order model error surrogates (ROMES)
method~\cite{carlberg_ROMES} was proposed in the context of static 
problems.
ROMES employs the same regression model (GP) as the model-discrepancy approach, but it employs
different features; namely, it employs the aforementioned \textit{error
indicators} (i.e., residual norm, approximated dual-weighted residual) as
features. Because these quantities are
indicative of the error, they provide more informative features for predicting
the error than the system parameters. Numerical experiments illustrated the
ability of 
the ROMES method to produce significantly more accurate
error predictions than the model-discrepancy approach. More recently, the
ROMES method was applied to
construct error models for the in-plane and out-of-plane state error incurred by
reduced-order models for static problems~\cite{romes_closure}.  Subsequent
work~\cite{freno_carlberg_ml} again considered static problems, but
constructed machine-learning error models that considered a wide range of
candidate regression models (e.g., neural networks, support vector machines,
random forests) and error-indicator-based features (e.g., residual
samples)
in order to predict the approximate-solution error response.
This work showed the ability of machine-learning methods to predict the error
with near perfect accuracy across a range of static problems and approximate-solution
types.

The construction of error models for dynamical systems is significantly more
challenging than doing so for static systems, as the errors exhibit
dependence on non-local quantities. Relevant work on constructing error models for
dynamical systems accounts for the time-dependent nature of the approximate-solution error either (1) by using both time itself and time-lagged error
indicators as features~\cite{trehan_ml_error} or (2) by constructing a different
regression method for each time instance or window~\cite{pagani_enkfrom}.  In
the first case, Ref.~\cite{trehan_ml_error} 
proposes the 
error modeling via machine learning (EMML) method. 
EMML is a regression approach, where the
{regression model} is provided by random forests or LASSO, the
{features} correspond to application-specific quantities that
include time itself and time-lagged quantities, and the {response} is the
(relative) QoI or state error at a given time instance.  While this work generated promising results on
an oil--water subsurface flow problem, it suffers from a noticeable
limitation: it treats error modeling as a classical (non-recursive)
regression problem as opposed to a time-series-prediction problem. As such,
it does not capture the intrinsic recursive structure of the error.  In the
second case, Ref.~\cite{pagani_enkfrom} proposes a regression model that is
tantamount to applying the model-discrepancy method \cite{ken_ohag} to model
the approximate-solution error at each
time instance or window, as 
the {regression model} corresponds to a GP, the {features} are
system parameters, and the {response} is the QoI error at a given time
instance/window. Unfortunately, this approach also fails to capture 
non-local dependencies of the error, and---as with the original model-discrepancy
method---does not use particularly informative features to
predict the error.  

The objective of this work is to construct accurate error models for dynamical
systems by (1) applying the same machine-learning error modeling framework of
Ref.~\cite{freno_carlberg_ml}, but (2) considering recursive regression
techniques developed for time-series modeling, as this naturally enables 
the method to capture the non-local dependencies of the error.
The proposed 
\methodName\ (\methodAcronym)
method, which views the error modeling problem from the 
{time-series-prediction} perspective, 
is characterized by a \textit{regression model}
provided by a recursive time-series-prediction model (e.g., autoregressive
model, recurrent neural network~\cite{rnn_rumelhart}),
\textit{features} corresponding to time-local error indicators (e.g., residual
samples), and a \textit{response} corresponding to the normed state
or QoI error at a given time instance.  For comparative purposes, we also consider classical
non-recursive regression methods (e.g., feed-forward neural networks) within
the \methodAcronym\ framework, which is analogous to the approach proposed in 
Ref.~\cite{trehan_ml_error}. To produce a statistical model of the
approximate-solution error, we also consider several noise models (e.g.,
Gaussian, Laplacian). The net
result is a random variable for the error whose statistics can be used to
assess the uncertainty in the error prediction.

%

The paper proceeds as follows.
Section~\ref{sec:dynamical_systems} sets the mathematical context
by introducing parameterized dynamical systems,
approximate solutions, and classical methods for error quantification.
Section~\ref{sec:ml_framework} outlines the proposed 
\methodAcronym\ method, including the structure of the
regression model, feature engineering, data generation, and training of the
regression and noise models.
Section~\ref{sec:numerical_experiments} provides numerical experiments that
apply the proposed framework to 1)
the parameterized advection--diffusion equation, where the approximate solution is provided
by a 
proper orthogonal decomposition (POD) ROM with Galerkin
projection, 2) the parameterized shallow-water equations, where the
approximate solution is again provided by a POD--Galerkin ROM, and 3) the
parameterized Burgers' equation, where the approximate solution is provided by
a coarse-mesh model. Finally, Section~\ref{sec:conclude} provides conclusions.

\section{Parametrized nonlinear dynamical systems}\label{sec:dynamical_systems}
This work considers the high-fidelity model to be a parameterized 
dynamical system characterized by a parameterized system of nonlinear ordinary
differential equations (ODEs) of the form
\begin{equation}\label{eq:fom}
\frac{d\statevecHFM }{dt} = \velocityFOM(\statevecHFM,t;\parametervec), \qquad
	\statevecHFM(0) = \statevec_0(\parametervec), \qquad t\in[0,T],
\end{equation}
where $T \in \mathbb{R}^+$ denotes the final time;
$\statevecHFM\equiv\statevecHFM(t,\parametervec)$ 
with 
$\statevecHFM:
[0,T]\times\parameterdomain\rightarrow
\mathbb{R}^N$
denotes the state implicitly defined as the solution to the initial-value problem
\eqref{eq:fom}; $\statevec_0:\parameterdomain\rightarrow
\RR{N}$ denotes the parameterized initial condition; $\parametervec \in \parameterdomain
\subseteq \RR{N_{\parametervec}}$ denote parametric inputs to the system; and
$\velocityFOM : \RR{N} \times [0,T]\times \parameterdomain \rightarrow \RR{N}
$ denotes the parameterized velocity, which may be linear or nonlinear in its first
argument. We refer to Eq.~\eqref{eq:fom} as the full-order-model system of ordinary
differential equations (FOM ODE), which
may arise, for example, from the spatial discretization of a system of PDEs.

Numerically computing a solution to Eq.~\eqref{eq:fom} requires applying a
time-discretization method. In this work, we assume the use of
a linear multistep method for this purpose. For notational simplicity, we
assume a uniform time discretization 
characterized by $\ntimesteps$ time instances 
$t^n = n\Delta t$, $n=0,\ldots,
\ntimesteps$ with 
time step
$\Delta t = T/\ntimesteps$; however, the proposed
methodology does not rely on this restriction.
Applying a linear multistep method to discretize the FOM ODE \eqref{eq:fom} yields the
following system of algebraic equations arising at the $n$th time instance,
which we refer to as the FOM
O$\Delta$E:
\begin{equation}\label{eq:fom_discretized}
\residual^{\tindx}(\statevec^{\tindx}; \statevec^{\tindx-1}, \hdots,
	\statevec^{\tindx - \kArg{\tindx}} , \parametervec) = \mathbf{0}, \qquad
	\tindx = 1,\ldots,\ntimesteps.
\end{equation}
The discrete state $\statevec^{\tindx}$ 
comprises the numerical approximation to
$\statevecHFM(t^\tindx)$
and is implicitly defined as the solution to
Eq.~\eqref{eq:fom_discretized} given the discrete state at
the previous $\kArg{\tindx}$ time instances and
the parameters $\parametervec$.
Here, the discrete residual
$\residual^{\tindx}: \mathbb{R}^N \otimes\RR{\kArg{\tindx}+ 1}\times
\parameterdomain \rightarrow \mathbb{R}^N$ is defined as
\begin{equation} \label{eq:FOMODeltaERes}
	\residual^{\tindx}:(\stateDummy;\statevec^{\tindx-1},\ldots,\statevec^{\tindx
	- \kArg{\tindx}},\parametervec)\mapsto
	\alpha_0\stateDummy - \Delta t\beta_0\velocityFOM(\stateDummy,t^\tindx;\parametervec)
	 + \sum_{i=1}^{\kArg{\tindx}}\alpha_i\statevec^{\tindx - i} - 
	 \Delta
	 t\sum_{i=1}^{\kArg{\tindx}}\beta_i\velocityFOM(\statevec^{\tindx-i},t^{\tindx-i};\parametervec).
\end{equation} 
Given the initial state $\statevec^{0}(\parametervec) = \statevec_0(\parametervec)$, solving
the FOM O$\Delta$E \eqref{eq:fom_discretized} for $n=1,\ldots,\ntimesteps$ yields the sequence of
(discrete) FOM solutions $\statevec^n(\parametervec)$, $n=0,\ldots,\ntimesteps$. 

In many cases, the analyst is primarily interested in computing a scalar-valued quantity
of interest (QoI) at each time instance, which we define as
\begin{equation} \label{eq:FOMQOI}
	\qoiArg{\tindx}:\parametervec\mapsto\qoiFuncArgs{\statevec^\tindx(\parametervec)}{\timeArg{\tindx}}{\parametervec},\quad
	n=0,\ldots,\ntimesteps,
\end{equation} 
where 
$\qoiArg{\tindx}:\parameterdomain\rightarrow\RR{}$ and
$\qoiFunc:\RR{N}\times[0,T]\times \parameterdomain\rightarrow\RR{}$ denotes the
QoI functional.

If the FOM state-space dimension $N$ or the number of time
instances $\ntimesteps$ is large, then computing the sequence of FOM
solutions $\statevec^n(\parametervec)$, $n=0,\ldots,\ntimesteps$ and
associated QoIs 
$\qoiArg{\tindx}(\parametervec)$, $\tindx=0,\ldots,\ntimesteps$
can be
computationally expensive. If additionally these solutions are sought at many
parameter instances, then the resulting \textit{many-query} problem is
computationally intractable.
In such scenarios, analysts typically aim to overcome this computational barrier by replacing the FOM with a surrogate model
that can generate low-cost \textit{approximate solutions}.

\subsection{Approximate solutions}

When a surrogate model is employed to mitigate the aforementioned
computational bottleneck, it generates a sequence of \textit{approximate
solutions} 
$\stateApprox^\tindx:\parameterdomain\rightarrow\RR{N}$,
$\tindx=0,\ldots,\ntimesteps$
with $\stateApprox^0(\parametervec)=\stateApprox_0(\parametervec)$ given
and attendant
approximate QoIs
\begin{equation} \label{eq:approxQOI}
	\qoiApproxArg{\tindx}:\parametervec\mapsto\qoiFuncArgs{\stateApprox^\tindx(\parametervec)}{\timeArg{\tindx}}{\parametervec},\quad
	n=0,\ldots,\ntimesteps,
\end{equation} 
where $\qoiApproxArg{\tindx}:\parameterdomain\rightarrow\RR{}$.
As in Ref.~\cite{freno_carlberg_ml}, we now discuss three types of surrogate
models that are often employed to generate the approximate solutions
$\stateApprox^\tindx(\parametervec)$, $\tindx=0,\ldots,\ntimesteps$.

\subsubsection{Inexact solutions}

If one applies an iterative method (e.g., Newton's method) to solve the FOM
O$\Delta$E system \eqref{eq:fom_discretized} arising at each time instance,
then inexact tolerances $\tol^\tindx>0$, $\tindx=1,\ldots,\ntimesteps$ can be
employed such that the approximate solutions satisfy
\begin{equation}
\|\residual^{\tindx}(\stateApprox^{\tindx}; \stateApprox^{\tindx-1}, \hdots,
	\stateApprox^{\tindx - \kArg{\tindx}} , \parametervec)\|_2 \leq\tol^\tindx,
	\qquad \tindx = 1,\ldots,\ntimesteps.
\end{equation}
In such cases, the surrogate model merely corresponds to early termination of
the associated iterative method.

\subsubsection{Lower-fidelity models}\label{sec:LFM}

Approximate solutions can also be obtained by employing a computational model
of lower fidelity, which can be achieved, for example, by introducing
simplifying physical assumptions or coarsening the spatial discretization.
In such cases, the corresponding low-fidelity-model
(LFM) ODE takes the form
\begin{equation}\label{eq:approx_model}
\frac{d\statevecLF}{dt}  = \velocityLF(\statevecLF,t;\parametervec),
	\qquad \statevecLF(0) = {\statevecLFInit}(\parametervec),\qquad
t\in[0,T],
\end{equation} 
where $\statevecLF\equiv \statevecLF(t,\parametervec)$
with $\statevecLF:[0,T]\times \parameterdomain\rightarrow
\RR{\dimLF}$ denotes the LFM state of dimension
$\dimLF(\ll N)$;
$\statevecLFInit:\parameterdomain\rightarrow\mathbb{R}^{\dimLF}$ denotes the
parameterized initial conditions; and $\velocityLF: \mathbb{R}^{\dimLF}
\times[0,T]\times \parameterdomain \rightarrow \mathbb{R}^{\dimLF}$ denotes
the parameterized LFM velocity. 

Applying a linear multistep method (with the same uniform time step 
$\Delta t$
as the FOM) to the
discretized LFM ODE \eqref{eq:approx_model} yields the
LFM O$\Delta$E arising at the $n$th time instance
\begin{equation}\label{eq:LF_discretized}
\residualLF^{\tindx}(\statevecLF^{\tindx}; \statevecLF^{\tindx-1}, \hdots,
	\statevecLF^{\tindx - \kArg{\tindx}} , \parametervec) = \mathbf{0}, \qquad
	\tindx = 1,\ldots,\ntimesteps,
\end{equation}
where the discrete state $\statevecLF^{\tindx}$ is the numerical approximation
to $\statevecLF(t^\tindx)$ and is implicitly defined as the solution to 
Eq.~\eqref{eq:LF_discretized}
given the LFM state at the previous $\kArg{\tindx}$ time instances and
the parameters $\parametervec$.
Analogously to definition \eqref{eq:FOMODeltaERes}, the LFM O$\Delta$E
residual $\residual^{\tindx}: \mathbb{R}^\dimLF \otimes\RR{\kArg{\tindx}+
1}\times \parameterdomain \rightarrow \mathbb{R}^N$ is defined as
\begin{equation} \label{eq:residualLF}
	\residualLF^{\tindx}:(\stateDummyLF;\statevecLF^{\tindx-1},\ldots,\statevecLF^{\tindx
	- \kArg{\tindx}},\parametervec)\mapsto
	\alpha_{\text{LF},0}\stateDummyLF - \Delta t\beta_{\text{LF},0}\velocityLF(\stateDummyLF,t^\tindx;\parametervec)
	 + \sum_{i=1}^{\kArg{\tindx}}\alpha_{\text{LF},i}\statevecLF^{\tindx - i} - 
	 \Delta
	 t\sum_{i=1}^{\kArg{\tindx}}\beta_{\text{LF},i}\velocityLF(\statevecLF^{\tindx-i},t^{\tindx-i};\parametervec).
\end{equation} 
We emphasize that the time integrators for the FOM and LFM can be different,
as denoted by the LF subscripts on the coefficients in definition \eqref{eq:residualLF}.
As with the FOM, given the initial LFM state $\statevecLF^0(\parametervec) =
{\statevecLFInit}(\parametervec)$, solving the LFM O$\Delta$E
\eqref{eq:LF_discretized} yields the sequence of (discrete) LFM solutions
$\statevecLF^{\tindx}$, $\tindx=0,\ldots,\ntimesteps$.

Critically, to recover the approximate solution in the FOM state space $\RR{N}$, we assume the
existence of a prolongation operator
$\prolongate:\RR{\dimLF}\times\parameterdomain\rightarrow\RR{N}$ such that 
\begin{equation} 
	\stateApprox^\tindx(\parametervec)  = 
	\prolongate( \statevecLF^\tindx(\parametervec);\parametervec),\quad
	\tindx=0,\ldots,\ntimesteps.
\end{equation} 

\subsubsection{Projection-based reduced-order models}\label{sec:romsec}
Projection-based reduced-order models (ROMs) comprise a popular technique for
generating approximate solutions. Such approaches can be derived by applying
projection to either the FOM ODE \eqref{eq:fom} or O$\Delta$E
\eqref{eq:fom_discretized}.
In the former case, ROMs correspond to a particular type of low-fidelity
model as described in Section \ref{sec:LFM}. In particular,
introducing a trial-basis matrix $\basismat\in\RRstar{N\times \dimrom}$ (e.g.,
computed via proper orthogonal decomposition \cite{berkooz_turbulence_pod}) and a
test-basis matrix $\testmat\in\RRstar{N\times\dimrom}$ (e.g., 
$\testmat = \basismat$
in the case of Galerkin projection), where $\dimrom\ll N$
and $\RRstar{m\times
n}$ denotes the set of $m\times n$ full-column-rank matrices (i.e., the
non-compact Stiefel manifold), ROMs seek an approximate solution in a
$\dimrom$-dimensional affine trial subspace, i.e.,
\begin{equation}\label{eq:ROM_POD}
	\statevec(t;\param) \approx \stateApprox(t;\param) = 
	\statevecIntercept(\param) + 
	\basismat
	\genstatevecROM(t;\param),
\end{equation}
where $\genstatevecROM: 
[0,T]\times\parameterdomain\rightarrow\RR{\dimrom}$ are the generalized coordinates and $\statevecIntercept : \parameterdomain \rightarrow \RR{\dimrom}$ 
denotes the reference state (typically chosen to be either $\statevecIntercept(\param) = \statevec_0(\param)$ or $\statevecIntercept(\param) = \boldsymbol 0$).
The generalized coordinates are computed by performing a projection process performed on the FOM ODE \eqref{eq:fom}. Namely, 
the approximation \eqref{eq:ROM_POD} is substituted in the FOM ODE \eqref{eq:fom}, and the
resulting residual is enforced to be orthogonal to range of the test-basis
matrix. This projection process yields
\begin{equation}\label{eq:ROM}
	\frac{d \statevecROM}{dt} = [\testmat^T\basismat]^{-1}\testmat^T
\velocityFOM (\statevecIntercept(\param) + 
	\basismat
	\genstatevecROM,t;\parametervec), \qquad
 \genstatevecROM(0;\param) =
 	\basismat^+
	(\statevec_0(\param) - \statevecIntercept(\param)),
\end{equation}
where the superscript $+$ denotes the Moore--Penrose pseudo-inverse.
This approach to model reduction is equivalent to a low-fidelity
model of the form \eqref{eq:approx_model} with $\statevecLF = \statevecROM$,
$\velocityLF:(\genstatevecROM,t;\parametervec)\mapsto
[\testmat^T\basismat]^{-1}\testmat^T
\velocityFOM (\statevecIntercept(\param) + 
	\basismat
	\genstatevecROM,t;\parametervec) $, and
	$\prolongate:(\genstatevecROM,\parametervec)\mapsto\statevecIntercept(\param) + 
	\basismat
	\genstatevecROM.$

The most common method that associates with applying projection to the FOM
O$\Delta$E \eqref{eq:fom_discretized} is the least-squares
Petrov--Galerkin (LSPG) method \cite{carlberg_lspg_v_galerkin}, which
substitutes the approximation \eqref{eq:ROM_POD} into the FOM
O$\Delta$E \eqref{eq:fom_discretized} and minimizes the resulting residual in
the $\ell^2$-norm. In this case, the approximate
solutions satisfy 
\begin{equation} 
	\stateApprox^\tindx \in\underset{\stateApprox\in 
	\statevecIntercept(\param) + 
	\range{\basismat}
	}{\mathrm{arg}\,\mathrm{min}}
	\|
\residual^{\tindx}(\stateApprox; \stateApprox^{\tindx-1}, \hdots,
	\stateApprox^{\tindx - \kArg{\tindx}} , \parametervec)
	\|_2.
\end{equation} 
It can be shown \cite{carlberg_lspg_v_galerkin} that
this LSPG approach can be expressed as a time-continuous projection
\eqref{eq:ROM} 
with $\testmat = \testmat(\genstatevecROM,t;\param)=(\alpha_0\identity - \Delta t\beta_0\partial \velocityFOM/\partial
\state(\statevecIntercept(\param) + 
	\basismat
	\genstatevecROM,t;\param))\basismat$
if $\beta_j=0$ for $j\geq 1$ (i.e., a single-step method is
used), the velocity $\velocityFOM$ is linear in its first argument, or if
$\beta_0=0$ (i.e., an explicit scheme is used).

\subsection{Approximate-solution errors}\label{sec:error}
Employing an approximate solution in lieu of the FOM solution incurs an error
that must be accounted for in the ultimate analysis. This work focuses on
quantifying two types of such errors, namely
\begin{enumerate}
	\item the quantity-of-interest error $ \qoiErrorArg{\tindx}(\parametervec)
	\defeq 
		\qoiArg{\tindx}(\parametervec) - 
		\qoiApproxArg{\tindx}(\parametervec)$, $\tindx=1,\ldots,\ntimesteps$ and
\item the normed state error $ 
	\stateErrorArg{\tindx}(\parametervec)	\defeq\norm{ \statevec^\tindx(\parametervec) -
		\stateApprox^\tindx(\parametervec) } $, $\tindx=1,\ldots,\ntimesteps$.
\end{enumerate}
Because 
$\stateApprox^0(\parametervec)$ and $\state^0(\parametervec)$ are given, we
can compute
$\qoiErrorArg{0}(\parametervec)$ and $\stateErrorArg{0}(\parametervec)$
explicitly.  

\subsection{Classical approaches for error
quantification}\label{sec:classicalError}


Classical approaches for quantifying approximate-solution errors for
parameterized dynamical systems correspond to \textit{a posteriori error
bounds} and \textit{error indicators}. We proceed by briefly reviewing these
methods, as they serve as a starting point for our work and will be used to
engineer features for the proposed regression framework. We focus in
particular on the dual-weighted-residual error indicator.

\subsubsection{\textit{A posteriori} error bounds}\label{sec:errorbounds}

We now state \textit{a posteriori} error bounds for both the errors
$\abs{\qoiErrorArg{\tindx}(\parametervec)}$ and
$\stateErrorArg{\tindx}(\parametervec)$, $\tindx=1,\ldots,\ntimesteps$, which
arise from \textit{any} arbitrary sequence of approximate solutions
$\stateApprox^\tindx(\parametervec)$, $\tindx=0,\ldots,\ntimesteps$. The first
result was presented in
Ref.~\cite[Theorem 4.3]{leeCarlberg}.

\begin{proposition}[\textit{A posteriori} error bounds \cite{leeCarlberg}]\label{thm:error_bound}
	For a given parameter instance $\parametervec\in\parameterdomain$, if the velocity $\velo$ is Lipschitz continuous, i.e.,
	there exists a constant $\lipschitz>0$ such that
	$\norm{\velo(\bm{x},t;\parametervec)-\velo(\bm{y},t;\parametervec)}\leq\lipschitz\norm{\bm{x}-\bm{y}}$
	for all $\bm{x},\,\bm{y}\in\RR{\nstate}$ and
	$t\in\{t^\tindx\}_{i=1}^{\ntimesteps}$, and
	the time step is sufficiently small such that
	$ \Delta t < {\abs{\alpha_0}}/{\abs{\beta_0}\lipschitz}$,
	then 
\begin{equation} \label{eq:finalGenResult}
	\stateErrorArg{n}(\parametervec)\leq\frac{1}{\hconstant}
	\norm{\resnROMParamsArg{\stateROMArg{n}\paramDep}} +
\sum_{j=1}^{\kArg{n}}\gammaconstantj\stateErrorArg{n-j}(\parametervec)
\end{equation} 
for all $n=1,\ldots,\ntimesteps$. 
	Here, $\hconstant\defeq\abs{\alpha_0} - \abs{\beta_0}\lipschitz\Delta t$
and $\gammaconstantj\defeq\left(\abs{\alpha_j} + 
	\abs{\beta_j}\lipschitz
	\Delta
	t\right)/\hconstant$.
	If additionally 
the QoI functional $\qoiFunc$ is Lipschitz continuous, i.e.,
	there exists a constant $\lipschitzQoI>0$ such that
	$\abs{\qoiFuncArgs{\bm{x}}{t}{\parametervec}-\qoiFuncArgs{\bm{y}}{t}{\parametervec}}\leq\lipschitzQoI\norm{\bm{x}-\bm{y}}$
	for all $\bm{x},\,\bm{y}\in\RR{\nstate}$ and
	$t\in\{t^\tindx\}_{i=1}^{\ntimesteps}$, then
\begin{equation} \label{eq:finalGenResultQoI}
	\abs{\qoiErrorArg{\tindx}(\parametervec)}\leq\frac{\lipschitzQoI}{\hconstant}
	\norm{\resnROMParamsArg{\stateROMArg{n}\paramDep}} +
\lipschitzQoI\sum_{j=1}^{\kArg{n}}\gammaconstantj
	\stateErrorArg{n-j}(\parametervec)
\end{equation} 
for all $n=1,\ldots,\ntimesteps$.
\end{proposition}
\begin{proof}
	Bound \eqref{eq:finalGenResult} was derived in Ref.~\cite[Theorem
	4.3]{leeCarlberg}. 
	Bound
	\eqref{eq:finalGenResultQoI}
follows directly from Lipschitz continuity of the QoI functional and the bound 
\ref{eq:finalGenResult}.
\end{proof}

We make three critical observations from inequalities~\eqref{eq:finalGenResult}
and \eqref{eq:finalGenResultQoI}:
\begin{enumerate}
	\item The error bounds are \textit{residual based}, i.e., the magnitudes of
				the error bounds are driven by the FOM residual operator $\resnFOM$ evaluated
				at the sequence of approximate solutions $\stateApprox^\tindx(\parametervec)$,
				$\tindx=0,\ldots,\ntimesteps$.  Note that the approximate solution incurs no
				error if the approximate solutions exactly satisfy the FOM O$\Delta$E
				\eqref{eq:fom_discretized} such that the residual is zero at all time
				instances. This also illustrates why the residual norm is often viewed as a
				useful \textit{error indicator} for guiding greedy methods for snapshot
				collection
				\cite{bui2008parametric,bui2008model,hinze2012residual,
				wu2015adaptive,yano2018lp} or trust-region optimization algorithms
				\cite{zahr2015progressive,zahrThesis,zahr2018efficient}.
	\item The error bounds exhibit dependence on \textit{non-local} quantities,
				i.e., the error at the $n$th time instance depends on the error at the
				previous $\kArg{n}$ time instances through recursion via the rightmost term in
				inequalities \eqref{eq:finalGenResult} and \eqref{eq:finalGenResultQoI}.  This
				comprises a fundamental difference from error bounds for static problems.
	\item The error bounds often \textit{lack sharpness}, i.e., they are often
		significant larger than the magnitudes of the actual errors. This
		overestimation is exacerbated for ill-conditioned problems (in which case
		the Lipschitz constants $\lipschitz$ and $\lipschitzQoI$ are large) and
		long time horizons (as the error bounds grow exponentially in time). As
		such, their practical utility for error estimation itself is often limited.
	\item The error bounds are \textit{difficult to compute}, i.e., they require
		computing Lipschitz constants $\lipschitz$ and
		$\lipschitzQoI$, which is challenging for general nonlinear
		dynamical systems. While bounding these constants is possible in some
		contexts \cite{huynh2010natural}, this often incurs substantial
		additional computational costs.
\end{enumerate}

\subsubsection{Error indicator: dual-weighted residuals}\label{sec:DWR}
Besides the residual norm, the dual-weighted residual is perhaps the most
popular error indicator. The dual-weighted residual is derived from a low-order approximation of
the QoI error and, as such, is often strongly correlated with the error
\cite{carlberg_ROMES}. 
Because deriving dual-weighted residuals for dynamical systems requires
considering the problem from the monolithic ``space--time'' perspective, we
now introduce the required space--time formulation of the problem.


We begin by defining a ``vectorization" function that assembles the full
space--time solution from a tuple of time-local solutions as
\begin{align*}
 \unroll &\vcentcolon (\vecdummy^1, \hdots, \vecdummy^{N_t})  \mapsto \begin{bmatrix} [ \vecdummy^1 ]^T & \hdots & [\vecdummy^{N_t}]^T \end{bmatrix}^T \\
&\vcentcolon \mathbb{R}^p \otimes \mathbb{R}^{N_t} \rightarrow \mathbb{R}^{p N_t},
\end{align*}
for arbitrary $p$. Then, we define the space--time FOM state and space--time approximate solution 
as
$ \ststate(\parametervec) \defeq \unroll( \state^1(\parametervec), \hdots,
\state^{N_t} (\parametervec) )$ and  
$ \ststateApprox (\parametervec)\defeq \unroll( \stateApprox^1(\parametervec),
\hdots, \stateApprox^{N_t} (\parametervec) ),$ respectively, 
where $\ststate, \ststateApprox:\parameterdomain\rightarrow \RR{N
N_t}$. We also define the space--time residual as
\begin{align*}
	\stresidual &: \big(\ststateDummy;\parametervec \big) \mapsto \unroll(
  \residual^1(\stateDummy^1; \parametervec)  , 
  \ldots  ,  
  \residual^{N_t}(\stateDummy^{N_t};\stateDummy^{N_t-1} , \ldots ,{\stateDummy^{N_t-k}}^{N_t}, \parametervec) 
), \\
&: \RR{N N_t} \times \parameterdomain \rightarrow \mathbb{R}^{N N_t},
\end{align*}
where
$\ststateDummy\equiv\unroll(\stateDummy^1,\ldots,\stateDummy^{\ntimesteps})$.
Lastly, we define
the space--time QoI functional as
\begin{align*}
\stqoiFunc &\vcentcolon \big( \ststateDummy;\parametervec) \mapsto \unroll \big( 
\qoiFunc(\stateDummy^1;t^1,\parametervec) ,
\hdots ,
\qoiFunc(\stateDummy^{N_t};t^{N_t},\parametervec) 
\big), \\
& \vcentcolon \mathbb{R}^{N N_t} \times \parameterdomain \rightarrow
	\mathbb{R}^{N_t},
\end{align*}
such that 
the space--time QoIs computed from the FOM and approximate solutions are given by
\begin{align*}
\stqoi: \parametervec\mapsto
\stqoiFunc( \ststate\paramDepoptional;\parametervec),
	\quad
\stqoiApprox: \parametervec\mapsto
\stqoiFunc( \ststateApprox\paramDepoptional;\parametervec),
\end{align*}
respectively, where  $\stqoi ,\stqoiApprox: \parameterdomain \rightarrow
\RR{N_t}$.

Assuming the residual is twice continuously differentiable
in its first
argument and noting
that $\stresidual(\ststate\paramDepoptional;\parametervec)=\zero$, we can write
\begin{equation}\label{eq:resFirst}
- \stresidual(\ststateApprox\paramDepoptional;\parametervec) = \bigg[ \frac{\partial
	\stresidual}{\partial
	\ststateDummy}(\ststateApprox(\parametervec);\parametervec) \bigg]
	\bigg(\ststate\paramDepoptional - \ststateApprox\paramDepoptional \bigg) + \mathcal{O}\bigg(
	\norm{\ststate\paramDepoptional - \ststateApprox\paramDepoptional}^2 \bigg).
\end{equation}
Assuming the space--time residual Jacobian $\bigg[ \frac{\partial
\stresidual}{\partial \ststateDummy}(\ststateApprox\paramDepoptional;\parametervec) \bigg]$ is
invertible, Eq.~\eqref{eq:resFirst} can be solved for a first-order
approximation of the solution error
\begin{equation}\label{eq:resFirstSolve}
\ststate\paramDepoptional - \ststateApprox\paramDepoptional = 
- \bigg[ \frac{\partial
	\stresidual}{\partial \ststateDummy}(\ststateApprox\paramDepoptional;\parametervec)
	\bigg]^{-1}\stresidual(\ststateApprox\paramDepoptional;\parametervec) +
	\mathcal{O}\bigg( \norm{\ststate\paramDepoptional - \ststateApprox\paramDepoptional}^2 \bigg).
\end{equation}
If we additionally assume the space--time QoI functional $\stqoiFunc$ is twice continuously differentiable
in its first argument, the error in the space--time QoI can be represented to first order as, 
\begin{equation}\label{eq:qoi_error_1}
 \stqoi(\parametervec) - \stqoiApprox(\parametervec) = \bigg[ \frac{\partial
	\stqoiFunc}{\partial \ststateDummy}(\ststateApprox\paramDepoptional;\parametervec)
	\bigg] \bigg(\ststate\paramDepoptional - \ststateApprox\paramDepoptional \bigg)+ \mathcal{O}\bigg(\norm{
		\ststate\paramDepoptional - \ststateApprox\paramDepoptional }^2 \bigg).
\end{equation}
Substituting Eq.~\eqref{eq:resFirstSolve} into Eq.~\eqref{eq:qoi_error_1} yields,
\begin{equation} \label{eq:dwrFirstOrder}
 \stqoi(\parametervec) - \stqoiApprox(\parametervec)  =
	\dualweightedresidual(\parametervec) + \mathcal{O}\bigg(\norm{
		\ststate\paramDepoptional - \ststateApprox\paramDepoptional }^2 \bigg),
\end{equation} 
where the dual-weighted residual comprises a first-order approximation of the
QoI error and is defined as
\begin{align}\label{eq:dual_final}
 \dualweightedresidual(\parametervec)  \defeq  -&\underbrace{ \bigg[
	 \frac{\partial \stqoiFunc}{\partial
	 \ststateDummy}(\ststateApprox\paramDepoptional;\parametervec) \bigg ]
 }_{N_t \times N N_t} \underbrace{ \bigg[ \frac{\partial \stresidual}{\partial
	\ststateDummy}(\ststateApprox\paramDepoptional;\parametervec)\bigg]^{-1}
	}_{N N_t \times N N_t}  \underbrace{
		\stresidual(\ststateApprox\paramDepoptional;\parametervec) }_{N N_t \times 1}, \\
  =&\begin{bmatrix*}[c]
\matshapea & \\[-5pt]
& \matshapea & \\[-5pt]
& & \hdots \\[-5pt]
 & & &\matshapea 
\end{bmatrix*}
\begin{bmatrix*}[c]
\matshapebf &  \\ 
\matshapebf & \matshapebf \\
\vdots &  \vdots  & \ddots \\
\matshapebf & \cdots   & \cdots & \matshapebf \\
\end{bmatrix*}
\begin{bmatrix*}[c]
\matshapec \\ \matshapec  \\ \vdots \\  \matshapec
\end{bmatrix*}.
\end{align}
Computing the dual-weighted residual \eqref{eq:dual_final} involves applying
the inverse of the $\nstate\ntimesteps\times \nstate\ntimesteps$ space--time
residual Jacobian and can be computed in one of two ways: (1) the direct
(forward) method, or (2) the adjoint (backward) method. The direct method
solves the (linear) \textit{sensitivity equations}
\begin{equation}\label{eq:stforward}
	\frac{\partial \stresidual}{\partial
\ststateDummy}(\ststateApprox\paramDepoptional;\parametervec) \stforward=
	\stresidual(\ststateApprox\paramDepoptional;\parametervec),
\end{equation}
where $\stforward\equiv\stforward(\parametervec)\equiv\unroll( \forward^1(\parametervec), \hdots,
\forward^{N_t} (\parametervec) )$ with 
$\forward^n:\parameterdomain\rightarrow\RR{\nstate}$, $n=1,\ldots,\ntimesteps$
is implicitly defined as the solution to Eq.~\eqref{eq:stforward} and
denotes the space--time \textit{sensitivity vector}. The direct method 
subsequently computes the dual-weighted residual as 
\begin{equation}\label{eq:dwrForward}
 \dualweightedresidual(\parametervec)  =   -\big[ \frac{\partial
	\stqoiFunc}{\partial \ststateDummy}(\ststateApprox\paramDepoptional;\parametervec) \big ]
	\stforward(\parametervec).
\end{equation}
Due to the sparsity pattern of the space--time residual Jacobian 
$\frac{\partial \stresidual}{\partial
\ststateDummy}(\ststateApprox\paramDepoptional;\parametervec)
$, this approach is equivalent to performing a linear forward-in-time solve,
which can be executed simultaneously with the forward-in-time solve used to
generate the approximate solutions, i.e., the sensitivities
$\forward^\tindx(\parametervec)$ 
can be computed contemporaneously with $\stateApprox^\tindx(\parametervec)$. In contrast,
the adjoint method solves the  multiple-right-hand-side (linear) \textit{adjoint
equations}
\begin{equation}\label{eq:stadjoint}
	\bigg[ \frac{\partial \stresidual}{\partial
\ststateDummy}(\ststateApprox\paramDepoptional;\parametervec)\bigg]^T \stdual = \bigg[
	\frac{\partial \stqoiFunc}{\partial
	\ststateDummy}(\ststateApprox\paramDepoptional;\parametervec) \bigg]^T,
\end{equation}
where 
$$\stdual\equiv\stdual(\parametervec)\equiv \begin{bmatrix} \unroll( \dual_1^1(\parametervec), \hdots,
\dual_1^{N_t} (\parametervec) ) & \cdots & \unroll( \dual_{N_t}^1(\parametervec), \hdots,
\dual_{N_t}^{N_t} (\parametervec) ) \end{bmatrix}\in\RR{\nstate\ntimesteps\times
\ntimesteps},$$
with 
$\dual_i^n:\parameterdomain\rightarrow\RR{\nstate}$, $i,n=1,\ldots,\ntimesteps$
is implicitly defined as the solution to Eq.~\eqref{eq:stadjoint} and
denotes $N_t$ space--time \textit{dual vectors}. The adjoint 
method subsequently computes the dual-weighted residual as
\begin{equation} \label{eq:dwrFormula}
	\dualweightedresidual(\parametervec)  =  -\stdual(\parametervec)^T
	\stresidual(\ststateApprox\paramDepoptional;\parametervec). 
\end{equation} 
Due to the sparsity pattern of the space--time residual Jacobian transpose, 
this approach is tantamount to performing a linear backward-in-time solve,
and must be computed \textit{after} the approximate-solution sequence has been
fully computed, as this sequence defines the argument at which the space--time
residual Jacobian is evaluated.

%
%
%

	Analogously to the observations made about error bounds in the Section
	\ref{sec:errorbounds}, 
	we now make several observations about the dual-weighted residual: 
\begin{enumerate}
\item The dual-weighted residual is
		\textit{residual-based}, i.e., it corresponds to a linear combination of
		elements of the space--time residual, where the weights in the linear
		combination are provided by the dual according to
		Eq.~\eqref{eq:dwrFormula}.
	\item The dual-weighted residual exhibits dependence on \textit{non-local}
		quantities, i.e., computing the dual-weighted residual using either the
		direct method \eqref{eq:dwrForward} or the adjoint method
		\eqref{eq:dwrFormula} requires solving a linear system either
		forward-in-time or backward-in-time over the entire time domain. Further, the
		sparsity pattern of the operators in Eq.~\eqref{eq:dual_final} illustrates that the QoI
		error at the $n$th time instance depends on quantities computed at all
		previous time instances.
	\item The dual-weighted residual may be an \textit{inaccurate} approximation
		of the error, i.e., it is a first-order approximation valid in the limit 
	$\norm{\ststate\paramDepoptional - \ststateApprox\paramDepoptional} = 
		\sqrt{\sum_{n=1}^{\ntimesteps}(\stateErrorArg{\tindx}\paramDepoptional)^2}
		\rightarrow 0$, 
		and thus 
	can be inaccurate for cases where the normed state error is not
		negligible. This is precisely the case we are interested in, as we aim to
		model this error.
	\item Dual-weighted residuals can be \textit{difficult to compute}, i.e., they
		require solving either the forward sensitivity \eqref{eq:stforward} or adjoint
		equations \eqref{eq:stadjoint}, each of which is characterized by a block triangular
		system matrix of dimension $\nstate\ntimesteps\times \nstate\ntimesteps$,
		with $\nstate\ntimesteps$
		the space--time dimension of the original full-order-model problem. The
		use of approximate solutions aims to avoid any computations that scale
		with the dimension $\nstate$, and thus computing the dual-weighted
		residual in this manner is often impractical. To mitigate this
		computational burden, analysts often adopt the adjoint method, solve
		the adjoint equations \eqref{eq:stadjoint} using a coarse model (e.g., a
		coarse mesh), and prolongate the computed (coarse) dual vector to a
		representation compatible with the (fine) full-order model in order to
		compute the dual-weighted residual \cite{VENDITTI2000204}. In addition to
		these computational-cost issues, adopting the adjoint method introduces
		other challenges, as it requires access to the residual-Jacobian transpose
		to solve Eq.~\eqref{eq:stadjoint}, which may not be available.
\end{enumerate} 

\subsubsection{Discussion}

This work aims to overcome the most prominent shortcomings of \textit{a
posteriori} error bounds and dual-weighted residuals: (1) their lack of
sharpness or inaccuracy, and (2) the difficulty in their computation.  We aim
to achieve this by applying modern machine-learning methods within a
regression setting, wherein the \textit{regression model} is provided by a
recursive time-series-prediction model that can capture dependencies on
non-local quantities, the \textit{features} correspond to residual-based
time-local error indicators that appear in the definitions of the \textit{a
posteriori} error bounds and dual-weighted residuals, and the
\textit{response} corresponds to the time-local normed state or QoI error.
Thus, we employ the classical methods for error quantification to guide the
design of the both the regression model and features.  We also aim to
construct a statistical model of the response error 
by considering several noise
models.  The next section describes the proposed methodology, which adopts
this strategy.



\section{\methodName}\label{sec:ml_framework}

We now present the proposed  \methodName\ (\methodAcronym) method, which
constructs a mapping from residual-based features to a random variable for the
error using regression techniques that can capture dependencies on non-local
quantities.  The method comprises an extension of the machine-learning error
modeling (MLEM) framework presented in Ref.~\cite{freno_carlberg_ml}---which
was applicable to parameterized static systems---to parameterized
dynamical systems. Because the method is applicable to modeling both the QoI
error $\qoiErrorArg{\tindx}(\parametervec)$ and normed-state error
$\stateErrorArg{\tindx}(\parametervec)$, we denote the error of interest
generically as $\errorGen^\tindx(\parametervec)$ in this section.

\subsection{Machine-learning framework}\label{sec:mlFramework}
The proposed method aims to predict errors on a \textit{subset} of the
time grid
$t^n = n\Delta t$, $n=0,\ldots,
\ntimesteps$ used to define the FOM O$\Delta$E \eqref{eq:fom_discretized}. We make
this choice because the time step $\Delta t$ is typically chosen to control
time-discretization errors, while we may be
interested in modeling the error at only a (relatively small) subset of these
grid points; furthermore, constructing a time-series model of these errors 
on a coarser time grid reduces the computational cost
of training the regression model.
We define this coarse time grid as $t^n = n\Delta t$, $n\in\timeIndexSet$, where
$\timeIndexSet\subseteq\{1,\ldots,\ntimesteps\}$ denotes the coarse time-index
set comprising $\ntimestepsCoarse\defeq\card{\timeIndexSet}(\leq \ntimesteps)$
time instances and equipped with a mapping
$\timeIndexMappingNo:\{0,\ldots,\ntimestepsCoarse\}\rightarrow\{0,\ldots,\ntimesteps\}$
from a coarse time index to the corresponding fine time index, which satisfies $\tau(0) = 0$.

\begin{figure}
\begin{centering}
\includegraphics[trim={0.0cm 9cm 0cm 6cm},clip,width=1.0\textwidth]{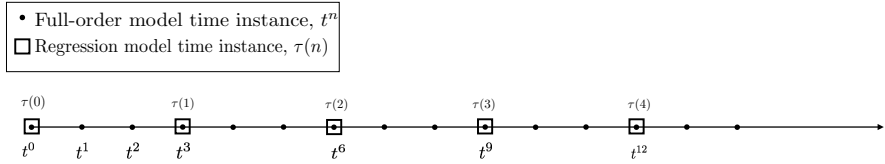}
\caption{Example of time discretizations for the full-order model and machine
	learning model. In this case, the regression model is solved on a temporal
	grid that is three times as coarse as the approximate and full-order models such 
	that $\timeIndexSet = \{3,6,9,12,\ldots \}$.
	}
\label{fig:time_grid}
\end{centering}
\end{figure}

We assume that we can compute \textit{features}
$\featurenmu\in\mathbb{R}^{\nfeatures}$ that are informative of the
\textit{response}
$\errorGennmu\in\RR{}$---which is the time-local error of
interest---at time instances $n\in\timeIndexSet$. Section
\ref{sec:featureEngineering} provides candidate features, which are inspired
by the classical error-quantification methods described in Section
\ref{sec:classicalError}. Given the ability to compute these features, we
propose to construct error models 
$\errorModelGennmu(\approx\errorGennmu)$ of
the form
\begin{equation}\label{eq:error_split}
\errorModelGennmu  =  \errorApproxMeanGennmu + \errorApproxRandGennmu,\quad
	\tindx \in \timeIndexSet,
\end{equation}
where $\errorApproxMeanGenn:\parameterdomain\rightarrow \RR{}$,
$n\in\{0\}\cup\timeIndexSet$ denotes the sequence of
\textit{deterministic regression-function models}
and $\errorApproxRandGenn:\parameterdomain \rightarrow \randomvars{}$,
$n\in\{0\}\cup\timeIndexSet$ denotes a sequence of mean-zero random variables
that serve as the \textit{stochastic noise models}
such that 
 $\errorModelGenn:\parameterdomain\rightarrow\randomvars{}$,
 $n\in\timeIndexSet$ with
$\expectation{\errorModelGennmu} = \errorApproxMeanGennmu$ 
 ; we denote the set of real-valued random variables
by $\randomvars{}$. 
%
%

As previously described, we propose to construct the regression
model using a recursive time-series-prediction model that can capture
dependencies on non-local quantities.
In particular, we impose the following structure on the 
regression-function model:
\begin{equation}\label{eq:ml_deterministic}
\errorApproxMeanGennmu =
	\approxregressionfunction(\featurenmu,\latentstatefnmu),\quad
	\tindx \in \timeIndexSet.
\end{equation}
Here, $\approxregressionfunction: \RR{\nfeatures}\times\RR{\nlatentf}\rightarrow
\RR{}$, while
$\latentstatefn:\parameterdomain\rightarrow \RR{\nlatentf}$,
$n\in\{0\}\cup\timeIndexSet$ denotes the sequence of \textit{latent
variables}. Critically, the inclusion of
these latent variables enables the proposed method to capture dependencies on non-local
quantities, as we enforce the latent variables to be governed
by the recursion
\begin{align}\label{eq:latentStateRecursion}
	\latentstatefnmuArg{\timeIndexMapping{n}} &=
	\recursionf(\featurefnmuArg{\timeIndexMapping{n}},
	\latentstatefnmuArg{\timeIndexMapping{n-1}},
	\errorApproxMeanGennmunArg{\timeIndexMapping{n-1}}
	),\quad
	n=1,\ldots,\ntimestepsCoarse,
\end{align}
where $\recursionf: 
\RR{\nfeatures}\times
\RR{\nlatentf}\times\RR{}
\rightarrow
\mathbb{R}^{\nlatentf}$; 
$\errorApproxMeanGennmunArg{0}=\errorGennmuArg{0}$ is given; and 
$\latentstatefnmuArg{0}=\zero$.
Similarly, we impose the following structure on the noise model:
\begin{equation}\label{eq:ml_stochastic}
	\errorApproxRandGennmunArg{\timeIndexMapping{n}}
=
	\approxnoisefunction(\errorApproxRandGennmunArg{\timeIndexMapping{n-1}},\noisenArg{\timeIndexMapping{n}}),\quad
	n=1,\ldots,\ntimestepsCoarse.
\end{equation}
Here, $\approxnoisefunction:
\randomvars{}\times\randomvars{}\rightarrow
\randomvars{}$, 
$\errorApproxRandGennmunArg{0}=0$, 
and $\noisenArg{n} \in \randomvars{}$,
$n\in\timeIndexSet$ denotes a sequence of independent and identically distributed mean-zero
random variables. 
\begin{remark}[Classical non-recursive regression]
We note that classical non-recursive regression methods can be described by
omitting latent variables from the regression function approximation such that 
$\approxregressionfunction(\featurenmu,\latentstatefnmu)\equiv
\approxregressionfunction(\featurenmu)
$,
omitting latent-variable dynamics \eqref{eq:latentStateRecursion}, and
	employing a simple
non-recursive noise model 
such that
$\errorApproxRandGennmunArg{n} = \noisenArg{n}
	$, $n\in\timeIndexSet$. 
\end{remark}
\begin{remark}[Noise models]
We note that more sophisticated functional forms of $\approxnoisefunction$
could be considered, e.g., by allowing parameters that define the
	probability distribution of 
$\noisenArg{n}$ for $n\in\timeIndexSet$ to
depend on features, or by including latent variables as inputs to the function
$\approxnoisefunction$.  
\end{remark}

The primary distinction between the proposed \methodAcronym\ method and the
methods
proposed in Refs.~\cite{freno_carlberg_ml,trehan_ml_error} is the
consideration of recursive regression models that employ
latent variables. The MLEM method~\cite{freno_carlberg_ml} considered only
static problems and thus did not require the use of such variables.  The
EMML framework~\cite{trehan_ml_error} accounted for dependencies on non-local
quantities by explicitly considering features corresponding to quantities computed
at previous time instances. This introduces difficulties, as it requires
knowledge about the memory of the system (i.e., how many previous time
instances to consider in constructing the features) and can lead to
a prohibitively large set of features. In contrast, the use of latent variables
and the specific functional forms
\eqref{eq:ml_deterministic}--\eqref{eq:ml_stochastic} allow the proposed \methodAcronym\ method to capture
non-local effects naturally.

As in Ref.~\cite{freno_carlberg_ml}, we aim to construct regression models
that exhibit the following properties:
\begin{Objective}
	\item\label{obj:cost} \textbf{Low cost}. The error model 
	$\errorModelGennmu$	
		should be computationally
		inexpensive to deploy, i.e., its evaluation should not dominate the
		computational cost of computing the approximate solutions 
		$\stateApprox^\tindx(\parametervec)$, $\tindx=0,\ldots,\ntimesteps$.
		This requirement implies that
		the features 
		$\featurenmu$ should be cheap to compute, the latent-variable dimension $\nlatentf$ 
		should be relatively small, and it should be inexpensive to evaluate the
		functions $\approxregressionfunction$, $\recursionf$,	 and 
		$\approxnoisefunction$.
	\item\label{obj:lowvar} \textbf{Low noise variance}. Because the variance of
		the noise model $\errorApproxRandGennmu$ 
		reflects the magnitude of
		epistemic uncertainty introduced by the use of an approximate solution,
		reducing this variance in turn reduces the uncertainty
		introduced into the analysis. A model with
		low noise variance has
		not been underfit to training data.
	\item\label{obj:generalize} \textbf{Generalizability}. For the regression
		model to be trustworthy, the distributions of the true error
		$\errorGennmu$ and the error model $\errorModelGennmu$ should closely
		match on an independent test set, which was not used to train the error
		model. A model exhibiting generalizability has not been overfit to
		training data.
\end{Objective}

We adopt the machine-learning framework proposed in
Ref.~\cite{freno_carlberg_ml} to satisfy these objectives.  This framework
comprises the following steps: 

\begin{Step}
\item \label{step:featureEng}\textbf{Feature engineering.} Select features that are inexpensive to
	compute (\ref{obj:cost}), informative of the error such that a
		low-noise-variance model can be constructed (\ref{obj:lowvar}),
		and low dimensional (i.e., $\nfeatures$ small) such that less training
		data are required for the resulting model to generalize
		(\ref{obj:generalize}). To satisfy these requirements, we engineer
		residual-based features informed by classical error-quantification methods
		as described in Section \ref{sec:classicalError}. We ensure these features
		are low dimensional and inexpensive to compute through the use of
		sampling-based approximations (e.g., gappy principal component analysis).
		Section \ref{sec:featureEngineering} describes this step.
\item \label{step:dataGen}\textbf{Data generation.} After selecting candidate features, 
	generate both training and test data in the form of a set of
		response--feature pairs on
		the coarse time-index set $\timeIndexSet$ for parameter instances sampled
		according to the imposed probability distribution on the parameter domain
		$\parameterdomain$.
		Constructing these data requires computing \textit{both} the FOM solutions
		and approximate solutions for the selected parameter instances. Sections
		\ref{sec:training}--\ref{ssec:pca}
		describe this step.
\item\label{step:regressionModel} \textbf{Train the deterministic regression-function model $\errorApproxMeanGennmu$.} 
Define candidate functional forms for the functions
		$\approxregressionfunction$ and $\recursionf$, compute model parameters by
		(approximately) minimizing a loss function defined on
		the training data, and perform hyperparameter selection
		using (cross-)validation. Sections \ref{sec:regressionfunction}--\ref{sec:regressionTraining}
		describes this step.
\item\label{step:noiseModel}\textbf{Train the stochastic noise model $\errorApproxRandGennmu$.}
	Finally, define candidate functional forms for the function
		$\approxnoisefunction$, and compute model parameters via maximum
		likelihood estimation on a training set independent from that used to
		train the deterministic regression function. In this work, we employ a
		subset of the test set used to evaluate the deterministic regression
		function for this purpose.
		Sections~\ref{sec:noise_modeling}--\ref{sec:noise}
		describe this step.

\end{Step}

\subsubsection{Feature engineering}\label{sec:featureEngineering}
We now describe \ref{step:featureEng} of the proposed framework: feature
engineering. Namely, we describe the candidate features $\featurevec^\tindx$ that we consider,
which are inspired by those proposed in Ref.~\cite{freno_carlberg_ml}.
Several considerations drive feature engineering, which balance
\ref{obj:cost}--\ref{obj:generalize}:
\begin{Consideration}
\item \textbf{Feature cost.} To satisfy \ref{obj:cost}, we aim to
	develop features that are inexpensive to compute. Satisfying this
		objective can compromise quality, as high-quality features are often
		costly to compute.
	\item \textbf{Feature quality.} To satisfy \ref{obj:lowvar}, we
		aim to develop features that are as informative of the error as possible,
		because employing high-quality explanatory features can lead to a
		low-noise-variance regression model.
	\item\label{cons:numberFeatures} \textbf{Number of features.} Employing a large number of features often
	leads to a high-capacity model that can yield low noise variance with
		sufficient training data (i.e., \ref{obj:lowvar} is bolstered). However, computing
		a large number of features can be computationally expensive (i.e., 
		\ref{obj:cost} suffers), and often requires a large amount of training
		data to generalize well and avoid overfitting (i.e., 
		\ref{obj:generalize} suffers). Regularization strategies can be employed
		when training a regression model with a large number of features to mitigate overfitting.
\end{Consideration}

The remainder of this section details a variety of candidate features that differ
in terms of the above considerations.  


\begin{Feature}
		\item\label{feat:params} \textbf{Parameters}:
			The system's parametric inputs $\parametervec$
			correspond to free-to-compute but low-quality features, where the number
			of features is equal to the number of parameters
			$\nparams$. Simply setting  $\featurenmu = \parametervec$ comprises
			the same feature choice as that employed by the classical
model-discrepancy approach \cite{ken_ohag}, which uses Gaussian-process
regression to construct the regression model.
These features are time independent.
\item \label{feat:t}\textbf{Time}: The time variable $t^n$ is a single,
	free-to-compute feature, which is likely to be low quality in most
	scenarios. This low quality arises from the fact that---in most
	applications---the error is not driven by the time variable itself. Instead,
	the classical error-quantification methods discussed in Section
	\ref{sec:classicalError} suggest that the error is instead determined by the
	residual, Lipschitz constants, and dual vectors. Intuitively, employing time
	as a feature can cause the resulting regression model to be ``tied'' to the
	particular time coordinates associated with the behavior observed during
	training. This implies that the resulting regression method will lack
	invariance with respect to the time grid, i.e., the model will not
	generalize well for problems characterized by phase shifts or for
	prediction beyond the training time interval.  The EMML framework 
	considered time as a feature \cite[Table I]{trehan_ml_error}. 
\item \label{feat:rnorm}\textbf{Residual norm}:
	As discussed in Section \ref{sec:classicalError}, classical
		error-quantification methods are residual based. In particular, 
		the first term on the right-hand side of \textit{a posteriori} error bound~\eqref{eq:finalGenResult}
		demonstrates that the residual norm
	$\norm{ \resApproxArg{\tindx}
	}$, where
		$\resApproxArgNo{\tindx}:\parametervec\mapsto\resnROMParamsArg{\stateROMArg{n}(\parametervec)}$,
		is
		informative of the state error
	$\stateErrorArg{n}(\parametervec)$; bound~\eqref{eq:finalGenResultQoI} shows that
		the residual norm 
		$\norm{ \resApproxArg{\tindx}
	}$
		is also informative of the magnitude of the QoI
		error $\abs{\qoiErrorArg{\tindx}(\parametervec)}$.  This feature is
		computationally expensive to compute, as it requires
		computing all $\nfom$ elements of the vector $\resApproxArg{\tindx}$
		before computing its norm.
		However, this (single) feature is likely of higher quality than the both \ref{feat:params}
		and \ref{feat:t}, as it appears explicitly
		in classical \textit{a posteriori} error bounds
		\eqref{eq:finalGenResult}--\eqref{eq:finalGenResultQoI}. Yet, 
		this feature
		will likely be of lower quality for predicting the (signed) QoI error
		$\qoiErrorArg{\tindx}(\parametervec)$ because
		it is not informative of the error's sign.
		We note
		that---in contrast to static systems---the
		residual norm 
	$\norm{ \resApproxArg{\tindx}
	}$	
		(along with stability constants) is not sufficient to 
bound the errors $\stateErrorArg{n}(\parametervec)$
	and 
	$\abs{\qoiErrorArg{\tindx}(\parametervec)}$.
		Instead, inequalities~\eqref{eq:finalGenResult}--\eqref{eq:finalGenResultQoI}
		illustrate that these bounds depend additionally on non-local quantities,
		namely $\stateErrorArg{n-j}(\parametervec)$, $j=1,\ldots,\kArg{n}$. Thus, we expect this feature
		choice---as well as all subsequent residual-based features---to perform best
		when used with recursive regression methods that can capture dependencies on
		non-local quantities.
	We note that the ROMES method
		\cite{carlberg_ROMES} employed features
$\featurevec =  \norm{ \resApproxArg{} }$ with Gaussian-process
		regression to construct a regression model for the normed state error.

\item \label{feat:r}\textbf{Residual}: To increase the number of features
	relative to \ref{feat:rnorm} yet maintain high feature quality, one may employ 
	all $\nfom$ elements of the residual vector 
	$\resApproxArgNo{\tindx}(\parametervec)$ as features.
	This has the effect of changing the number of features
	from one extreme to another, as we assume the full-order-model dimension
	$\nfom$ is large.
	Employing $\nfom$ features 
	can enable very high-capacity regression models, which can lead to
	low noise variance given sufficient training data; however, the amount of
	training data needed to ensure generalizability may be excessive. Further, 
	these features are costly to compute.
	These features are of high quality, as the residual not only appears in the
	\textit{a posteriori} error bounds
	\eqref{eq:finalGenResult}--\eqref{eq:finalGenResultQoI}, but it also appears
	in the dual-weighted residual discussed in Section \ref{sec:DWR}. Namely,
	inserting Eq.~\eqref{eq:dwrFormula} into Eq.~\eqref{eq:dwrFirstOrder} shows that the QoI
	error can be approximated to first order by a linear combination of
	space--time residual elements. 
\item \textbf{Residual principal components}\label{feat:params_rpca}
	$\resPCAApproxArg{\tindx}$.
	As proposed in Ref.~\cite{freno_carlberg_ml}, to retain the high quality of
	\ref{feat:r}, but reduce the number of features 
	such that less
	training data is needed for generalization,
	we can employ \textit{principal components} of the residual as features. In particular, 
one can use  
$$\resPCAApproxArg{\tindx} \defeq
	\basismat_{\residual}^T  (\resApproxArg{\tindx}- \bar{\resApprox})$$ as
	features, where the columns of 
$\basismat_{\residual} \in \RRstar{\nfom\times\npca}$ 
comprise the first $\npca (\ll \nfom)$ principal components of the training set
$\residualTrainSet\subset\RR{\nfom}$ (which will be described in Section \ref{ssec:pca}), and $\overline{\resApprox}
\defeq \frac{1}{\card{\residualTrainSet} }
\sum_{\resApprox\in\residualTrainSet}\resApprox$ denotes the training-set mean.
These features are costly to compute, as they require computing all $\nfom$
elements of the residual $\resApproxArg{\tindx}$ before applying projection
with the
principal-component matrix $\basismat_{\residual}$. However, these features
retain the high quality of the residual vector if the truncation level $\npca$
is set such that the principal components explain most of the variance of the
data in the
training set $\residualTrainSet$. Critically,
this approach associates with using $\npca$ features and thus constitutes a
compromise between \ref{feat:rnorm} and \ref{feat:r}: it will lead to a higher-capacity model
than can be obtained using \ref{feat:rnorm}, but will 
likely generalize with significantly less training data than would be the case
with \ref{feat:r}.

\item \textbf{Residual gappy principal components}\label{feat:params_rgpca}:
 This approach aims to address the primary shortcoming of
 \ref{feat:params_rpca}: the high computational cost incurred by the
 need to compute all $\nfom$ elements of the residual $\resApproxArg{\tindx}$.
In particular, this feature choice aims to approximate the residual principal components
using the gappy POD method~\cite{everson_sirovich_gappy}, which
replaces orthogonal projection with
a sampling-based linear least-squares approximation to realize computational-cost savings.
Specifically, introducing a sampling matrix
$\sampMat \in
\{0,1\}^{\nsample\times\nfom } $ comprising $\nsample$ selected rows of the
identity matrix
(with
$\npca\leq\nsample\ll\nfom$), the residual gappy principal components are
$$\resGappyApproxArg{\tindx} \defeq
[\sampMat\basismat_{\residual}]^{+} \sampMat ( \resApproxArg{\tindx}- \overline{\resApprox} ).$$
In this work, we construct the sampling matrix $\sampMat$ according to the
q-sampling method~\cite{qdeim_drmac}; Appendix \ref{appendix:qsampling}
presents the associated algorithm. 
We note that the number of features for this method is
the same as that for \ref{feat:params_rpca}, and the feature quality is
similar to that of \ref{feat:params_rpca} if least-squares reconstruction
yields similar results to orthogonal projection.
However, computing these features requires
computing only $\nsample(\ll\nfom)$ elements of the residual vector 
$\resApproxArg{\tindx}$; thus, this feature choice is substantially less
costly to compute than \ref{feat:params_rpca}.

\item \textbf{Sampled residual}\label{feat:params_pr}:
	\ref{feat:params_rgpca} constitutes an affine transformation
	applied to the sampled residual $\sampMat\resApproxArg{\tindx}$. Thus, it is
	sensible to consider employing only the underlying sampled residual
	$\sampMat\resApproxArg{\tindx}$ as features, given that many regression
	techniques are likely to discover this affine transformation if it yields
	superior prediction performance.
	Further, employing the sampled residual does not
	require computing the residual principal components
	$\basismat_{\residual}$ (although q-sampling employs this matrix to
	construct the
	sampling matrix $\sampMat$). We note that the number of features in this case
	is $\nsample(\geq\npca)$.
\end{Feature}
In the numerical experiments, we consider feature-engineering methods that
employ different combinations of the above features. However, all considered
methods use at most one of the residual-based features, i.e., at most one of
\ref{feat:rnorm}--\ref{feat:params_pr}. Table~\ref{tab:feature_summary}
summarizes the feature-engineering methods considered in the numerical
experiments. 

\begin{table}
\centering
\renewcommand{\arraystretch}{2}
\begin{tabular}{ p{5cm} c c  c }\hline
Description & Features $\featurevec^\tindx$&  Number of features $\nfeatures$
	& 
\begin{tabular}{@{}c@{}}
	Number of residual 
	\vspace{-0.3cm}
	\\ elements required\end{tabular}	
	\\ \hline
Parameters & $\parametervec$ &  $N_{\parametervec}$  & 0   \\
Parameters and time & $[\parametervec;t^{\tindx}]$ & $\nparams + 1$ & $0$ \\
Residual norm & $\norm{ \resApproxArg{\tindx} }$ &  $1$  & $\nfom$  \\
Parameters and \newline residual norm & $[\parametervec; \norm{\resApproxArg{\tindx}}]$ &  $\nparams + 1$  & $\nfom$  \\
Parameters, residual norm, \newline and time & $[\parametervec;\norm{\resApproxArg{\tindx}} ;t^{\tindx}]$ &  $\nparams + 2$  & $\nfom$   \\
Parameters and residual  & $[\parametervec;\resApproxArg{\tindx} ]$ &  $\nparams+ \nfom$  & $\nfom$   \\
Parameters, residual, \newline and time  & $[\parametervec;\resApproxArg{\tindx};t^{\tindx} ]$ &  $\nparams + \nfom + 1$  & $\nfom$   \\
Parameters and residual \newline principal components & $[\parametervec; \resPCAApproxArg{\tindx} ]$ &  $\nparams + \npca$  & $\nfom$   \\
Parameters, residual \newline principal components, and time & $[\parametervec;\resPCAApproxArg{\tindx} ; t^{\tindx}]$ &  $\nparams  + \npca + 1$  & $\nfom$   \\
Parameters and gappy \newline residual principal components & $[\parametervec;\resGappyApproxArg{\tindx}]$ &  $\nparams + \npca$  & $\nsample$  \\
Parameters, gappy residual \newline principal components, and time & $[\parametervec;\resGappyApproxArg{\tindx};t^{\tindx} ]$ &  $\nparams + \npca + 1$  & $\nsample$  \\
Parameters and sampled residual & $[\parametervec;\sampMat \resApproxArg{\tindx}]$ &  $\nparams + \nsample$  & $\nsample$  \\
 Parameters, sampled \newline residual, and time & $[\parametervec;\sampMat\resApproxArg{\tindx} ;t^{\tindx} ]$ &  $\nparams + \nsample + 1$  & $\nsample$  \\
\hline
\end{tabular}
\caption{Summary of considered feature-engineering methods.
	}
\label{tab:feature_summary}
\end{table}

\subsubsection{Data generation}\label{sec:training}
We now describe \ref{step:dataGen} of the proposed framework: data generation.
After selecting the features, the method generates response--feature pairs
that comprise the training, validation, and test sets used to select the deterministic
regression-function model and stochastic noise model, respectively.
Specifically, we first define 
$$\dataFunction(\parameterdomainDum,\timeIndexSetDum) \defeq \{ (
\errorGennmuArg{n},\featurefnmuArg{n} )\, |\, \parametervec \in \parameterdomainDum, \tindx \in \timeIndexSetDum \}$$
as the set of response--feature  pairs generated by computing the full-order
and approximate states 
for $\parametervec\in\parameterdomainDum$ and computing the
features and errors at time instances $t^n$ with $n \in
	\timeIndexSetDum\subseteq\{1,\ldots,\ntimesteps\}.$
We then define the training, validation, and test data as 
	\begin{equation}\label{eq:trainTestSet}
		\trainSet \defeq \dataFunction(\parameterdomainTrain,\timeIndexSetTrain),\qquad
\valSet \defeq
		\dataFunction(\parameterdomainVal,\timeIndexSetVal),\qquad
\testSet \defeq \dataFunction(\parameterdomainTest,\timeIndexSetTest),
	\end{equation}
respectively,
where
$\parameterdomainTrain,\parameterdomainVal,\parameterdomainTest \subset
	\parameterdomain$; and
$\timeIndexSetTrain\defeq\{\timeIndexMapping{1},\ldots,\timeIndexMapping{\card{\timeIndexSetTrain}}\}\subseteq \timeIndexSet$, 
	$\timeIndexSetVal\defeq\{\timeIndexMapping{1},\ldots,\timeIndexMapping{\card{\timeIndexSetVal}}\}\subseteq \timeIndexSet$,
	and $\timeIndexSetTest
\defeq\{\timeIndexMapping{1},\ldots,\timeIndexMapping{\card{\timeIndexSetTest}}\}
\subseteq \timeIndexSet $. For example, 
$\parameterdomainTrain$, $\parameterdomainVal$, and $\parameterdomainTest$ can be generated by
drawing samples from a probability distribution defined on the parameter domain
$\parameterdomain$, and the time indices can be set to
	$\timeIndexSetTrain=\timeIndexSetVal=\timeIndexSetTest=\timeIndexSet$ (which we do in the numerical
experiments).
\subsubsection{Training data for principal component analysis and q-sampling}\label{ssec:pca} 
\ref{feat:params_rpca} and \ref{feat:params_rgpca}
require computing principal
components of the residual; \ref{feat:params_rgpca} and \ref{feat:params_pr}
require constructing the sampling matrix $\sampMat$, which we do via
q-sampling~\cite{qdeim_drmac} (which in turn requires
residual principal components).
In these cases, computing residual principal
components requires the additional training set
\begin{equation}\label{eq:residualTrainSet}
	\residualTrainSet \defeq  \{ \resApproxArg{\tindx} \,|\,
\parametervec \in \parameterdomainTrain, \tindx \in \timeIndexSetTrain \}.
\end{equation}
%
Because principal component analysis is an unsupervised machine learning
method, it does not require validation or test sets.
\subsubsection{Regression function}\label{sec:regressionfunction}
We now describe \ref{step:regressionModel} of the proposed framework:
training the deterministic regression-function model $\errorApproxMeanGennmu$
of the form \eqref{eq:ml_deterministic}--\eqref{eq:latentStateRecursion}. In
particular, this step constructs the mappings
$\approxregressionfunction$
and $\recursionf$.
To achieve this, we consider a variety of both recursive and non-recursive
regression methods from classical time-series analysis and supervised machine
learning. In particular, we consider three categories of regression methods:
\begin{Category}
\item \label{cat:nonrecursive}\textbf{Non-recursive.} Non-recursive regression methods do not include
	latent variables or latent-variable dynamics. In this case, the latent variables can
	be omitted from the regression function arguments, i.e., 
$\approxregressionfunction(\featurevec,\latentstate)\equiv
\approxregressionfunction(\featurevec)$, and the mapping $\recursionf$
	is not employed.
\item \label{cat:recursiveLatentPred}\textbf{Recursive with latent state equal to previous prediction.}
	This category of methods employs latent variables, but sets them equal to
	the predicted response at the previous time instance, i.e.,
		$
	\latentstatefnArg{\timeIndexMapping{n}} = 
		\errorApproxMeanGennnArg{\timeIndexMapping{n-1}}$,
		$n=1,\ldots,\ntimestepsCoarse$.
This is equivalent to prescribing the recursion function as
		simply
	\begin{equation} \label{eq:recursionCase2}
\recursionf:(\featurevec,\latentstate,
	\errorApproxMeanGen
	)\mapsto 
	\errorApproxMeanGen.
\end{equation} 
These methods place no particular restriction on the mapping
	$\approxregressionfunction$.
\item \label{cat:recursive}\textbf{Recursive.} This category of methods follows the general
	framework outlined in Section \ref{sec:mlFramework} with no particular
	restrictions on the forms of the mappings
$\approxregressionfunction$ or $\recursionf$.
\end{Category}
Table~\ref{tab:regression_summary} summarizes the considered
methods and their associated categories.

\begin{table}
\begin{small}
\centering
\renewcommand{\arraystretch}{2}
\begin{tabular}{l c c c c }\hline
Description &  Description & Classification & Latent dynamics & Latent dimension $\nlatentf$\\
\hline
	 $k$-nearest neighbors & kNN & \ref{cat:nonrecursive} & None  & N/A\\
	 Artificial neural network & ANN &  \ref{cat:nonrecursive}  & None & N/A \\
 	 Autoregressive w/ exogenous inputs & ARX  &\ref{cat:recursiveLatentPred}  & Linear & 1 \\
 Integrated artificial neural network &	ANN-I &  \ref{cat:recursiveLatentPred} & Linear & 1\\
	Latent autoregressive w/ exogenous inputs & LARX  &  \ref{cat:recursive}  & Linear & Arbitrary\\
	Recurrent neural network &RNN &  \ref{cat:recursive}  & Nonlinear & Arbitrary\\
	Long short-term memory  network  &LSTM &  \ref{cat:recursive} & Nonlinear & Arbitrary\\
\hline
\end{tabular}
	\caption{Summary of considered regression methods.}
\label{tab:regression_summary}
\end{small}
\end{table}

\begin{itemize}


\item \textbf{$k$-nearest neighbors (kNN)} (\ref{cat:nonrecursive}): kNN is a
	simple regression method wherein the prediction comprises a weighted average
		of the nearest neighbors in feature space such that
		\begin{equation}
\approxregressionfunction :\featurevec
			\mapsto \frac{1}{k} \sum_{
				(\bar\errorGen,\bar\featurevec)\in\trainSetKNN(\featurevec)} w(
			\featurevec,\bar\featurevec ) \bar\errorGen
		\end{equation}
where
		$\trainSetKNN(\featurevec)\subseteq\trainSet$	with
		$\card{\trainSetKNN(\featurevec)}=k$ denotes the $k$ response--feature
		pairs of the
		training set $\trainSet$ whose features are closest to the vector $\featurevec$ in
		feature space. Common choices for the weights are uniform weights,
		$w:(\featurevec,\bar\featurevec)\mapsto 1/k$ and weights based on
		Euclidean distance, i.e.,
\begin{equation} 
	w:(\featurevec,\bar\featurevec)\mapsto  \frac{ (
	\sum_{(\tilde\errorGen,\tilde\featurevec)\in\trainSetKNN(\featurevec)} \norm{
		\featurevec - \tilde\featurevec }^{-1} )^{-1} }
{\norm{ \featurevec-\bar\featurevec  } } . 
\end{equation} 
The hyperparameters associated with this method correspond to the number of
		neighbors $k$, and the definition of the weights $w$. We note that all
		parameters characterizing this model are hyperparameters; there are no
		parameters that are subject to training via optimization.
\item \textbf{Artificial neural network (ANN)} (\ref{cat:nonrecursive}):
	Feed-forward ANNs (i.e., multilayer perceptrons) employ function composition
		to create nonlinear input--output maps.  In the
		context of error modeling, Ref.~\cite{freno_carlberg_ml} observed ANNs to
		yield the best performance for static systems. In the present context,
		feed-forward ANNs can be used to construct the (parameterized)
		non-recursive regression function that satisfies
$\approxregressionfunction(\featurevec,\latentstate) 
	\equiv
\approxregressionfunction(\featurevec) 
$
and is given by
\begin{equation*}
\approxregressionfunction : (\featurevec;\weightsf) \mapsto
	\activation_{\nlayers} ( \cdot ,\weights_{\nlayers} ) \circ \cdots
	\circ \activation_1 ( \featurevec; \weights_{1} ) ,
\end{equation*}
where 
		$\weightsf\equiv(\weights_1,\ldots,\weights_{\nlayers})$ with $\weights_i
		\equiv(\weightsActivation_i , \bias_i)\in\RR{p_i\times p_{i-1}}\times
		\RR{p_i}$, $i=1,\ldots,\nlayers$ denotes
		neural-network weights and biases;
		$\activation_i(\cdot;\weights_i): \RR{p_{i-1}} \rightarrow \RR{p_{i}}$,
		$i=1,\ldots,\nlayers$ denotes the activation function employed at the
		$i$th layer; $d$ denotes the depth of the network; and
		$p_i$, $i=0,\ldots,\nlayers$ (with $p_0 = \nfeatures$ and $p_\nlayers=1$) denotes
		the number of neurons at layer $i$. The activation
		function takes the form
$$ \activation_i : (\dummyArg;\weights_i) \mapsto
		\activationf_i(\weightsActivation_i \dummyArg+ \bias_i),\quad
		i=1,\ldots,\nlayers,$$
		where $\activationf_i: \RR{p_{i}} \rightarrow \RR{p_{i}}$,
		$i=1,\ldots,\nlayers$ denote activations (e.g., 
		rectified linear unit
		(ReLU), hyperbolic tangent) applied elementwise.


In this work, we consider only fully connected feed-forward ANNs with ReLU activation functions. As such,
we consider the number of layers $\nlayers$ and the number of neurons at each layer $p_i$,
$i=1,\ldots,\nlayers-1$ to be hyperparameters.

\item \textbf{Autoregressive with exogenous inputs (ARX)}
	(\ref{cat:recursiveLatentPred}):
Autoregressive methods are linear models that have been widely adopted for time-series
		analysis. The ARX model is an extension of the autoregressive (AR) model that
		allows for exogenous inputs, which we treat as the features. ARX$\arxargs$
		methods regress the response at given
		time index as a linear combination of the value of the response at the 
		previous $\arxarga$ time instances and the exogenous inputs at both the
		current time instance and the previous 
		$\arxargb-1$ time instances. This work considers the ARX$(1,1)$ model, in which case 
		the (parameterized) regression function is given by
\begin{equation} 
	\approxregressionfunction:(\featurevec,\latentstate;\weightsf)
\mapsto 
\ARXfweightsOne^T \featurevec+ \ARXfweightsTwo\latentstate  + \scalarbias,
\end{equation} 
		where
		$\weightsf\equiv(\ARXfweightsOne,\ARXfweightsTwo,\scalarbias)\in\RR{\nfeatures}\times\RR{}\times \RR{}$
		denote the regression-function weights and bias, 
		and the latent variable
		(of dimension $\nlatentf=1$)
		corresponds to the 
	predicted response at the previous time instance such that 
	$
	\latentstatefnArg{\timeIndexMapping{n}} = 
		\errorApproxMeanGennnArg{\timeIndexMapping{n-1}}$,
		$n=1,\ldots,\ntimestepsCoarse$
		and
	the
	latent-variables dynamics are prescribed via \eqref{eq:recursionCase2}. 

	We remark that classical autoregressive
models typically include the noise term directly within the recursion,
i.e., the latent variables $\latentstatefnArg{\timeIndexMapping{n}}$
correspond to the regression-function prediction perturbed by stochastic noise.
As a result, generating statistical predictions requires simulating an
ensemble of model realizations. To avoid this complication, this work
constructs the stochastic noise model \textit{separately} from the
deterministic regression-function model; for the stochastic noise models considered 
in this work, the statistical distribution of the response is given analytically.
Sections \ref{sec:noise_modeling}--\ref{sec:noise} describe this step. We also
note that it is
possible to construct higher-order autoregressive models (i.e., AR($p,q$) with
$p>1$ and/or $q>1$) that include additional lagged states. While we do not
consider such approaches in this work, the ARX formulation presented above could be modified
to address this case by changing the definition of the latent variables and
recursion function. 
\item \textbf{Integrated artificial neural network (ANN-I)}
	(\ref{cat:recursiveLatentPred}):
To facilitate time-series modeling, it is common to 
apply transformations to the training data in order to make the associated
processes weakly
stationary. The most commonly adopted transformation is time differencing. This
process leads to \textit{integrated} methods. Here, we investigate the use
of an ``integrated" artificial neural network (ANN-I) model, wherein a neural
network is used to predict the
time-differenced error. This leads to a modification of the ANN model wherein
the (parameterized) recursive regression
function takes the form
\begin{equation*}
\approxregressionfunction :
	(\featurevec,\latentstate;\weightsf)
	\mapsto \latentstate +
	\activation_{\nlayers} ( \cdot ,\weights_{\nlayers} ) 
	\circ\cdots
	\circ \activation_1 (
	\featurevec , \weights_{1} ) ,
\end{equation*}
where the parameters
		$\weightsf\equiv(\weights_1,\ldots,\weights_{\nlayers})$ and activations
	$\activation_{i}$,	 $i=1,\ldots,\nlayers$
		are defined as in the ANN model and 
		the latent variable 
		(of dimension $\nlatentf=1$)
		corresponds to the 
	predicted response at the previous time instance such that 
	$
	\latentstatefnArg{\timeIndexMapping{n}} = 
		\errorApproxMeanGennnArg{\timeIndexMapping{n-1}}$,
		$n=1,\ldots,\ntimestepsCoarse$
		and
	\eqref{eq:recursionCase2} holds.  
As with ANN,
we consider only fully connected feed-forward networks with ReLU
activation functions such that the hyperparameters correspond to 
the number of layers $\nlayers$ and the number of neurons at each layer $p_i$,
$i=1,\ldots,\nlayers-1$. 
\item \textbf{Latent autoregressive with exogenous inputs (LARX)}
	(\ref{cat:recursive}):
	Because ARX falls into \ref{cat:recursiveLatentPred}, it employs the
	predicted response at the previous time instance as the latent variables
	according to Eq.~\eqref{eq:recursionCase2}. The LARX method relaxes this
	assumption and considers higher-dimensional latent variables governed by
	general linear
	dynamics.
	This equips the model with a higher capacity than ARX and thus enables it to capture
	more complex non-local dependencies. In particular, LARX defines the (parameterized) regression
	function as a linear mapping from the latent variables to the prediction
	that satisfies
	$\approxregressionfunction(\featurevec,\latentstate) 
	\equiv
\approxregressionfunction(\latentstate) 
$ and is
	given by
\begin{equation*}
\approxregressionfunction : (\latentstate;\weightsf)
\mapsto 
\weights_1^T \latentstate + \scalarbias
\end{equation*}
where $\weightsf \equiv (\weights_1 , \scalarbias)
\in\RR{\nlatentf}\times\RR{}
$. The latent-variable recursion operator
is linear in its first two arguments such that
$
\recursionf(\featurevec,\latentstate,
	\errorApproxMeanGen
	)\equiv
\recursionf(\featurevec,\latentstate
	)
$ with
\begin{equation*} 
\recursionf:(\featurevec,\latentstate
	; \weightsg)\mapsto
        \weightsmat_{\feature} \featurevec + \weightsmat_{\latentstate}
				\latentstate + \bias_{\latentstate}, 
\end{equation*} 
where $\weightsg\equiv(\weightsmat_{\feature},\weightsmat_{\latentstate},\bias_{\latentstate})
\in\RR{N_{\latentstate} \times \nfeatures}\times
\RR{N_{\latentstate} \times N_{\latentstate}}\times\RR{\nlatentf}
$.

LARX is characterized by one hyperparameter: the number of latent variables
$\nlatentf$. 
\item \textbf{Recurrent neural network (RNN)} (\ref{cat:recursive}):
RNNs~\cite{rnn} comprise a particular neural-network architecture that have
been widely adopted for modeling time sequences.  The main idea of an RNN is to
employ the same feed-forward neural network for prediction at each point in a sequence,
but to treat the hidden layer of the neural network as inputs to the neural
network at the next point in the sequence; these hidden layers associate with the latent variables.
RNNs comprise a generalization of LARX models: LARX methods develop a linear recursion for the latent-variable 
dynamics, while RNNs enable nonlinear latent-variable dynamics.
Figure~\ref{fig:RNN} illustrates the computational graph for an ``unrolled'' RNN. 
An RNN of depth $d$ employs latent variables
$\latentstate \equiv \left(
	\latentstate_1, \ldots, \latentstate_{\depth}\right),
$ 
with 
$\latentstate_i\in\RR{p_i}$, $i=1,\ldots,\nlayers$ and
$\nlatentf=\sum_{i=1}^d p_i$.  For RNNs, the
deterministic regression-function model is given by an output activation
function acting on the latent variables at the final layer;
it
satisfies
$\approxregressionfunction(\featurevec,\latentstate) 
	\equiv
\approxregressionfunction(\latentstate) 
$  
and is given by
\begin{equation}\label{eq:RNNactivation}
\approxregressionfunction : (\latentstate;\weightsf) \mapsto
\activationfscalar(\weightsActivationVec_f^T \latentstate_{\depth}+ \scalarbias_f)
\end{equation}
where 
$\weightsf \equiv (\weightsActivationVec_f , \scalarbias_f) \in \RR{
	\nneurons_d}\times \RR{} $ and
		$\activationfscalar(\cdot)$ is typically the identity operator.

		The latent variables satisfy the recursive relationships
\begin{equation}\label{eq:RNNrecursion}
\latentstatefnArg{\timeIndexMapping{n}}_i = 
	\activationRecurrent_i(  \latentstatefnArg{\timeIndexMapping{n}}_{i-1},
	\latentstatefnArg{\timeIndexMapping{n-1}}_i ; \weights_i ),\quad i
	=1,\ldots,\nlayers,\quad n=1,\ldots,\ntimestepsCoarse,
%
\end{equation}
where
$\activationRecurrent_i(\cdot,\cdot,\weights_i) : \RR{\nneurons_{i-1}} \times
\RR{\nneurons_i} \rightarrow \RR{\nneurons_i}$, $i=1,\ldots,\nlayers$ 
denote activation functions and  the latent state at layer zero is defined
from the input features such that
$
\latentstatefnArg{\timeIndexMapping{n}}_{0} =
\latentstatezero(\featurefnArg{\timeIndexMapping{n}})$
with $\latentstatezero:\RR{\nfeatures}\rightarrow\RR{{\nneurons_0}}$.
These recursive relationships can be expressed by a latent-variable
recursion operator, which satisfies $
\recursionf(\featurevec,\latentstate, 
	\errorApproxMeanGen
	)\equiv
\recursionf(\featurevec,\latentstate
	)
$  
and is given by
\begin{equation}\label{eq:RNNrecursionOp}
\recursionf : (\featurevec,\latentstate
	;\weights_{\recursionf}) \mapsto 
\begin{bmatrix}
	\activationRecurrent_1(  \latentstatezero(\featurevec) , \latentstatefnArg{}_1   ;  \weights_1 ) \\
\activationRecurrent_2(   \cdot ,\latentstatefnArg{}_2 ; \weights_2 ) \circ \activationRecurrent_1( \latentstatezero(\featurevec),  \latentstatefnArg{}_1  ;  \weights_1 )  \\
 \vdots \\
\activationRecurrent_d(   \cdot ,\latentstatefnArg{}_d ; \weights_d ) \circ \cdots \circ  \activationRecurrent_1(\latentstatezero(\featurevec),  \latentstatefnArg{}_1  ;  \weights_1 )
\end{bmatrix}
,
\end{equation}
where $\weights_{\recursionf} \equiv (\weights_1 ,\ldots, \weights_d
)$.
 
For standard RNNs, the function applied at the $i$th layer takes the form
\begin{equation}\label{eq:activation_recurrent}
\activationRecurrent_i : ( \latentstate_{i-1}, \latentstate_{i}; \weights_i)
	\mapsto \activationf_i ( \weightsActivation_{i,1} \latentstate_{i-1} +
	\weightsActivation_{i,2} \latentstate_{i} + \bias_i ),
\end{equation}
where 
$\weights_i \equiv
(\weightsActivation_{i,1},\weightsActivation_{i,2},\bias_i)
\in\RR{\nneurons_i\times\nneurons_{i-1}}\times\RR{\nneurons_i\times\nneurons_i}\times\RR{\nneurons_i}
$, $i=1,\ldots,d$
denote the
recursive weights and biases,
$\activationf_i : \RR{\nneurons_{i}}\rightarrow\RR{\nneurons_i}$,
$i=1,\ldots,\nlayers$ again denote activations (e.g.,
ReLU) applied elementwise, and $\latentstatezero$ is the identity operator
such that $\latentstatezero^{\timeIndexMapping{n}}
=\featurefnArg{\timeIndexMapping{n}}$, $n=1,\ldots,\ntimestepsCoarse$ and $p_0
=\nfeatures$.

Different RNN architectures can be obtained, for example, by modifying the
activation functions, the number of latent variables, or the network depth.
The LARX method can be recovered from RNN by setting the depth $\depth=1$ and
setting $\activationfscalar(\cdot)$ and $\activationf_1(\cdot)$ to be the
identity mappings.   We consider hyperparameters corresponding to the
number of neurons at each layer and the depth of the network.

\begin{figure}
\centering
\begin{subfigure}[t]{1.\textwidth}
\includegraphics[trim={0cm 5cm 10cm 5cm},clip,width=1.\linewidth]{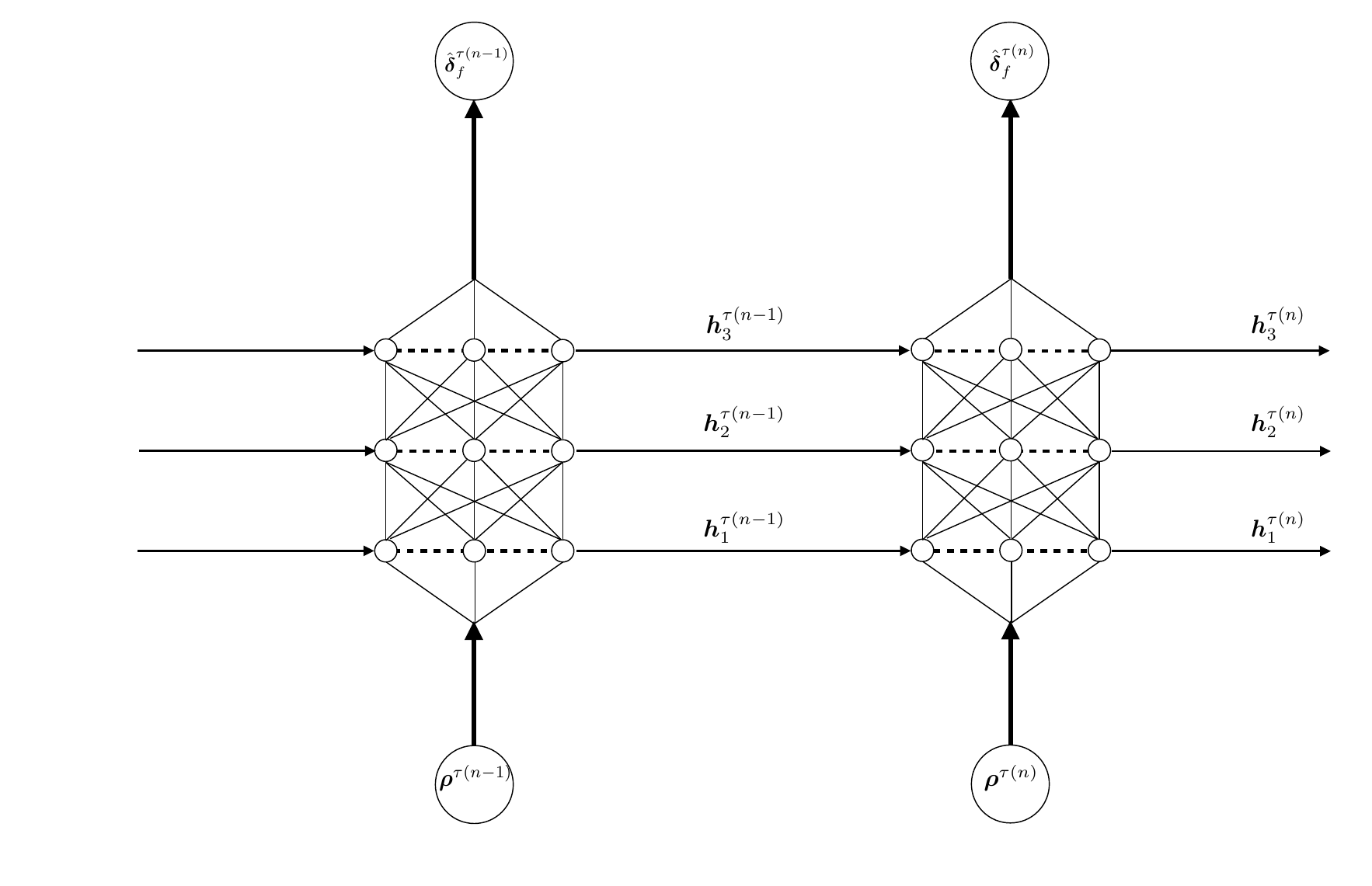}
\end{subfigure}

\caption{Diagram of an ``unrolled" recurrent neural network with depth
	$\depth=3$.}
\label{fig:RNN}
\end{figure}

\item \textbf{Long short-term memory network (LSTM)} (\ref{cat:recursive}):
The LSTM network~\cite{lstm} is a particular type of RNN designed to capture
dependencies on long-term non-local quantities. Figure~\ref{fig:LSTM} provides
a diagram of an ``unrolled" LSTM. LSTMs are able to capture long-term
non-local dependencies through the use of two latent variables: the
\textit{hidden state} and the \textit{cell state}. The cell state, which is
depicted as the topmost horizontal line running through the cells in
Figure~\ref{fig:LSTM}, undergoes only minor interactions~\cite{colah_blog}:
elementwise multiplication and elementwise addition.  Due to the minor
changes experienced by the cell state, it can (in principle) more easily
capture long-term dependencies. The LSTM architecture is arguably the most
widely adopted RNN architecture as of this writing. 

Formally, an LSTM decomposes latent variables at the $i$th layer as
$\latentstate_i\equiv\left(\latentstateh_i,\cellstate_i\right)$,
where $\latentstateh_i \in \RR{\nneurons_i/2}$ and
$\cellstate_i \in \RR{\nneurons_i/2}$ denote the hidden state and cell state, respectively.
The deterministic regression function for LSTMs is defined as in
\eqref{eq:RNNactivation}, but with elements $\nneurons_i/2+1$ to $\nneurons_i$
of the weighting vector $\weightsActivationVec_f$ set to zero so that the
output depends only on $\latentstateh_\nlayers$.


The latent-variable recursion is also defined as in
\eqref{eq:RNNrecursion}--\eqref{eq:RNNrecursionOp}; however, the activations
take a special form based on the LSTM ``cell''.
In particular, the LSTM cell corresponds to defining the $i$th layer's
function as
\begin{equation}\label{eq:activation_recurrent}
\activationRecurrent_i : ( \latentstate_{i-1}, \latentstate_{i}; \weights_i)
\mapsto \begin{bmatrix}
\activationRecurrent_{\sigma,i} ( \latentstateh_{i-1} , \latentstateh_i;
	\weights_{i,1} ) 
		\odot \tanh (
\activationRecurrent_{\text{cell},i}
(\latentstateh_{i-1},\latentstate_{i};\weights_{i}
)		) \\
		\activationRecurrent_{\text{cell},i}
(\latentstateh_{i-1},\latentstate_i;\weights_{i}
)
\end{bmatrix},
\end{equation}
where 
$\latentstate_{i}\equiv(\latentstateh_i,\cellstate_i)$, 
$\weights_i\equiv(\weights_{i,1},\ldots,\weights_{i,4})$, 
and $\odot$ denotes the elementwise (Hadamard)
product, and
\begin{equation*}
\activationRecurrent_{\text{cell},i}:
(\latentstateh_{i-1},\latentstate_{i};\weights_{i}
)
\mapsto
\activationRecurrent_{\sigma,i}
		(\latentstateh_{i-1},\latentstateh_i;\weights_{i,2}) \odot \cellstate_i + 
    \activationRecurrent_{\sigma,i}
		(\latentstateh_{i-1},\latentstateh_i;\weights_{i,3}) \odot
		\activationRecurrent_{\tanh,i}(\latentstateh_{i-1},\latentstateh_i;\weights_{i,4})
\end{equation*}
	denotes the cell activation. Here,
\begin{align*}
	&\activationRecurrent_{\sigma,i}:(\latentstateh_{i-1},\latentstateh_{i};\weights_{i,j})\mapsto
	\sigma( \weightsActivation_{i,j,1} \latentstateh_{i-1} +
	\weightsActivation_{i,j,2} \latentstateh_{i} + \bias_{i,j} ), \quad
	j=1,\ldots,3\\
	&\activationRecurrent_{\tanh,i}:(\latentstateh_{i-1},\latentstateh_{i};\weights_{i,4})\mapsto
\tanh( \weightsActivation_{i,4,1} \latentstateh_{i-1} +
	\weightsActivation_{i,4,2} \latentstateh_{i} + \bias_{i,4} ),
\end{align*}
where
$\weights_{i,j}\equiv(\weightsActivation_{i,j,1},\weightsActivation_{i,j,2},\bias_{i,j})\in
\RR{\nneurons_i/2\times\nneurons_{i-1}/2}\times\RR{\nneurons_i/2\times\nneurons_i/2
}\times\RR{\nneurons_i/2}$, $i=1,\ldots,\nlayers$, $j=1,\ldots,4$
denote the sigmoid and hyperbolic tangent operators, respectively. For LSTM,
the latent state at layer zero is defined via
$\latentstatezero:\featurevec\mapsto[\featurevec^T\ \zero^T]^T$ such that
such that $\latentstateh_0^{\timeIndexMapping{n}}
=\featurefnArg{\timeIndexMapping{n}}$, $n=1,\ldots,\ntimestepsCoarse$ and $p_0
=2\nfeatures$.

Different types of LSTM networks can be constructed, e.g., 
by
modifying 
the definition of the LSTM cell. For instance, the ``peephole" LSTM 
network can be obtained by replacing the hyperbolic tangent function in Eq.~\eqref{eq:activation_recurrent} 
with the identity function.
We consider hyperparameters corresponding to the
number of neurons $\nneurons_i$, $i=1,\ldots,\nlayers$ and the network depth
$\depth$.  


\begin{figure}
\centering
\includegraphics[trim={0cm 3cm 9.6cm 3cm},clip,width=1.0\linewidth]{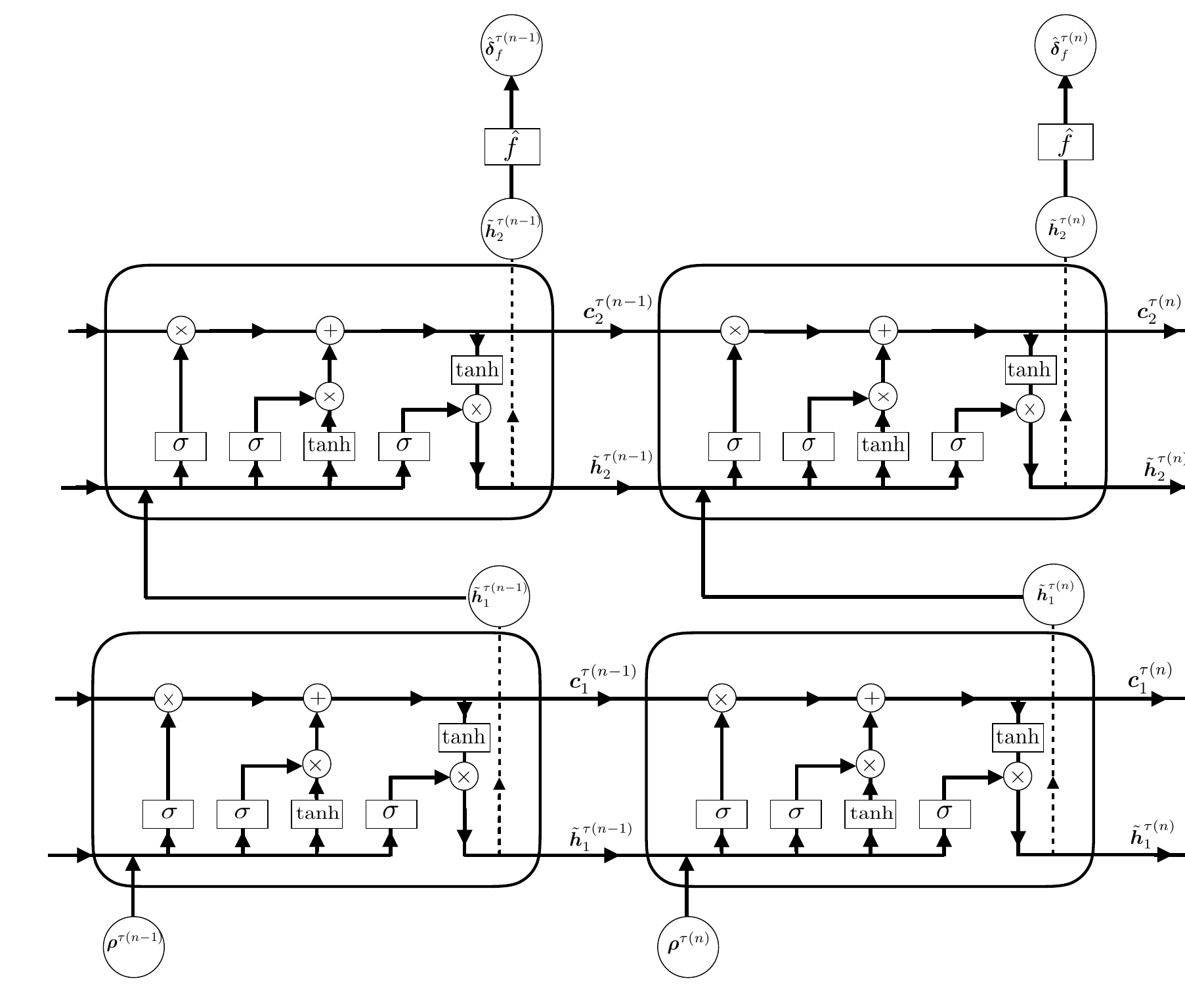}
\caption{Diagram for a deterministic regression-function model utilizing the LSTM architecture with a depth of $\depth=2$. Figure concept adopted from
	Ref.~\cite{colah_blog}.}
\label{fig:LSTM}
\end{figure}
\end{itemize}


To summarize, this section outlined a variety of non-recursive and recursive
regression-function models. This exposition started with a description of the
kNN and ANN regression-function models. These regression-function models are
non-recurrent and employ no latent variables; non-local effects cannot be
captured with kNN and ANN. To capture non-local effects, we introduced the ARX and ANN-I
regression-function models. Both of these models employ a single
latent variable corresponding to the prediction at the previous time instance;
both methods are thus characterized by linear latent dynamics with a latent
dimension of $\nlatentf = 1$. To increase the capacity of the latent-dynamics
recursion, we introduced the LARX regression-function model. Unlike ARX
and ANN-I, LARX employs an arbitrary number of latent variables governed by
general linear dynamics. Finally, to further increase the capacity of the
latent dynamics recursion, we introduced the RNN and LSTM regression-function
models. These methods comprise an extension of LARX as, in addition to
supporting an arbitrary number of latent variables, they allow for nonlinear
latent-variable dynamics. The regression-function models described in this
section are thus characterized by the following hierarchy of models with increasing
capacity: no latent variables (kNN, ANN), one latent variable with linear
latent-variable dynamics (ARX, ANN-I), an arbitrary number of latent
variables with linear latent-variable dynamics (LARX), and an
arbitrary number of latent variables with nonlinear latent-variable dynamics
(RNN, LSTM).

\subsubsection{Regression-function training}\label{sec:regressionTraining}

To train a regression model, we employ a validation procedure consisting of
two loops: the outer loop modifies hyperparameters characterizing the model, while the inner loop computes optimal model parameters $\weightsf^\star$ and
$\weightsg^\star$ via optimization executing on the training set $\trainSet$
given fixed hyperparameter values.  The final model is chosen based on
hyperparameter values that yield the smallest MSE on the validation set
$\valSet$.
In the case of limited data, this process can be replaced by executing
cross-validation within the training set $\trainSet$; this obviates the need
for an independent validation set $\valSet$.

The inner-loop optimization problem depends on the category within which the
regression model falls.\footnote{We note that kNN is characterized by only
hyperparameters and thus the inner loop does not require solving an
optimization problem.}
\begin{enumerate} 
\item 
For regression models in \ref{cat:nonrecursive}, we compute optimal model parameters
$\weightsf^\star$ as the solution to the minimization problem
\begin{align}\label{eq:nonrecursiveOpt}
\begin{split}
	\underset{\weightsf}{\text{minimize}}&\ \sum_{\parametervec \in
	\parameterdomainTrain}\sum_{
		n\in\timeIndexSetTrain}  \norm{ 
\errorGennmuArg{n}	
-
	\approxregressionfunction(\featurefnmuArg{n};\weightsf)
	}^2 + \regTermf(\weightsf),
\end{split}
\end{align}
where $\regTermf(\cdot)\in\RR{}_+$ is a regularization operator, e.g., LASSO,
ridge, elastic--net. We note that optimization problem
\eqref{eq:nonrecursiveOpt} is nonrecursive, as the prediction at one time
instance does not affect the objective-function contribution associated with
the prediction at the next time instance.
\item 
For regression models in \ref{cat:recursiveLatentPred}, we can train the
models in one of two ways. The first approach is non-recursive in nature and
computes the optimal model parameters
$\weightsf^\star$ as the solution to the minimization problem
\begin{align}\label{eq:recursiveLatentPredOpt}
\begin{split}
	\underset{\weightsf}{\text{minimize}}&\ \sum_{\parametervec \in
	\parameterdomainTrain}\sum_{
		n=1}^{\card{\timeIndexSetTrain}}  \norm{ 
\errorGennmuArg{\timeIndexMapping{n}}	
-
\approxregressionfunction(\featurefnmuArg{\timeIndexMapping{n}},\errorGennmuArg{\timeIndexMapping{n-1}};\weightsf)
	}^2 + \regTermf(\weightsf).
\end{split}
\end{align}
Because the actual error $\errorGennmuArg{\timeIndexMapping{n-1}}$ 
from the training data is employed
as the second argument in $\approxregressionfunction$ rather than the
predicted value 
$
		\errorApproxMeanGennnArg{\timeIndexMapping{n-1}}$ (which is the
actual
value of the hidden state), the optimization problem
\eqref{eq:recursiveLatentPredOpt} is non-recursive.
We refer to this 
approach as non-recursive training (NRT).
\item 
Finally, regression models in either 
\ref{cat:recursiveLatentPred} and \ref{cat:recursive}
can be trained via a dynamics-constrained recursive training strategy. In this approach, the 
optimal weights $\weightsf^\star$ and $\weightsg^\star$ are the solution to the problem
\begin{align}\label{eq:BPTTopt}
\begin{split}
	\underset{\weightsf,\weightsg}{\text{minimize}}&\ \sum_{\parametervec \in
	\parameterdomainTrain}\sum_{
		n\in\timeIndexSetTrain}  \norm{ 
\errorGennmuArg{n}	
-
	\approxregressionfunction(\featurefnmuArg{n},\latentstatefnmuArg{n};\weightsf)
	}^2+\regTermf(\weightsf) + \regTermg(\weightsg)\\
	\text{subject to}\ &
	\latentstatefnmuArg{\timeIndexMapping{n}} =
	\recursionf(\featurefnmuArg{\timeIndexMapping{n}},\latentstatefnmuArg{\timeIndexMapping{n-1}},
	\errorApproxMeanGennmunArg{\timeIndexMapping{n-1}}
	;\weightsg),\quad
\parametervec \in
	\parameterdomainTrain,\ 
	n=1,\ldots,\card{\timeIndexSetTrain},
\end{split}
\end{align}
where $\regTermg(\cdot)\in\RR{}_+$ is a regularization operator and the
latent-variable dynamics appear explicitly as constraints.
In this approach, the prediction
$\errorApproxMeanGennnArg{\timeIndexMapping{n-1}}(\parametervec)$ \textit{does} affect the
contribution to the objective function at time instance $\timeIndexMapping{n}$ and
parameters $\parametervec$ through
the constraints. Thus, this formulation reflects the recursive nature of the
regression model and can enable more accurate long-time predictions at the
expense of more computationally expensive training. We refer to this approach 
as recursive training (RT). We note that the ``backpropagation through time" (BPTT) 
method~\cite{Mozer1989AFB,robinson:utility,Goodfellow-et-al-2016} is a popular approach for solving the minimization 
problem~\ref{eq:BPTTopt}. We additionally note that for models in
\ref{cat:recursiveLatentPred}, problem \eqref{eq:BPTTopt} can be derived by
substituting 
$
\errorGennmuArg{\timeIndexMapping{n-1}}\leftarrow
\errorApproxMeanGennmunArg{\timeIndexMapping{n-1}}
$
in problem \eqref{eq:recursiveLatentPredOpt}.
\end{enumerate}

In the outer loop, for each candidate
value of the hyperparameters, optimization problem \eqref{eq:nonrecursiveOpt}, \eqref{eq:recursiveLatentPredOpt}, or
\eqref{eq:BPTTopt} is solved for parameters $\weightsf^\star$ and
$\weightsg^\star$, and the resulting model is evaluated on the
validation set $\valSet$. For regression models in
\ref{cat:nonrecursive}, we evaluate the validation criterion
\begin{align}\label{eq:valOne}
\begin{split}
	\sum_{\parametervec \in
	\parameterdomainVal}\sum_{
		n\in\timeIndexSetVal}  \norm{ 
\errorGennmuArg{n}	
-
	\approxregressionfunction(\featurefnmuArg{n};\weightsf^\star)
	}^2,
\end{split}
\end{align}
while for methods in \ref{cat:recursiveLatentPred} or \ref{cat:recursive},
we evaluate the validation criterion
\begin{equation}\label{eq:valTwo}
	\sum_{\parametervec \in
	\parameterdomainVal}\sum_{
		n\in\timeIndexSetVal}  \norm{ 
\errorGennmuArg{n}	
-
	\approxregressionfunction(\featurefnmuArg{n},\latentstatefnmuArg{n};\weightsf^\star)
	}^2,
\end{equation}
where the dynamics of the latent variables are defined as
	$\latentstatefnmuArg{\timeIndexMapping{n}} =
	\recursionf(
	\featurefnmuArg{\timeIndexMapping{n}},
	\latentstatefnmuArg{\timeIndexMapping{n-1}},
	\errorApproxMeanGennmunArg{\timeIndexMapping{n-1}}
	;\weightsg^\star)$,
$\parametervec \in
	\parameterdomainVal$,
	$n=1,\ldots,\card{\timeIndexSetVal}$. For each regression model, the
	hyperparameter values yielding the smallest validation-criterion value are selected.

\subsubsection{Noise model}\label{sec:noise_modeling}
We now describe \ref{step:noiseModel} of the proposed framework: training the
stochastic noise model $\errorApproxRandGennmu$ of the form
\eqref{eq:ml_stochastic}.  While complex noise models could be considered in
principle, for the sake of simplicity, we consider here only three types of
noise models: a homoscedastic Gaussian model, a homoscedastic Laplacian model,
and an autoregressive model.

\begin{table}
\centering
\begin{tabular}{c c c}\hline
Description &Noise Distribution &  Time-dependent  \\ \hline
Stationary Gaussian & Gaussian & No\\
Stationary Laplacian & Laplacian & No\\
Autoregressive 1 (AR1) & Gaussian & Yes\\
\hline
\end{tabular}
	\caption{Summary of considered noise models.}
\label{tab:noise_summary}
\end{table}

\begin{itemize}
\item \textbf{Stationary Gaussian homoscedastic noise model}:\label{noise:gaussian}
Arguably, the simplest way to quantify the discrepancy between the error model
		and the true error is to assume the errors to be uncorrelated, Gaussian,
		stationary, and homoscedastic. The resulting error model is independent of
		the prediction at the previous time instance such that
		$\approxnoisefunction(\errorApproxRandGenDummy,\noiseDummy)\equiv
		\approxnoisefunction(\noiseDummy)
		$
		with
\begin{equation}\label{eq:noise1}
\approxnoisefunction: \noiseDummy \mapsto
	\noiseDummy, 
\end{equation}
		and defining $\noisenArg{n} \in \randomvars{}$,
$n\in\timeIndexSet$ as a sequence of independent and identically distributed
		mean-zero Gaussian random variables, i.e.,
		\begin{equation}\label{eq:seqGaussian}
\noisenArg{n} \sim \normalDistribution(0,\std^2),\quad n\in\timeIndexSet.
		\end{equation}
This model does not have any latent variables 
and requires estimating only the variance parameter $\std^2$.

\item \textbf{Stationary Laplacian homoscedastic noise model}:\label{noise:laplace}
This approach is similar to the Gaussian model, but instead models the random
		variables as Laplacian, i.e., \eqref{eq:noise1} holds with
$$\noisenArg{n}
		\sim \laplaceDistribution(0,b),\quad n\in\timeIndexSet,$$
where $b>0$ denotes the diversity, which must be estimated.
Again, this error model does not require any input features or latent variables.
\item \textbf{Autoregressive Gaussian homoscedastic discrepancy model:}\label{noise:ar1}
Finally, we consider an autoregressive noise model, as we expect errors to
		display time correlations. We refer to this model as AR1.
		The model is given by
\begin{equation*}
\approxnoisefunction : (\errorApproxRandGen,\noiseDummy) \mapsto  c
	\errorApproxRandGen +  \noiseDummy ,
\end{equation*}
where $c \in \RR{}$ is a model constant, and 
		defining $\noisenArg{n} \in \randomvars{}$,
$n\in\timeIndexSet$ as a sequence of independent and identically distributed
		mean-zero Gaussian random variables such that \eqref{eq:seqGaussian} holds.
This method requires estimating the parameters $c$ and $\sigma^2$.  The auto-regressive model produces a 
series of (correlated) mean-zero Gaussian distributions where the variance at time-instance $\timeIndexMapping{n}$ is
defined recursively as
\begin{equation}\label{eq:ar1_var}
	\sigma^{\timeIndexMapping{n}}_{\noise} \defeq \sqrt{\sigma^2 +
	\sum_{i=1}^{n-1}  (c \sigma^{\timeIndexMapping{i}} )^2}, \qquad n = 1,\ldots,\ntimestepsCoarse.
\end{equation}
\end{itemize}

\subsubsection{Noise-model training}\label{sec:noise}
The stationary Gaussian, stationary Laplacian, and AR1 noise models require
specification of the parameters $\sigma^2$, $b$, and $( c, \sigma^2)$,
respectively. We compute these parameters via
maximum likelihood estimation (MLE) on a training set
$\trainSetNoise\subseteq\testSet$, and evaluate the performance of the noise model on a test
set $\testSetNoise\subseteq\testSet$, where $\trainSetNoise\cap\testSetNoise=
\emptyset$ and
\begin{equation} \label{eq:trainSetNoise}
\trainSetNoise\defeq\dataFunction(\paramtrainsetNoise,\timeIndexSetTrainNoise),\qquad
\testSetNoise\defeq\dataFunction(\paramtestsetNoise,\timeIndexSetTestNoise).
\end{equation} 
In this work, we employ 
$\paramtrainsetNoise, \paramtestsetNoise\subseteq \parameterdomainTest$ 
with 
$\paramtrainsetNoise\cap \paramtestsetNoise=\emptyset$ and
$\paramtrainsetNoise\cup \paramtestsetNoise=\parameterdomainTest$ 
with
$\timeIndexSetTrainNoise=\timeIndexSetTestNoise=\timeIndexSet$.
For the stationary Gaussian model, the MLE estimate for
$\sigma^2$ is given by
\begin{equation}\label{eq:trainStatGauss}
	\sigma^2_\text{MLE} = \frac{1}{\card{\trainSetNoise}}  
	\sum_{\parametervec\in\paramtrainsetNoise}\sum_{n\in\timeIndexSetTrainNoise}	
	 ( \errorGennmuArg{n}  - \errorApproxMeanGennmunArg{n} )^2 .
\end{equation}
For the stationary Laplacian model, the MLE estimate for $b$ is given by
\begin{equation}\label{eq:trainStatLaplace}
b_\text{MLE} =
	\frac{1}{\card{\trainSetNoise}}  
	\sum_{\parametervec\in\paramtrainsetNoise}\sum_{n\in\timeIndexSetTrainNoise}	
	 | \errorGennmuArg{n}  - \errorApproxMeanGennmunArg{n} |.
\end{equation}
Lastly, for the AR1 model, the MLE estimates for the parameters are given by
\begin{align}\label{eq:trainAR1}
	\begin{split}
	& c_\text{MLE} = \frac{ \sum_{\parametervec \in
	\paramtrainsetNoise}\sum_{\tindx=1}^{\card{\timeIndexSetTrainNoise}} (
	\errorGennmuArg{\timeIndexMapping{\tindx-1}} -
	\errorApproxMeanGennmunArg{\timeIndexMapping{\tindx-1}} ) (
	\errorGennmuArg{\timeIndexMapping{\tindx}} -
	\errorApproxMeanGennmunArg{\timeIndexMapping{\tindx}}) }
{ \sum_{\parametervec \in
	\paramtrainsetNoise}\sum_{\tindx=1}^{\card{\timeIndexSetTrainNoise}} ( \errorGennmuArg{\timeIndexMapping{\tindx-1}} - \errorApproxMeanGennmunArg{\timeIndexMapping{\tindx-1}})^2 },\\ 
& \sigma^2_\text{MLE} = \frac{1}{ \card{\trainSetNoise} } \sum_{\parametervec \in
	\paramtrainsetNoise}\sum_{\tindx=1}^{\card{\timeIndexSetTrainNoise}} (
	c_\text{MLE} \errorApproxRandGennmunArg{\timeIndexMapping{\tindx-1}}  -  ( \errorGennmuArg{\timeIndexMapping{\tindx}} - \errorApproxMeanGennmunArg{\timeIndexMapping{\tindx}})  )^2 .
	\end{split}
\end{align} 

\section{Numerical Experiments}\label{sec:numerical_experiments}

This section investigates the performance of the proposed \methodAcronym\
methods.  In particular, we consider all combinations of the 13
feature-engineering methods reported in Table \ref{tab:feature_summary} with
the 7 deterministic regression function models listed in Table
\ref{tab:regression_summary}. For the regression methods in
\ref{cat:recursiveLatentPred} (i.e., ARX and ANN-I), we construct two separate
regression-function models: one using non-recursive training (NRT) and one
using recursive training (RT) as described in
Section~\ref{sec:regressionTraining}.  In total, we consider 117 candidate
\methodAcronym\ methods.  We also compare the performance of the
\methodAcronym\ methods with that achieved by time-local Gaussian-process (GP)
regression in parameter space~\cite{pagani_enkfrom};
Appendix~\ref{appendix:gp} describes this method in detail.

We consider three examples. The first two examples employ approximate
solutions generated by a ROM constructed via Galerkin projection (as discussed
in Section \ref{sec:romsec}), while the third example considers approximate
solutions generated by a coarse-mesh lower-fidelity model 
(as discussed in Section \ref{sec:LFM}).

\subsection{Implementation and model evaluation details}

We now provide details on the implementation of the machine-learning
regression methods, training algorithms, hyperparameter grids, and evaluation
metrics. 
\subsubsection{Training and model selection of machine learning
algorithms}\label{sec:trainingModelSel}

We implement $k$-nearest neighbors and GP regression
algorithms with Scikit-learn~\cite{scikit-learn}, and all other
machine-learning regression methods with
Keras~\cite{chollet2015keras}. We employ
the following workflow to train and test 
all machine-learning regression models:
\begin{enumerate}
	\item\label{step:timeIndexSet} Define the coarse time grid $\timeIndexSet$ on which regression
	methods will operate.
\item\label{step:trainValTest} Define the training set 
	$\trainSet$, validation set $\valSet$,
	and test set 
		$\testSet$ according to \eqref{eq:trainTestSet} with
		$\timeIndexSetTrain=\timeIndexSetVal=\timeIndexSetTest=\timeIndexSet$ and
		$\parameterdomainTrain$, $\parameterdomainVal$, and $\parameterdomainTest$ generated by drawing
		independent samples of the parameters 
		$\parametervec\sim \mathcal{U}(\parameterdomain)$, where $\mathcal{U}(\cdot)$ denotes the
		uniform distribution. We ensure that $\card{\trainSet} = 4\card{\valSet}$
		(i.e., we employ an 80/20 split between training and validation).
	\begin{itemize}
	\item For~\ref{feat:params_rpca},~\ref{feat:params_rgpca},
		and~\ref{feat:params_pr}, additionally extract the residual training set
$\residualTrainSet$
			from $\parameterdomainTrain$ and $\timeIndexSetTrain$ according to
			\eqref{eq:residualTrainSet}.
	\end{itemize}
\item \label{train:normalizedata} Standardize features and responses according to their respective means and variances on the training set $\trainSet.$ 
\item \label{step:defineGrid}Define the hyperparameter grid for each considered regression
	model.
\item \label{train:train} For each value of the hyperparameters, train the
	regression model by solving optimization problem \eqref{eq:nonrecursiveOpt},
		\eqref{eq:recursiveLatentPredOpt}, or \eqref{eq:BPTTopt}, depending on the
		category of the method.  For the ARX, LARX, ANN, ANN-I, RNN, and LSTM
		models, we employ the Adam algorithm~\cite{Kingma2015AdamAM} with early
		stopping (assessed on a 20\% holdout set) for this purpose. For regression methods employing 
		the recursive training (RT) approach, as described in Section~\ref{sec:regressionTraining}, BPTT 
		is used to solve the optimization problem. Within the
		optimization problems, we employ ridge regularization such that
		$\regTermf(\cdot)\equiv\alpha\|\cdot\|_2^2$ and
		$\regTermg(\cdot)\equiv\alpha\|\cdot\|_2^2$, where $\alpha\in\RR{}_+$ is a
		regularization hyperparameter.  As the Adam algorithm is stochastic, each
		model is trained 20 times. 
\item \label{train:gridsearch} For each regression method, select the
	model---which corresponds to one value of the hyperparameters and one
		execution of the Adam algorithm---by evaluating validation criterion
		\eqref{eq:valOne} or \eqref{eq:valTwo}, depending on the category of the method.
	\item \label{train:noiseextract} 
		Define the noise training set $\trainSetNoise$ and test set
		$\testSetNoise$ according to \eqref{eq:trainSetNoise}
	with	
$\timeIndexSetTrainNoise=\timeIndexSetTestNoise=\timeIndexSet$,
$\paramtrainsetNoise$ drawn randomly from 
	$\parameterdomainTest$, and 
		$\paramtestsetNoise= \parameterdomainTest\setminus\paramtrainsetNoise$.
\item \label{train:noisetrain} Train stochastic noise models corresponding to
	a stationary Gaussian, stationary Laplacian, or AR1
		via \eqref{eq:trainStatGauss}, \eqref{eq:trainStatLaplace}, and \eqref{eq:trainAR1}, respectively.
\item Evaluate the trained noise models using statistical validation criteria
	(e.g., empirical prediction intervals, Komolgorov--Smirnov test) on the test
		set $\testSetNoise$.
\end{enumerate}
As discussed in Section \ref{sec:regressionTraining}, the validation procedure
described in steps \ref{train:train}--\ref{train:gridsearch} could be replaced with a
cross-validation process.  

\subsubsection{Hyperparameter grid and optimization}
We now describe the specific hyperparameter grid employed by each
considered regression method as defined in step \ref{step:defineGrid}.
\begin{itemize}
\item \textbf{$k$-nearest neighbors (kNN)}: The hyperparameters for kNN
	correspond to 
	the number of neighbors $k$, and the definition of the weights $w$. We
		employ the
		hyperparameter grid 
		$$
		(k,\text{weighting function}) \in \{1,\ldots,5\}\times\{\text{uniform,
		Euclidean distance}\}.
		$$
\item \textbf{Artificial neural network (ANN)}: The hyperparameters for ANN
	correspond to the network depth $\nlayers$ and the number of neurons per
		layer $p$, which we set to be constant such that
		$p_i=p$,
$i=1,\ldots,\nlayers$. We also consider the regularization parameter 
		 $\alpha$ that appears in 
		ridge regularization terms $\regTermf(\cdot)\equiv\alpha\|\cdot\|_2^2$ and
		$\regTermg(\cdot)\equiv\alpha\|\cdot\|_2^2$ 
		 as a hyperparameter. We employ the hyperparameter grid
		 $$
( \nlayers ,p , \alpha)\in\{1,2\}\times
\{10,25,50,100\}\times 		\{10^{-i}\}_{i=1}^5.
		 $$
		Listing~\ref{listing:NN} provides the 
		Python code for
the Keras implementation of this method.
\item \textbf{Autoregressive with exogenous inputs (ARX)}:  Because we
	consider only ARX(1,1) methods, the only hyperparameter for this technique
		is the regularization parameter
		$\alpha$; we employ the hyperparameter grid
$$
\alpha \in \{10^{-i}\}_{i=1}^5.
$$
		We construct ARX models using both the NRT and RT approaches (see Section
		\ref{sec:regressionTraining}). We 
refer to the resulting models as ARX (NRT) and ARX (RT). Listing~\ref{listing:ARX} provides the Python code for the Keras implementation of ARX (RT).
\item \textbf{Integrated artificial neural network (ANN-I)}: For ANN-I, we
	employ the same hyperparameter grid as for ANN.
		As with ARX, we construct two different ANN-I models: one with NRT and one with RT. We 
refer to the resulting models as ANN-I (NRT) and ANN-I (RT). Listing~\ref{listing:NNI} provides the Python code for the Keras implementation of ANN-I (RT).
\item \textbf{Latent autoregressive with exogenous inputs (LARX)}:
	The hyperparameter for LARX is the number of latent
		variables $\nlatentf$. We also consider the regularization term $\alpha$
		and employ the hyperparameter grid
		$$
( \nlatentf , \alpha )\in\{
	10,25,50,100\}\times\{10^{-i}\}_{i=1}^5.
		$$
		Listing~\ref{listing:LARX} provides the Python code for the Keras implementation of LARX.
\item \textbf{Recurrent Neural Network (RNN)}: For RNN, we employ the same
	hyperparameter grid as for ANN.  Listing~\ref{listing:RNN} reports the
	Python code used to develop the RNNs used in this work; we employ
	Keras' \verb+SimpleRNN+ cell. 
\item \textbf{Long short-term memory Network (LSTM)}: 
	For LSTM, we employ the same
	hyper-parameter grid as for ANN. 
	Listing~\ref{listing:LSTM} gives the Python code used to develop the LSTM
	networks in this work; we employ Keras' \verb+LSTM+ cell. 
\item \textbf{GP Regression}: GP regression is characterized by one hyperparameter: the
	noise magnitude $\gpNoise$. We employ the hyperparameter grid 
$$ \gpNoise \in \{ 10^{-8 +
		\frac{8(i-1)}{19}}\}_{i=1}^{20}.$$ 
\end{itemize}
Additionally, for feature methods employing the principal components of the
residual or residual samples, the number of principal components $\npca$ and
number of residual samples $\nsample$ can be viewed as hyperparameters. Here,
we employ $\npca=\nsample$, where $\npca$ is determined such that the retained
singular values associated with the truncated residual principal components
contain $99\%$ of the total statistical energy (see, e.g., 
Ref.~\cite[Appendix A]{carlberg_lspg_v_galerkin}). 

\subsubsection{Evaluation metrics}
To evaluate the performance of the deterministic regression-function models,
we employ the fraction of variance unexplained (FVU) computed on the test set
$\testSet$, which is given by
$$\text{FVU} \defeq 1 - r^2,$$
where the coefficient of determination is defined as 
$$r^2 \defeq 1 - \frac{ \sum_{ \tindx \in \timeIndexSetTest  , \parametervec \in \parameterdomainTest }
	(\errorApproxMeanGennmunArg{\tindx} - 
	 \errorGennmuArg{\tindx})^2 }
	{\sum_{ \tindx \in \timeIndexSetTest ,\parametervec \in \parameterdomainTest }\big( 
          \errorGennmuArg{\tindx} - \meanErrorGen \big)^2},$$
where the mean response is
$\meanErrorGen \defeq {1}/{\card{\testSet}} \sum_{ \tindx \in \timeIndexSetTest, \parametervec \in \parameterdomainTest
 } \errorGennmuArg{\tindx} .$

%
\subsection{Example 1: 
Advection--diffusion equation with
projection-based reduced-order model}
The first example considers 
the advection--diffusion equation and approximate solutions generated by a 
POD--Galerkin reduced-order model.
The governing PDE is 
\begin{equation}\label{eq:ad}
\frac{\partial }{\partial t}u(\spatialvar ,t) 
+\parameter_1 \frac{\partial
	}{\partial \spatialvar }u(\spatialvar ,t) 
	=   \parameter_2 \frac{\partial^2
	}{\partial \spatialvar ^2}u(\spatialvar ,t), \hspace{0.1 in} u(0,t) = u(2,t)
	= 0, \hspace{0.1 in} u(\spatialvar ,0) = \spatialvar (2-\spatialvar
	)\exp(2\spatialvar),
\end{equation}
on the domain $\spatialvar  \in \physDomain =  [0,2]$ with $t \in [0,T]$  and $T=0.3$ and $u:
\physDomain \times [0,T] \rightarrow \RR{}$. 
We employ a parameter domain of
$(\parameter_1,\parameter_2)\in\parameterdomain = [-2,-0.1]\times[0.1,1]$, where
the parameter $\parameter_1$ is the wave speed and the parameter
$\parameter_2$ is the
diffusion coefficient.

To derive a parameterized dynamical system of the form \eqref{eq:fom} for the
FOM, we apply a finite-difference spatial discretization to the governing
PDE \eqref{eq:ad}, which uses
central differencing for the diffusion term and upwind differencing for the
advection term. This discretization is obtained by partitioning the spatial
domain into 101 cells of equal width, yielding a FOM with
$\nfom=100$ degrees of freedom of the form 
$$\frac{d x_k}{d t} + \parameter_1 \frac{ x_{k+1} - x_k}{\Delta
\spatialvar } = \parameter_2
\frac{ x_{k+1} - 2 x_k + x_{k-1} }{\Delta \spatialvar ^2},
\quad k = 1,\ldots,\nfom,$$
where the state vector is
$\state(t,\parametervec)\equiv[x_1(t,\parametervec)\ \cdots\
x_\nstate(t,\parametervec)]^T$ and $x_k(t,\parametervec)$ is the finite-difference approximation to
$u(\spatialvar_k,t,\parametervec)$ with $\spatialvar_k\defeq 2k/(\nfom+1)$
the $k$th point in the
finite-difference discretization. 
The initial condition is characterized by
$\statevec_0(\parametervec)  \equiv \statevec_0 =  [
\spatialvar_1 (2-\spatialvar_1
	)\exp(2\spatialvar_1)\ \cdots\ 
\spatialvar_\nfom (2-\spatialvar_\nfom
	)\exp(2\spatialvar_\nfom)
	]^T$.

To define the FOM O$\Delta$E of the form \eqref{eq:fom_discretized}, we employ
a time step 
$\Delta t = 3\times 10^{-4}$ with the
implicit Crank--Nicolson scheme, which is characterized by multistep
coefficients $\alpha_0 = 1$, $\alpha_1 = -1$, $\beta_0=\beta_1 = 1/2$ in
\eqref{eq:FOMODeltaERes}, which yields $\ntimesteps = 10^3$ time instances. 

To construct approximate solutions, we employ projection-based reduced-order
models as described in Section \ref{sec:romsec}.  
For this purpose, we
construct the trial-basis matrix $\basismat\in\RRstar{N\times \dimrom}$ using
proper orthogonal decomposition (POD) by executing Algorithm \ref{alg:pod} in
Appendix \ref{appendix:pod} with inputs  $\parameterdomainROM =
\{0.1,1.05,2.0\} \times \{0.1,0.55,1.0\}$, $\nskip= 10$, reference state 
$\statevecIntercept(\param) = \statevec_0(\param)$,
and $\dimrom= 5$.
For the ROM, we employ Galerkin projection such that the ROM ODE corresponds
to \eqref{eq:ROM} with $\testmat = \basismat$.  We employ the same time
discretization for the ROM as was employed for the FOM, i.e., the
Crank--Nicolson scheme with time step $\Delta t = 3\times 10^{-4}$.

We model both the QoI error $\qoiErrorArg{\tindx}(\parametervec)$,
where the QoI corresponds to the solution at the midpoint of the domain such
that the associated functional is 
$
\qoiFunc(\statevec;t,\parametervec)\mapsto \kroneckerVecArg{51}^T\statevec
$ with $\kroneckerVecArg{i}$ the $i$th canonical unit vector,
and the normed state error $\stateErrorArg{\tindx}(\parametervec)$.

\subsubsection{Coarse time grid, training, validation, and test sets}
We now describe how we execute Steps
\ref{step:timeIndexSet}--\ref{step:trainValTest} and \ref{train:noiseextract} described in Section
\ref{sec:trainingModelSel}. For Step \ref{step:timeIndexSet},
we set $\timeIndexSet = \{20n\}_{n=1}^{50}$; the coarse grid is 
selected such that the error responses on the training set are well represented. 
For Step \ref{step:trainValTest}, we employ set sizes of
$\card{\parameterdomainTrain} = 40$, 
$\card{\parameterdomainVal}=10$, and
$\card{\parameterdomainTest} = 50$.
To assess the impact of the training set size on the performance of the
deterministic regression-function model,
we consider four smaller training and validation sets
with $(\card{\parameterdomainTrain},\card{\parameterdomainVal})
\in\{(8,2),(16,4), (24,6), (32,8)\}$, which are constructed by sampling the
original sets $\parameterdomainTrain$ and $\parameterdomainVal$.
In Step \ref{train:noiseextract}, we set
$\card{\paramtrainsetNoise}=20$.

\subsubsection{Regression results}

Figures~\ref{fig:ad_converge}--\ref{fig:ad_hist} summarize the performance of the 
considered error models.
First, Figure~\ref{fig:ad_converge} reports the dependence of the performance of each 
deterministic regression-function model (with its best performing
feature-engineering method) on the size of the training set
$\card{\parameterdomainTrain}$. This figure illustrates several trends. 
First, we observe that traditional time-series-modeling methods,
i.e., ARX and LARX, nearly always yield the worst performance.
Here, ARX yields the worst performance, likely due to the fact that it is a
low-capacity model and contains only a
single latent variable. 
LARX outperforms ARX; this is likely
due to its inclusion of additional latent variables.
We observe that the performance of ANN-I is slightly better than ARX or
LARX, but not as good as the (non-integrated) ANN. This indicates that, in this case, 
differencing the data provides no particular benefit. 
We also observe that
recursive-neural-network models, i.e., RNN and LSTM, yield the best performance nearly uniformly, and generate FVU values
at least one order of magnitude smaller than both traditional
time-series-modeling methods. This highlights
that there is substantial benefit to modeling nonlinear latent dynamics as is
done with RNN and LSTM models. Of these two, LSTM tends to yield best performance in most cases. 
 Finally, LSTM and RNN outperform the GP
 regression-function model. This is likely due to the fact that the time-local
 GP method employs only the parameters $\parametervec$ as features, which are low
 quality compared to the residual-based quantities employed by the
 \methodAcronym\ methods.

Figure~\ref{fig:ad_train_hist} shows a histogram that, for regression methods
in~\ref{cat:recursiveLatentPred}, compares the ratio of the fraction of
variance unexplained using RT for training to fraction of variance unexplained
using NRT for training. This figure shows that RT training generally
outperforms NRT training; this is likely due to the fact that RT training
properly accounts for the latent dynamics while training the model. For ARX,
regression functions trained through RT outperform their NRT counterparts
90.8\% of the time; for ANN-I the regression functions trained through RT
outperform their NRT counterparts 72.3\% of the time.   

Next, Table~\ref{tab:ad_summary} summarizes the percentage of cases that
a given deterministic regression-function model produces results with the lowest
FVU; each case is defined by one of the 13 feature-engineering methods
and one of the 5 
training-set sizes, yielding 65 total cases. We do not include the GP method in these results, as it 
employs only the parameters as 
features.
Results are presented for the cases where (1) all feature methods are
considered and (2) only feature methods that don't include time are
considered. The table shows that LSTM is the best
overall performing regression-function model, leading to the lowest FVUs
approximately 60\% of the time. RNN is the next best performing
regression-function model, followed by ANN. For the normed state error
$\stateError$, recursive regression methods comprise the best performing
regression-function models in $87.7\%$ of cases when all feature-engineering
methods are considered, and in $97.1 \%$ of cases when only
feature-engineering
engineering methods that omit time are considered; this shows that including
time can improve the relative performance of non-recursive regression methods.
For the QoI error $\qoiError$, recurrent architectures comprise the best
performing methods in over 97\% of cases whether or not time is considered as
a feature.

\begin{figure}
\begin{center}
\begin{subfigure}[t]{0.49\textwidth}
\includegraphics[width=1.\linewidth]{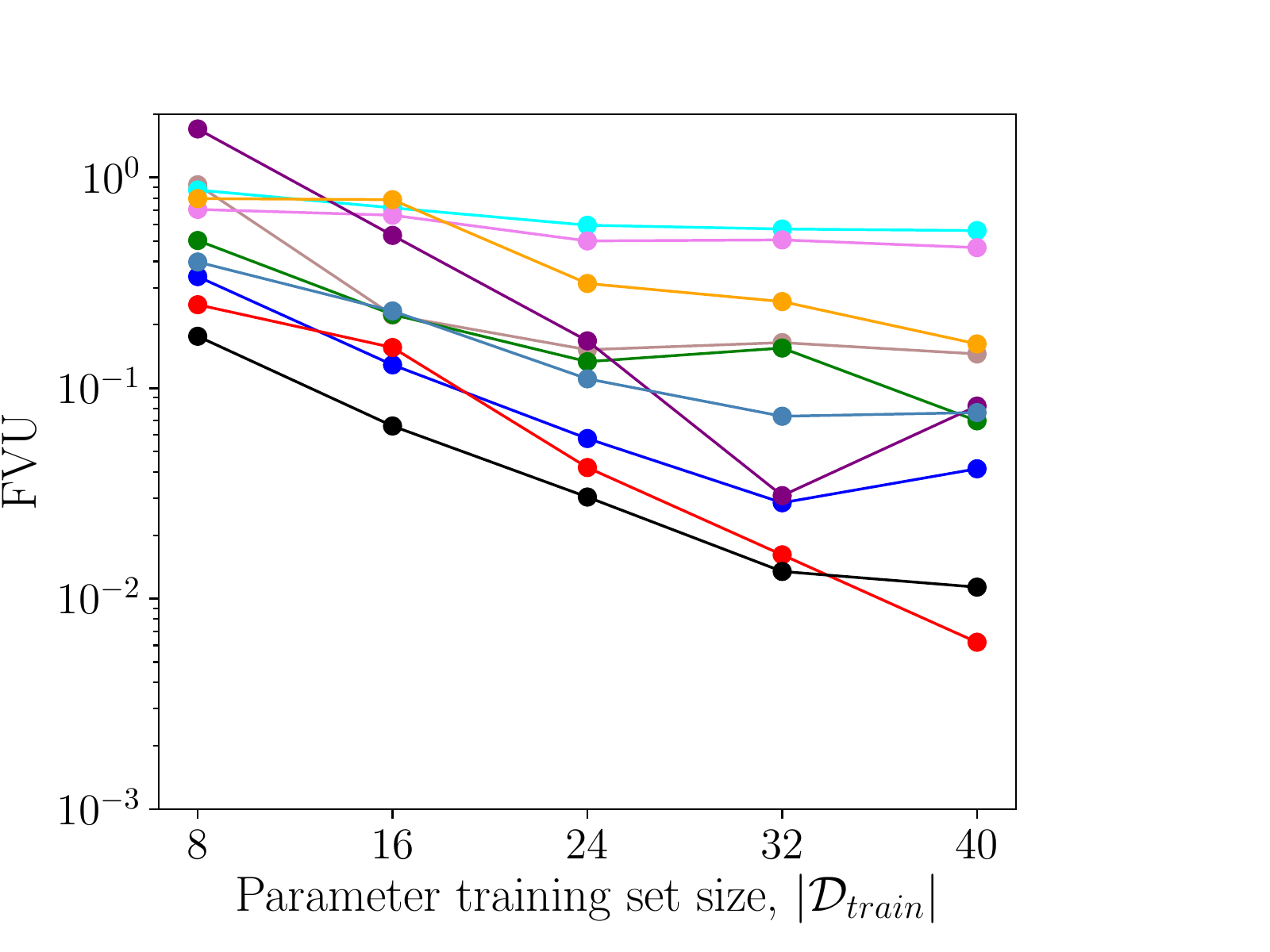}
\caption{Normed state error $\stateError$}
\end{subfigure}
\begin{subfigure}[t]{0.49\textwidth}
\includegraphics[width=1.\linewidth]{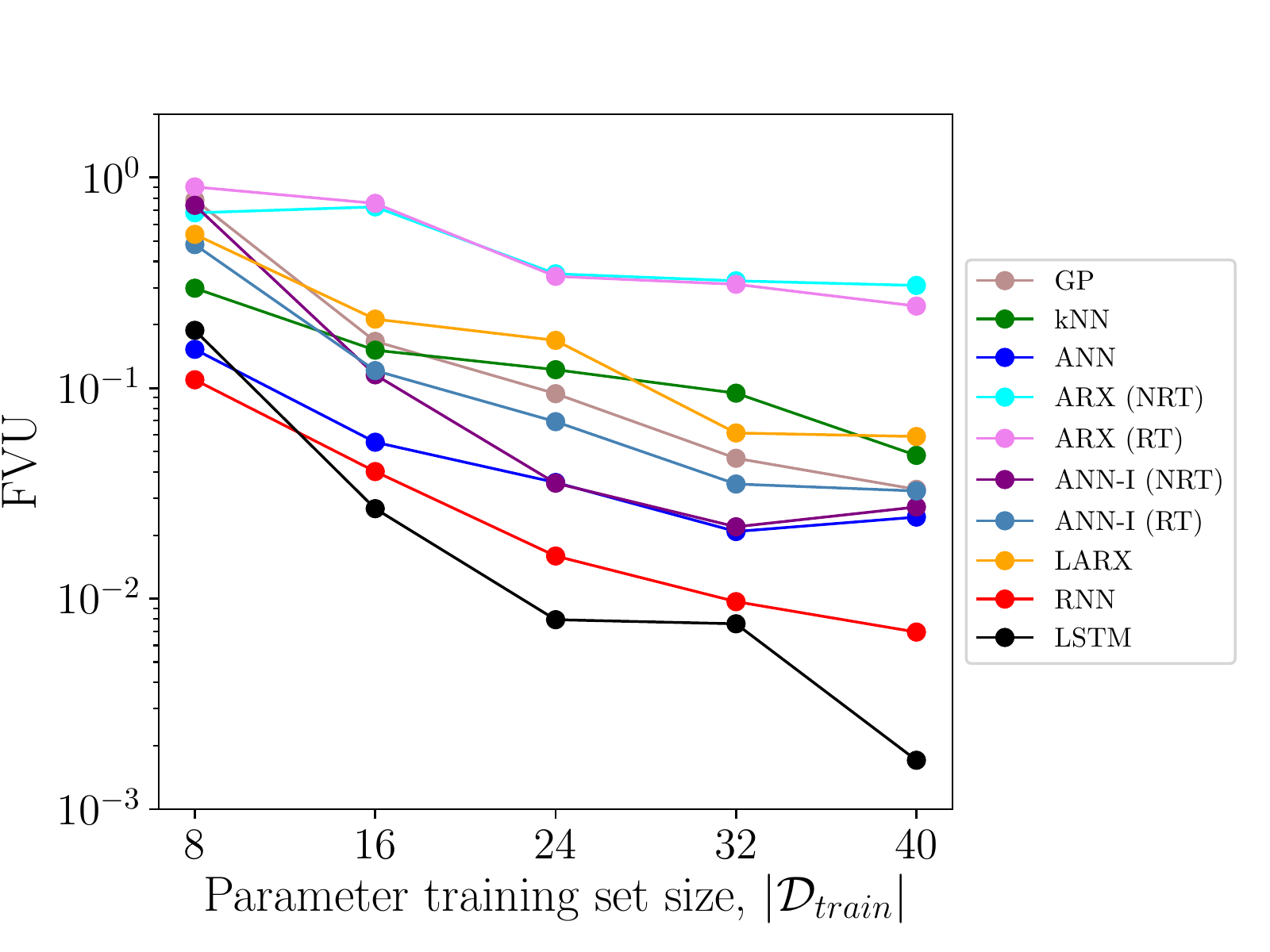}
\caption{QoI error $\qoiError$}
\end{subfigure}
\end{center}
	\caption{\textit{Advection--diffusion equation.} Fraction of variance
	unexplained for all considered deterministic regression-function models,
	with each employing their best respective feature-engineering method. }
\label{fig:ad_converge}
\end{figure}

\begin{figure}
\begin{center}
\begin{subfigure}[t]{0.49\textwidth}
\includegraphics[width=1.\linewidth]{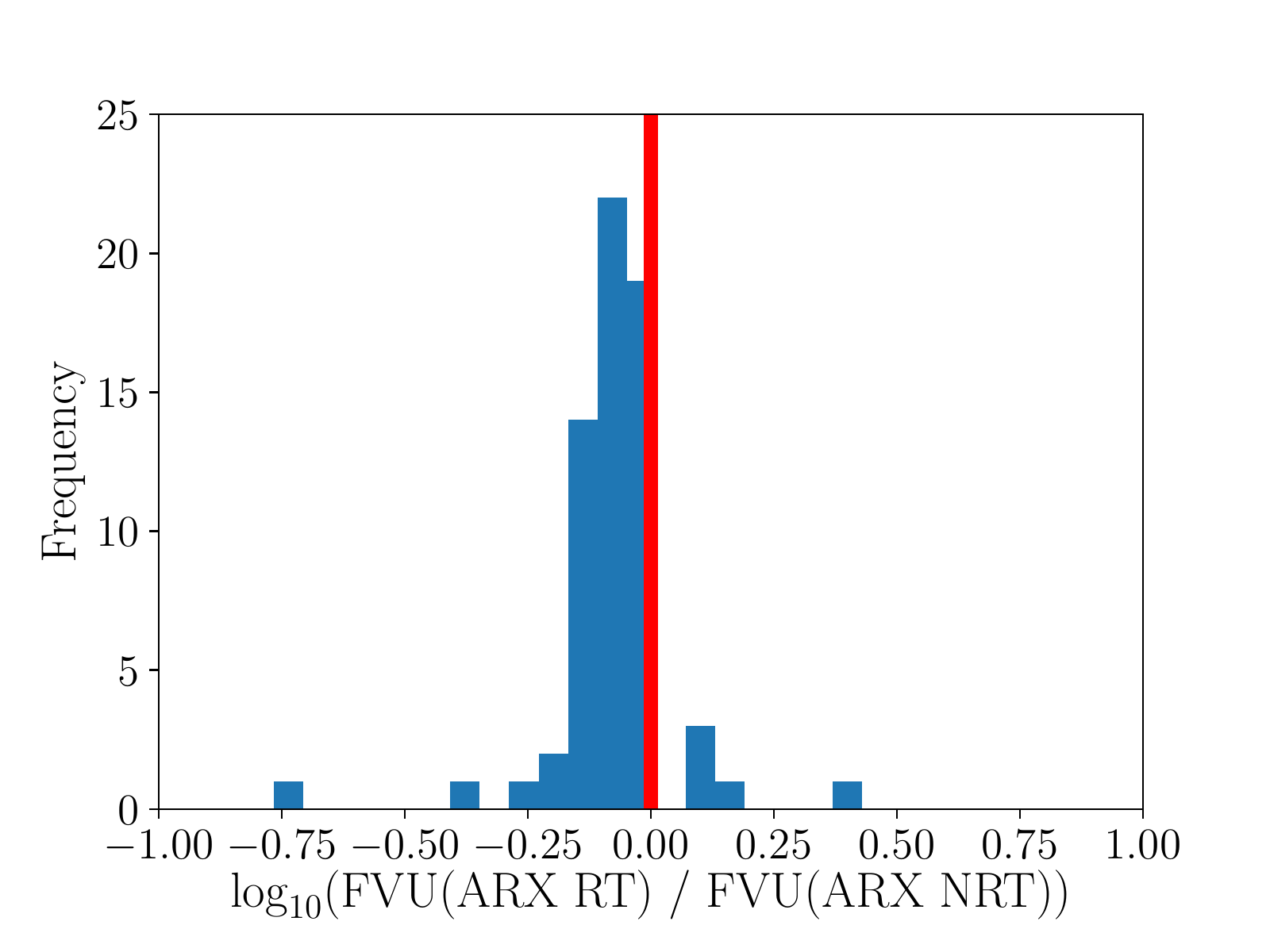}
\caption{Normed state error $\stateError$}
\end{subfigure}
\begin{subfigure}[t]{0.49\textwidth}
\includegraphics[width=1.\linewidth]{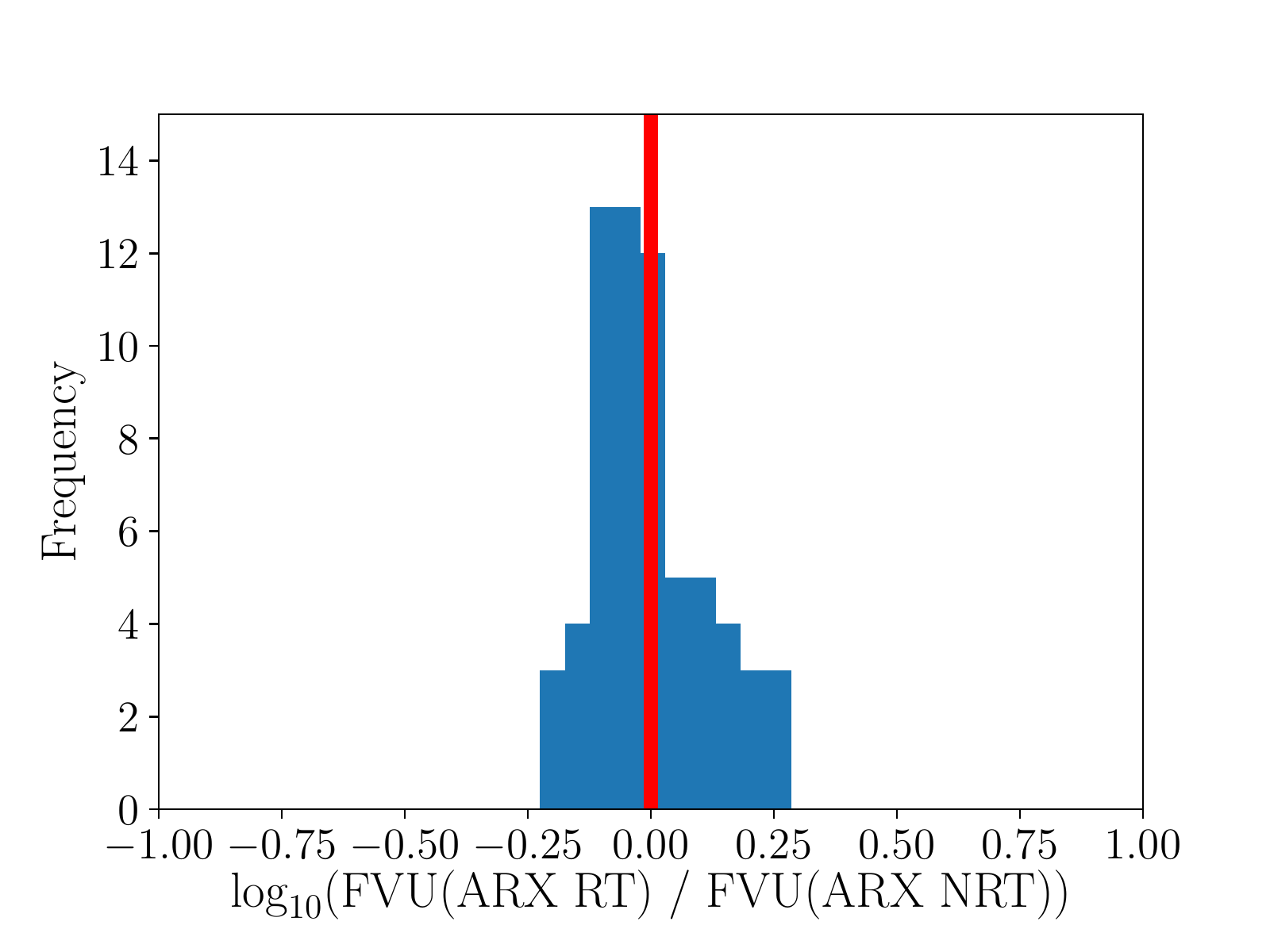}
\caption{QoI error $\qoiError$}
\end{subfigure}

\begin{subfigure}[t]{0.49\textwidth}
\includegraphics[width=1.\linewidth]{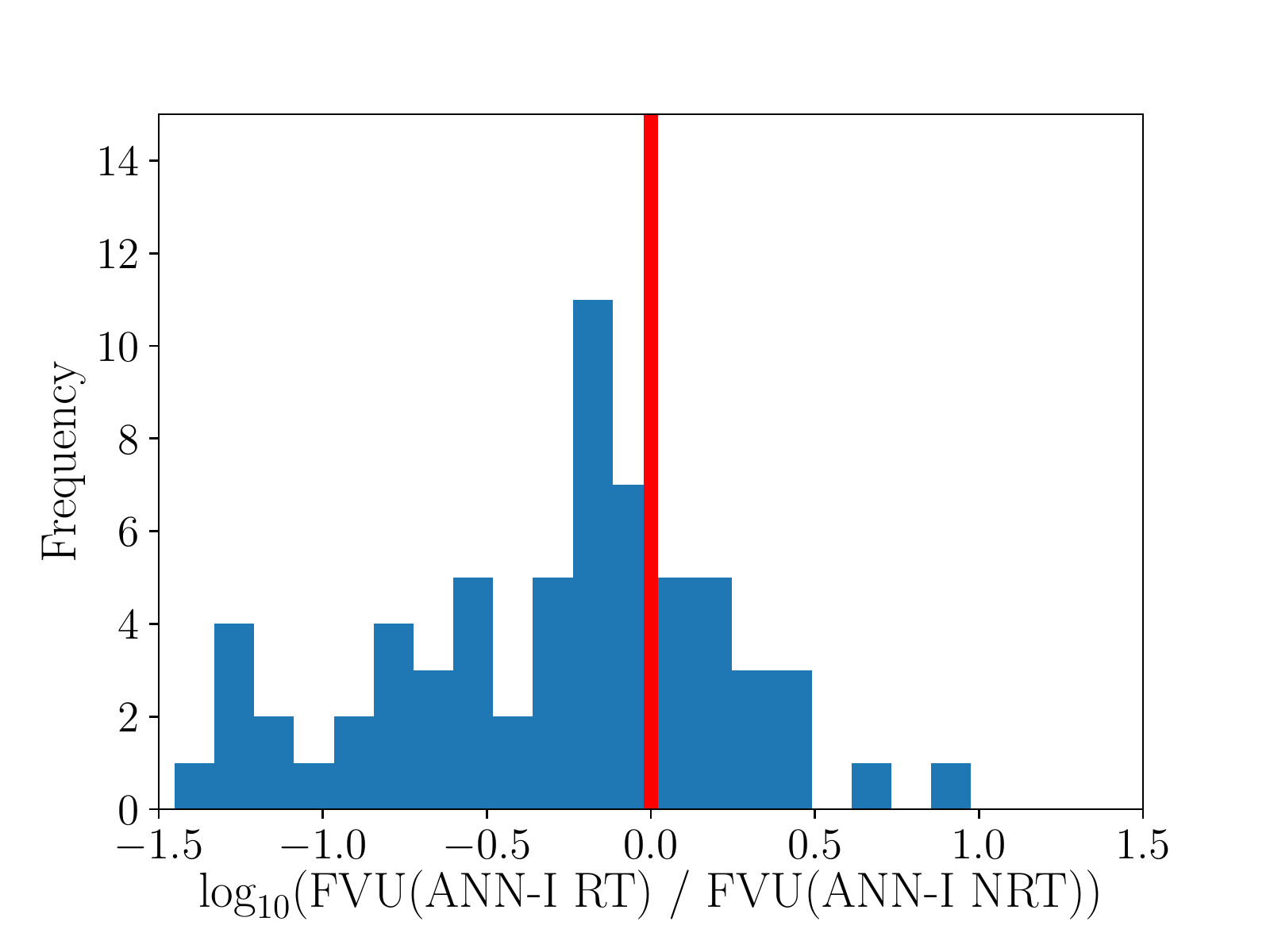}
\caption{Normed state error $\stateError$}
\end{subfigure}
\begin{subfigure}[t]{0.49\textwidth}
\includegraphics[width=1.\linewidth]{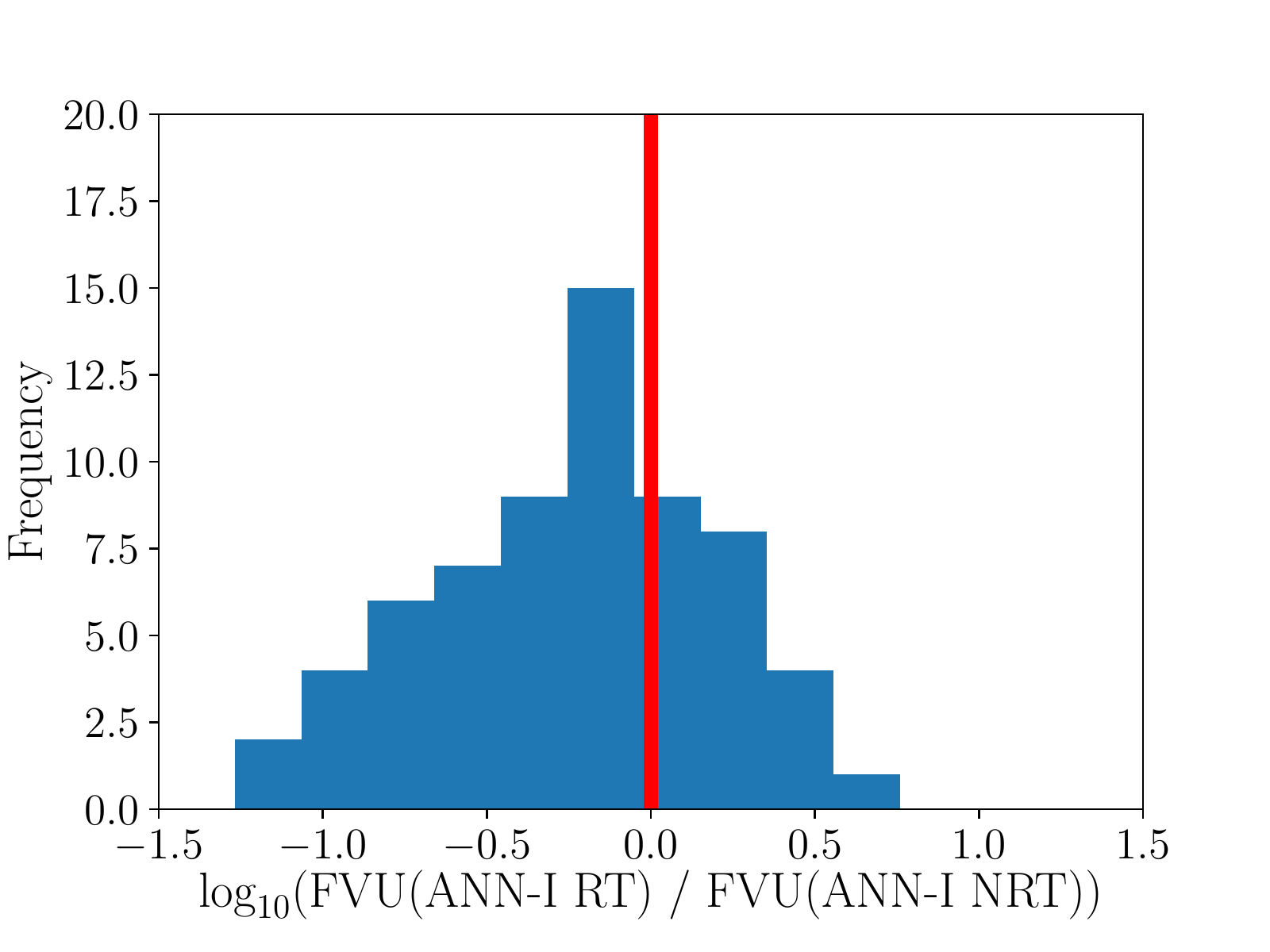}
\caption{QoI error $\qoiError$}
\end{subfigure}
\end{center}
	\caption{\textit{Advection--diffusion equation.} Histogram of the ratio of the fraction of variance unexplained using RT for training to 
fraction of variance unexplained using NRT for training. Results are shown for ARX (top) and ANN-I (bottom). For ARX, RT outperforms NRT 90.8\% of the time. For 
ANN-I, RT outperforms NRT 72.3\% of the time.}
\label{fig:ad_train_hist}
\end{figure}

\begin{table}[ht]
\begin{center}
\small{
\setlength{\tabcolsep}{4pt}
\begin{tabular}{|c| c | c c c c c c c c c | c|}
	\hline
error	&
\begin{tabular}{c }
	include feat.\\
	eng.\ w/ time?
\end{tabular}
	&kNN & ANN &  \begin{tabular}{c} ARX \\ (NRT) \end{tabular} & \begin{tabular}{c} ARX \\ (RT) \end{tabular} &  \begin{tabular}{c} ANN-I \\ (NRT) \end{tabular} & \begin{tabular}{c} ANN-I \\ (RT)\end{tabular} & LARX &  RNN & LSTM & 
	\begin{tabular}{c}	
	Recursive\\ total
	\end{tabular}\\ \hline
		$\stateError$ &Yes & 0.0 & 12.3 &0.0 & 0.0 & 1.5  & 0.0   & 0.0 & 27.7 & 58.5 & 87.7\\  
		$\stateError$ &No  & 0.0 & 2.9 &0.0 & 0.0 & 2.9 & 0.0 & 0.0   & 31.4 & 60.0  & 97.1\\ 
		$\qoiError$   &Yes & 0.0 & 1.5 &0.0 & 0.0 & 0.0 & 1.5  & 1.5  & 32.3 & 63.1  & 98.5 \\ 
		$\qoiError$   &No  & 0.0 & 2.9 &0.0 & 0.0 & 0.0 & 2.9 & 2.9  & 40.0& 51.4  & 97.1 \\ \hline
\end{tabular}
}
\end{center}
\caption{\textit{Advection--diffusion equation.} Percentage of cases having
	the lowest predicted FVU for each regression method. Results are summarized
	over all values $\card{\paramtrainindex} \in\{ 8,16,24,32,40\}$ and all
	feature-engineering methods. Recursive total reports the sum of all
	regression methods
	in \ref{cat:recursiveLatentPred} and 
	\ref{cat:recursive}.}
\label{tab:ad_summary}
\end{table}
 
Figure~\ref{fig:ad_grid} summarizes the performance of each combination of
regression method and feature-engineering method for the case
$\card{\parameterdomainTrain} = 40$. These results clearly show two critical
trends. First, as hypothesized, using residual-based features can improve
performance significantly; this is why they were considered high-quality
features.  This trend was also observed in Ref.~\cite{freno_carlberg_ml}.
Further, we note that---in this case---excellent results were obtained by
computing only $7(\ll\nfom)$ samples of the residual.  Second, as
hypothesized, the recurrent regression methods yield the best performance (even
for low-quality features), with LSTM performing best of all. In fact, LSTM and
RNN outperform the time-local GP by approximately one order of
magnitude. We also observe that including time slightly improves performance
for these methods. Finally, we observe that for the deterministic regression-function 
models in ~\ref{cat:recursiveLatentPred}, both training methods yield 
comparable results.

\begin{figure}
\begin{center}
\begin{subfigure}[t]{0.49\textwidth}
\includegraphics[width=1.\linewidth]{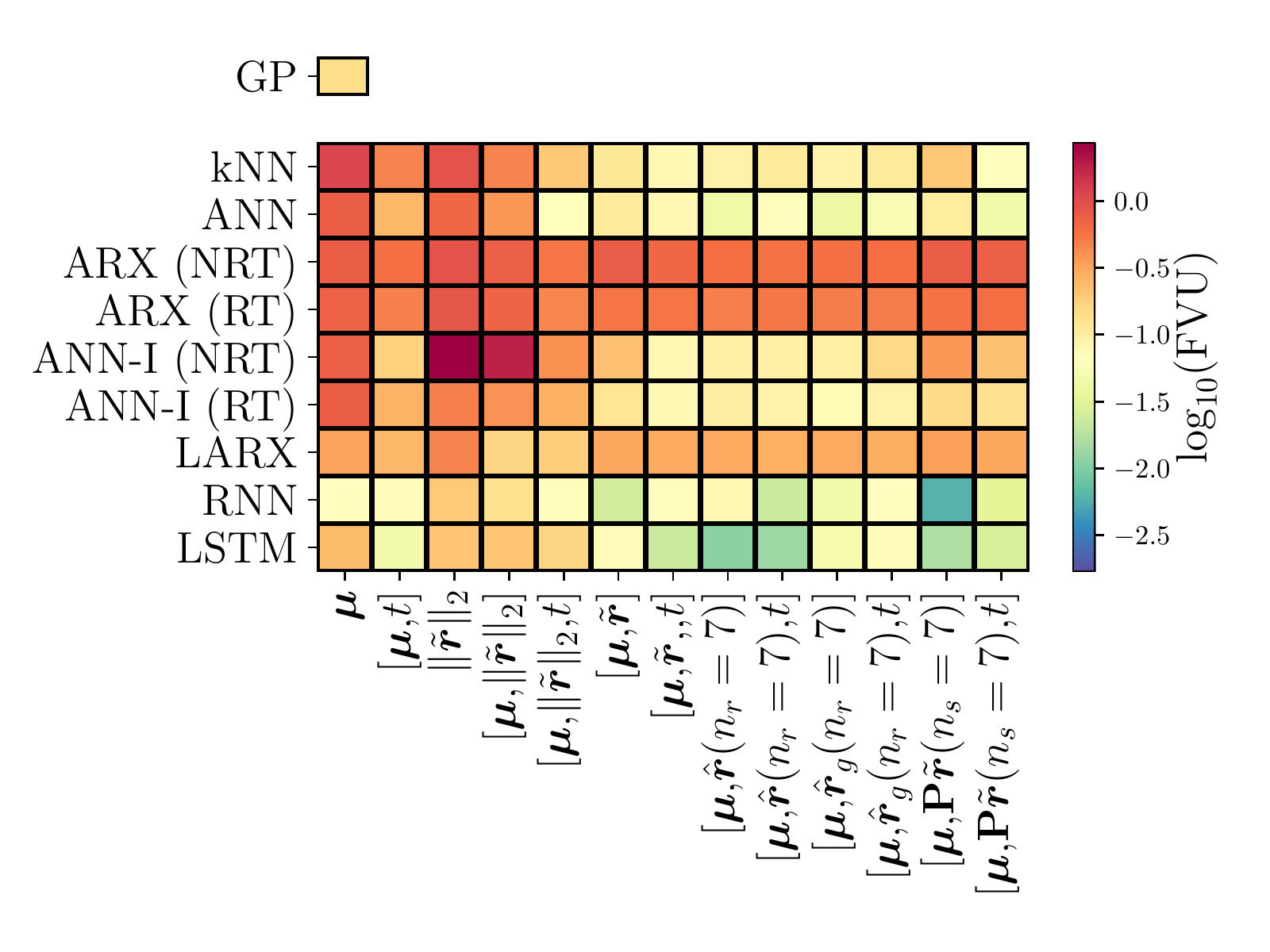}
\caption{Normed state error $\stateError$}
\end{subfigure}
\begin{subfigure}[t]{0.49\textwidth}
\includegraphics[width=1.\linewidth]{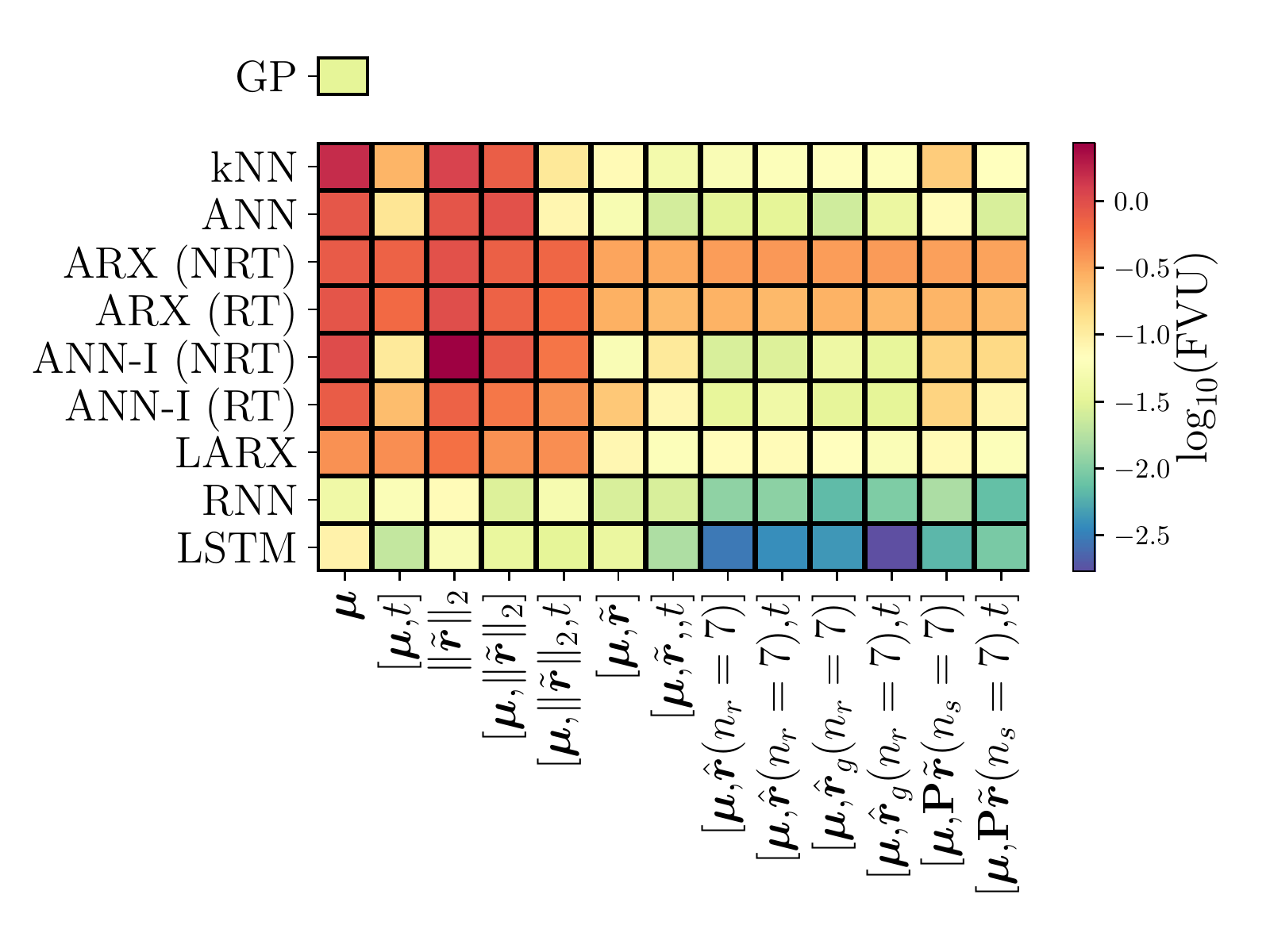}
\caption{QoI error $\qoiError$}
\end{subfigure}
\end{center}
\caption{\textit{Advection--diffusion equation.} Summary of the performance of
	each combination of regression method and feature-engineering method for the
	case $\card{\paramtrainindex} = 40$. }
\label{fig:ad_grid}
\end{figure}

 
Figure~\ref{fig:ad_yhat_vs_t} shows the predictions made by the different
considered regression methods (for their best performing respective
feature-engineering method) as a function of time for a randomly drawn element
of $\parameterdomainTest$ for cases $\card{\paramtrainset} = 8$ and
$\card{\paramtrainset} = 40$. First, we note that the case
$\card{\paramtrainset}=8$ yields significant errors for all
considered models, with LSTM yielding the best performance.
For the case
$\card{\paramtrainset}=40$, all regression methods display substantially improved 
performance; LSTM, RNN, and ANN yield particularly accurate results. 


\begin{figure}
\begin{center}
\begin{subfigure}[t]{0.49\textwidth}
\includegraphics[width=1.\linewidth]{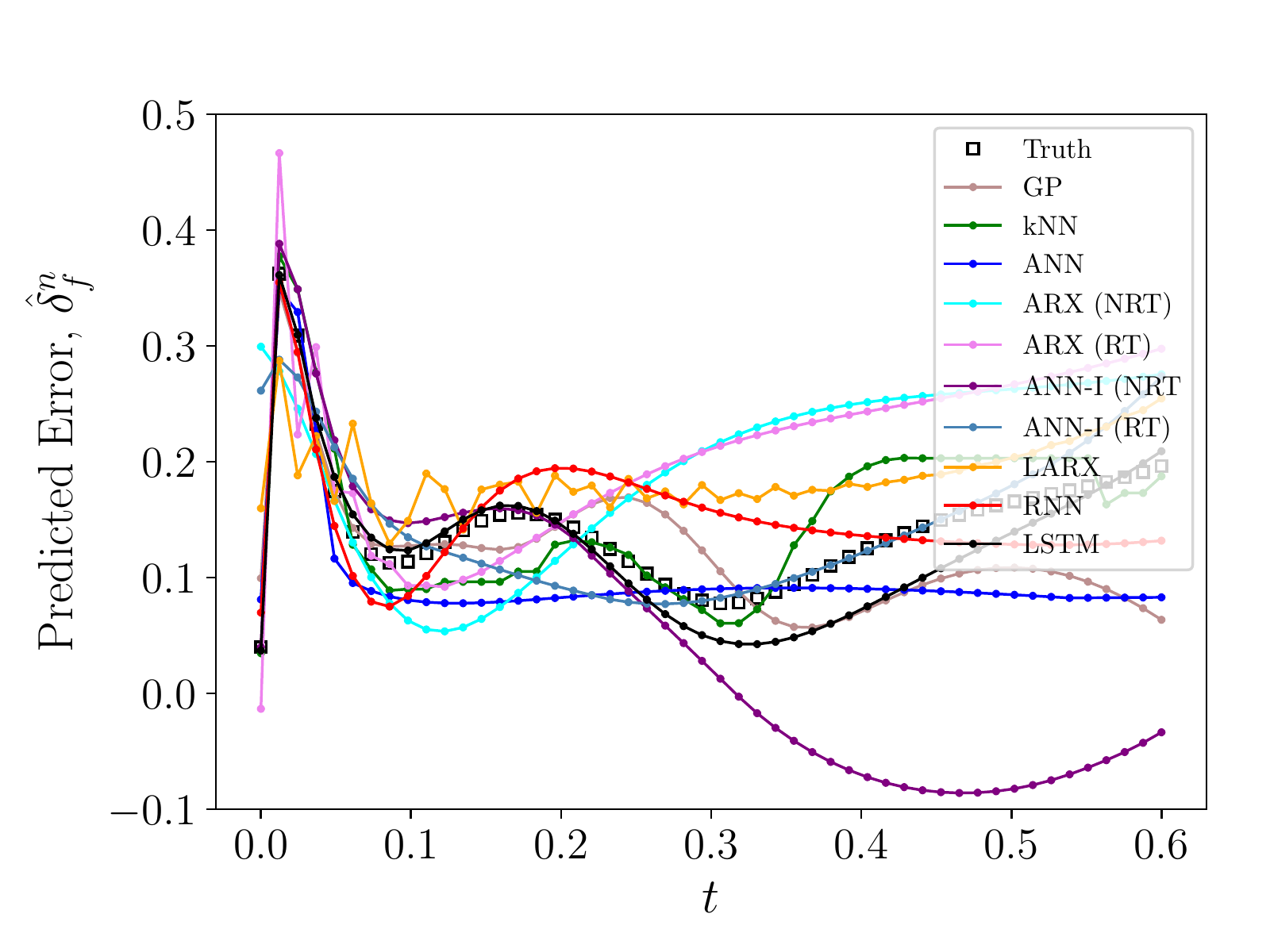}
	\caption{Normed state error $\stateError$, $\card{\paramtrainindex}=8$}
\end{subfigure}
\begin{subfigure}[t]{0.49\textwidth}
\includegraphics[width=1.\linewidth]{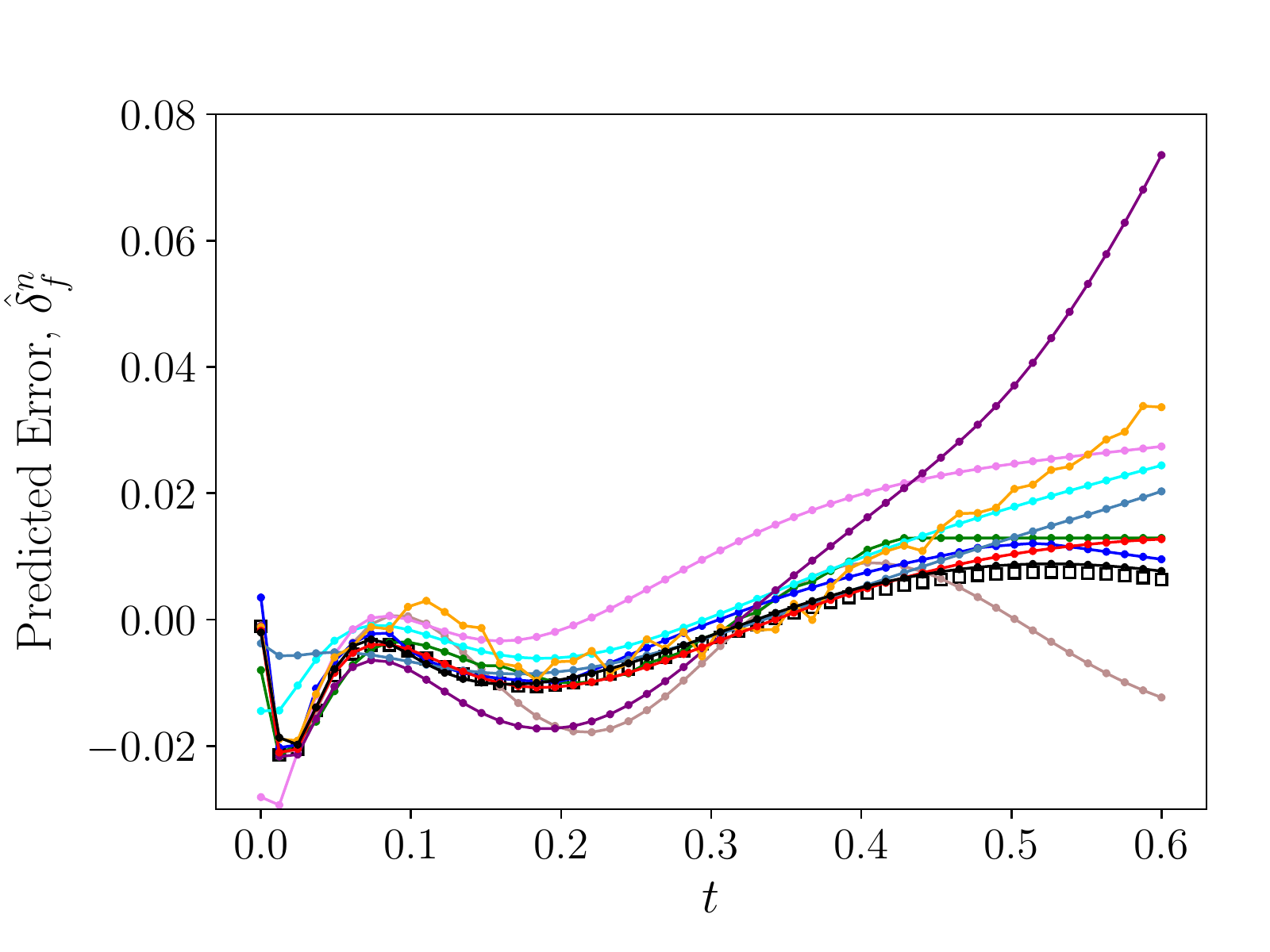}
	\caption{QoI error $\qoiError$, $\card{\paramtrainindex}=8$}
\end{subfigure}
\begin{subfigure}[t]{0.49\textwidth}
\includegraphics[width=1.\linewidth]{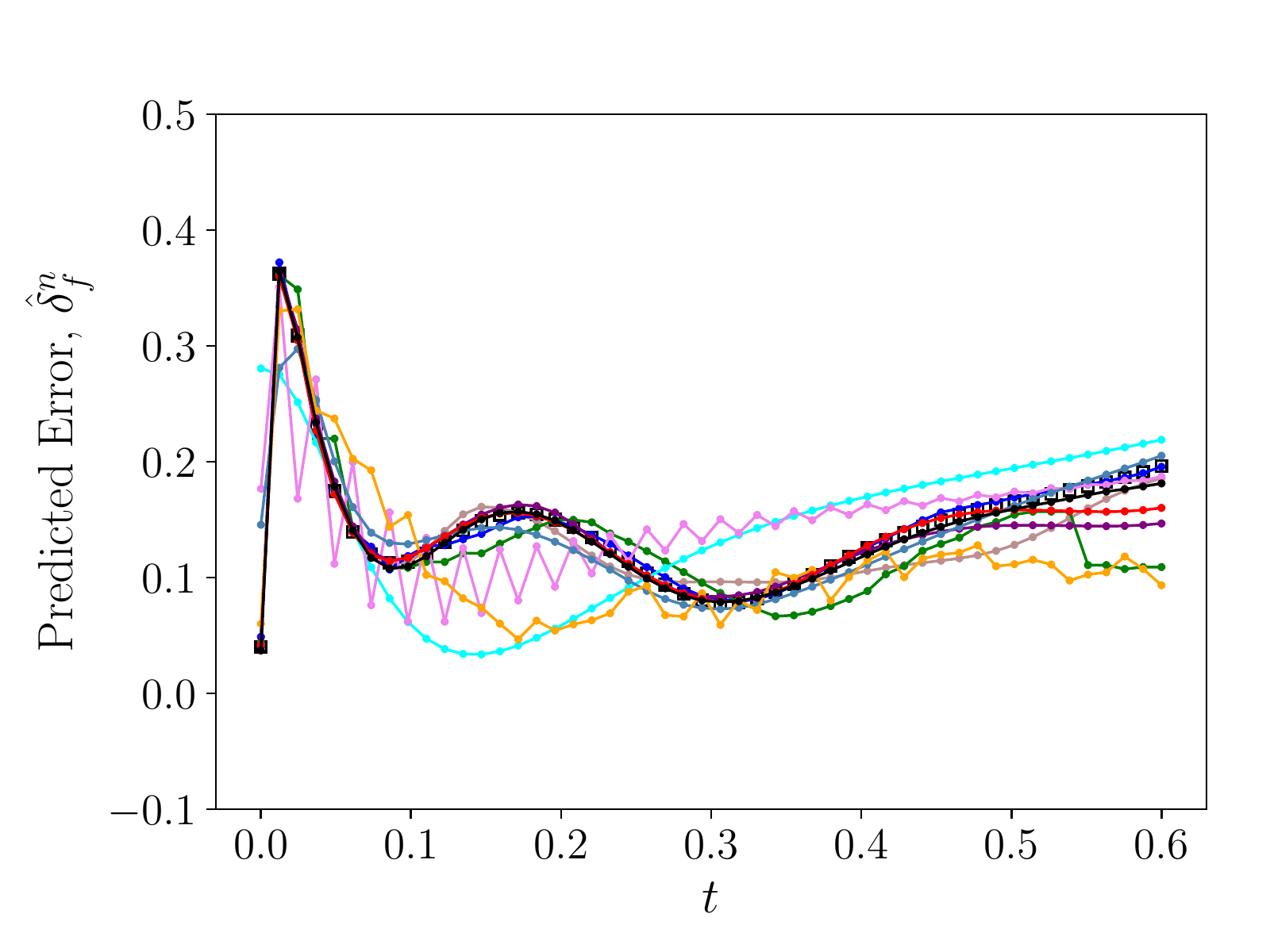}
	\caption{Normed state error $\stateError$, $\card{\paramtrainindex}=40$}
\end{subfigure}
\begin{subfigure}[t]{0.49\textwidth}
\includegraphics[width=1.\linewidth]{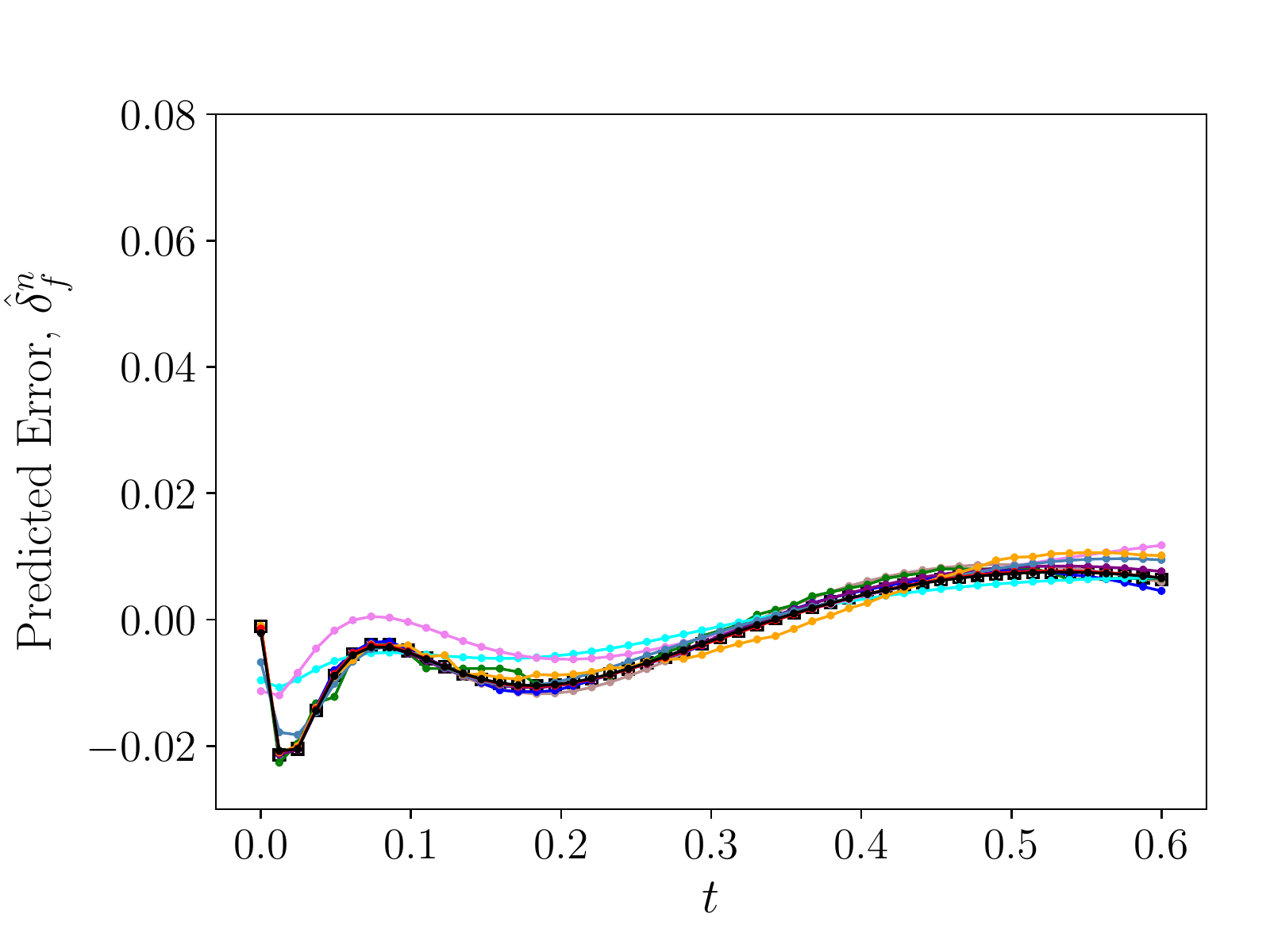}
	\caption{QoI error $\qoiError$, $\card{\paramtrainindex}=40$}
\end{subfigure}
\end{center}
\caption{\textit{Advection--diffusion equation.} Error response $\stateError$ (left) and $\qoiError$ (right)
	predicted by each regression method as a function of time. Results are shown
	for the best performing feature engineering method on the
	$\card{\paramtrainindex} = 8$ case (top) and the $\card{\paramtrainindex} = 40$ case (bottom).
	}
\label{fig:ad_yhat_vs_t}
\end{figure}

Figure~\ref{fig:ad_hist} shows the regression errors (i.e., the response error
of the deterministic regression-function models) for the best performing
regression method (i.e., LSTM) employing its best performing
feature-engineering method. The figure compares the regression-error
distribution with the distributions predicted by the Gaussian, Laplacian, and
AR1 stochastic noise models. Table~\ref{tab:ad_hist} reports validation
frequencies of the approximated distributions, where $\omega(C)$ corresponds
to the fraction of test data that lie within the $C$-prediction interval
generated by the stochastic noise model. We also report the Kolmogorov--Smirnov (K-S)
statistic. We observe that the Laplacian noise model yields the most
accurate prediction-interval frequencies, with AR1 yielding
similar performance to that of the Laplacian model.
\begin{figure}
\begin{center}
\begin{subfigure}[t]{0.49\textwidth}
\includegraphics[width=1.\linewidth]{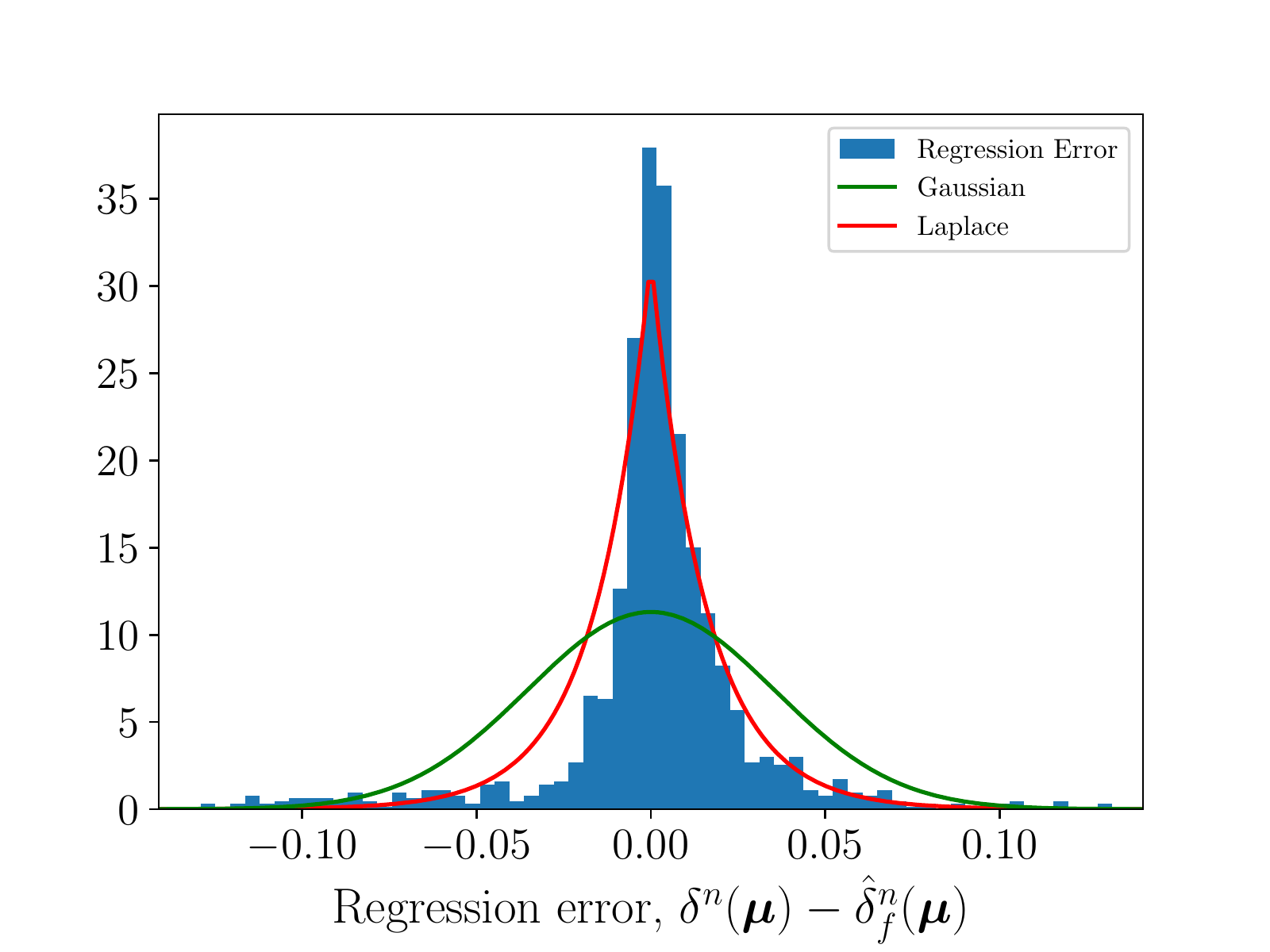}
\caption{Normed State Error}
\end{subfigure}
\begin{subfigure}[t]{0.49\textwidth}
\includegraphics[width=1.\linewidth]{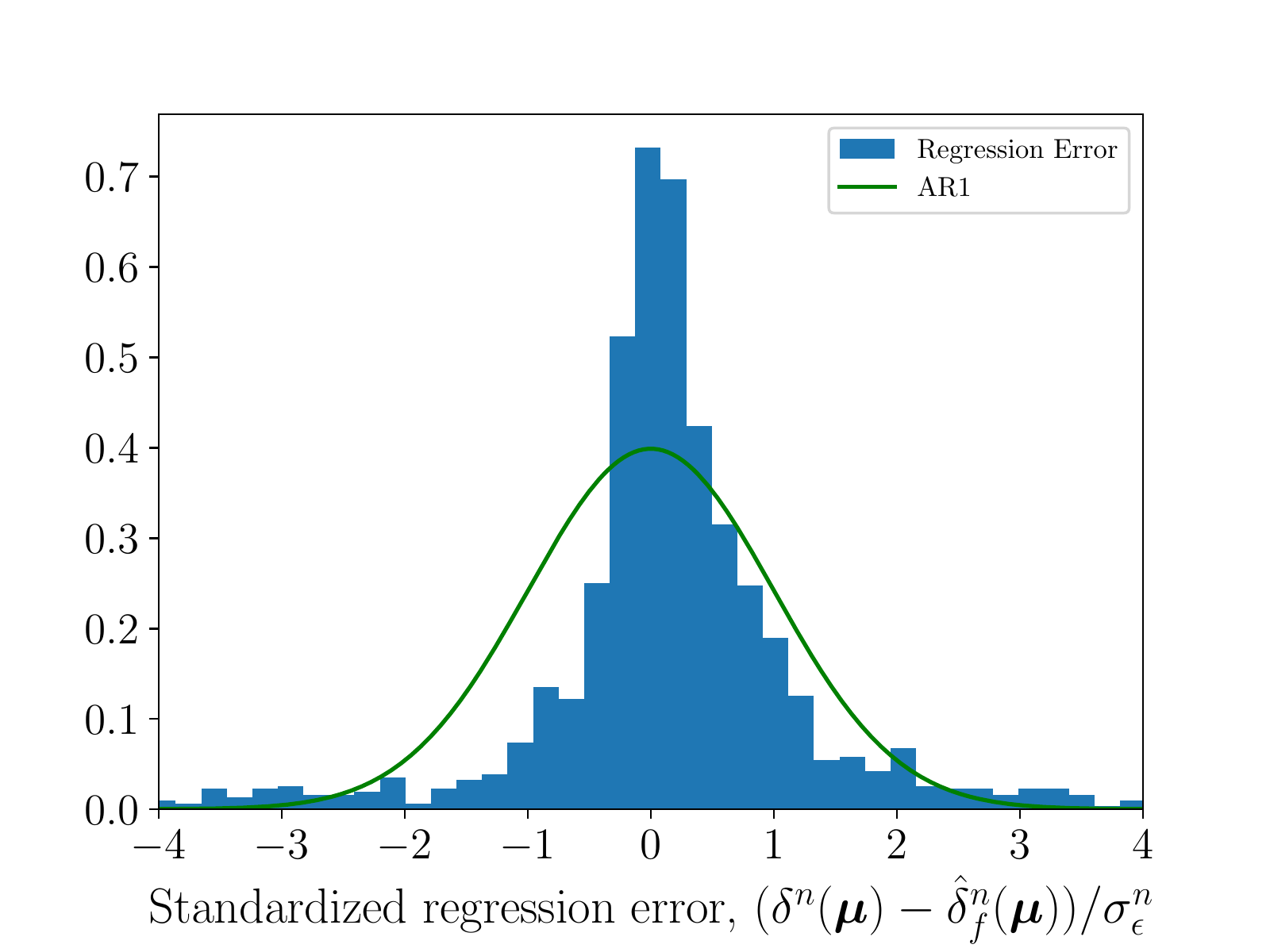}
\caption{Standardized Normed State Error}
\end{subfigure}

\begin{subfigure}[t]{0.49\textwidth}
\includegraphics[width=1.\linewidth]{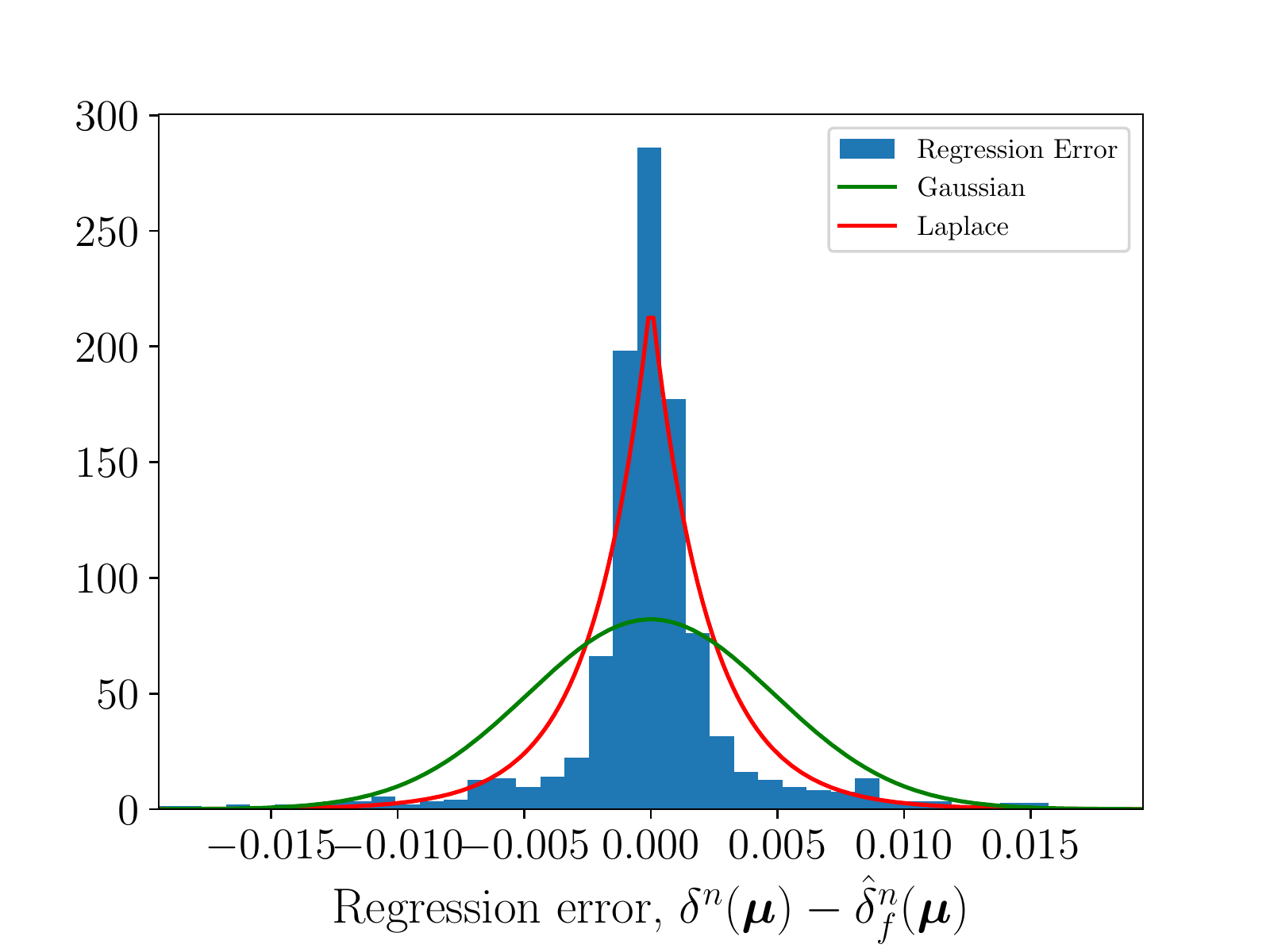}
\caption{QoI Error}
\end{subfigure}
\begin{subfigure}[t]{0.49\textwidth}
\includegraphics[width=1.\linewidth]{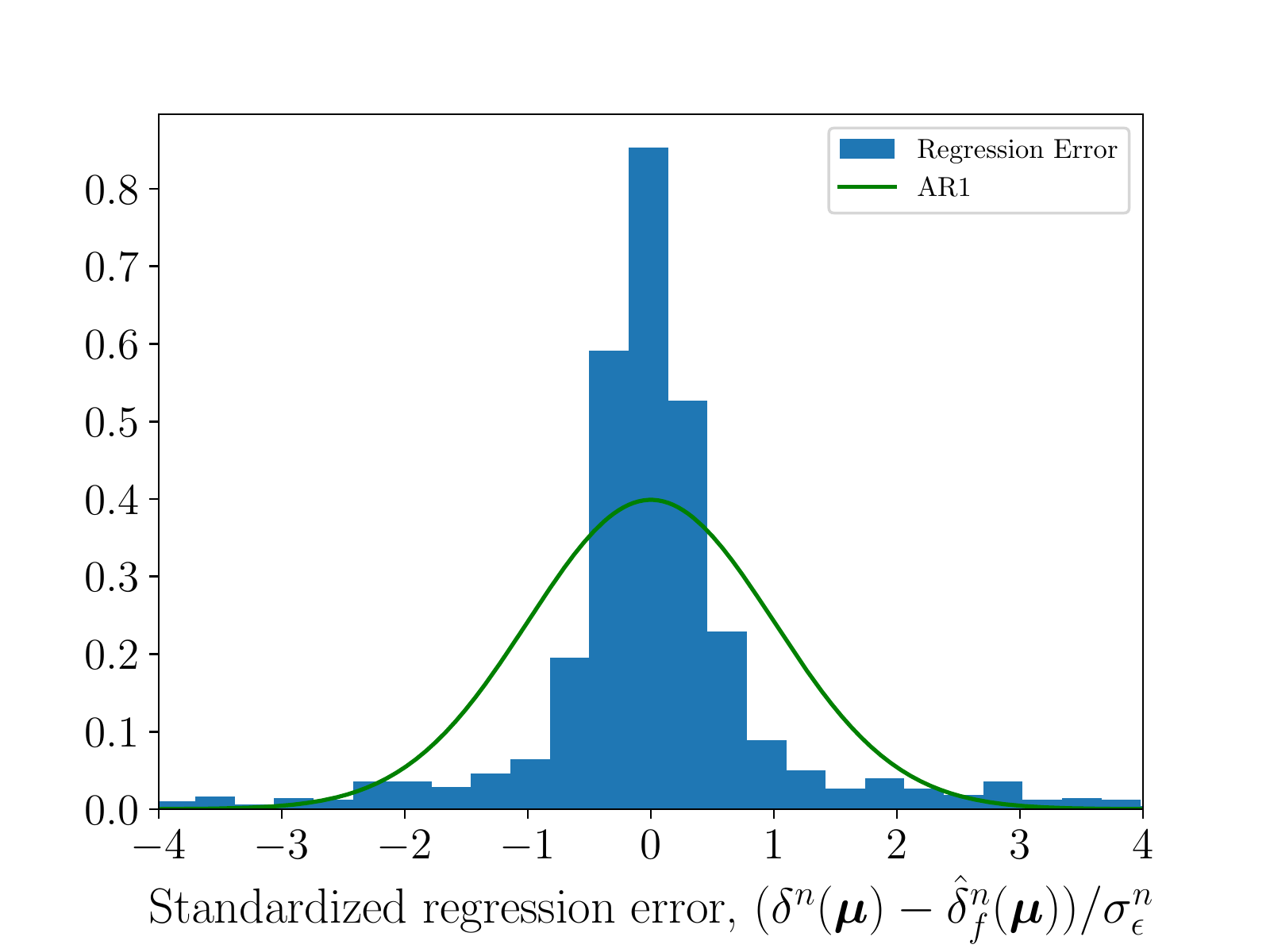}
\caption{Standardized QoI Error}
\end{subfigure}

\end{center}

\caption{\textit{Advection--diffusion equation.} Histogram of regression errors on the noise test set, $\testSetNoise$, for the normed state error (top) and QoI error (bottom). On the left, we show absolute errors $(\errorGennmuArg{n} - \errorApproxMeanGennmunArg{n})$, $\parametervec \in \paramtestsetNoise$, $n \in \timeIndexSet$. On the right, we show standardized errors, $(\errorGennmuArg{n} - \errorApproxMeanGennmunArg{n})/\sigma_{\epsilon}^n$, $\parametervec \in \paramtestsetNoise$, $n \in \timeIndexSet$. The true regression errors are compared to the distributions predicted by the Gaussian, Laplacian, and AR1 error models. Regression errors are shown 
for the LSTM deterministic regression function model with the best performing feature-engineering method on the $\card{\paramtrainindex} = 40$ case.}
\label{fig:ad_hist}
\end{figure}

In summary, for the advection--diffusion example, the LSTM model yielded best
performance, followed by RNN and ANN. We observed the residual-based features
to yield best performance as well. Finally, the Laplacian model
yielded the most accurate stochastic noise model.
\begin{table}[ht]
\begin{center}
\begin{tabular}{c |c c| c c |c c}
    \hline
    & \multicolumn{2}{c|}{Gaussian}  & \multicolumn{2}{c|}{Laplacian} &
		\multicolumn{2}{c}{AR1} \\
    Prediction interval &  $\stateError$ & $\qoiError$ &  $\stateError$ & $\qoiError$ &  $\stateError$ & $\qoiError$   \\ \hline
 $\omega( 0.68 )$ &  0.842 & 0.805  & 0.724 & 0.673 & 0.714 & 0.725\\
 $\omega( 0.95 )$ & 0.924 & 0.932 & 0.883 & 0.913  & 0.834 & 0.880 \\
 $\omega( 0.99 )$ & 0.943 & 0.967 & 0.868 & 0.967 & 0.868 & 0.913 \\
 \hline

 K-S Statistic & 0.216 & 0.216 & 0.114 & 0.088 & 0.146 & 0.131\\ \hline
\end{tabular}
\end{center}
\caption{\textit{Advection--diffusion equation.} Prediction
	intervals for the various stochastic noise models.
	Results correspond to those generated by the
	LSTM 
	regression function model with its  best performing feature-engineering method on the
	$\card{\paramtrainindex} = 40$ case, with samples collected over $\testSetNoise$.
	}
\label{tab:ad_hist}
\end{table}

\subsection{Example 2: 
Shallow-water equations with
projection-based reduced-order model}
The second example considers the shallow water equations and approximate
solutions generated by a POD--Galerkin reduced-order model.
The governing system of PDEs is 
\begin{align*}
&\frac{\partial h}{\partial t} + \frac{\partial}{\partial \spatialvar }(  h u) + \frac{\partial}{\partial \spatialvarTwo }( h v) = 0\\
&\frac{\partial h u}{\partial t} + \frac{\partial}{\partial \spatialvar } (h u^2 + \frac{1}{2} \parameter_1 h^2) + \frac{\partial}{\partial \spatialvarTwo }( h u v) = 0\\
&\frac{\partial h v}{\partial t} + \frac{\partial}{\partial \spatialvar } (h u v) + \frac{\partial}{\partial \spatialvarTwo }( h v^2 +  \frac{1}{2} \parameter_1 h^2) = 0
\end{align*}
on the domain $(\spatialvar ,\spatialvarTwo ) \in \physDomain = [0,5] \times [0,5]$, $t \in [0,T]$
with $T=10$, and 
where $h,u,v: \physDomain \times [0,T] \rightarrow \RR{}$, with $h$ denoting the height
of the water surface, $u$ denoting the $\spatialvar$-velocity, and $v$
denoting the $\spatialvarTwo$-velocity.
The parameter $\parameter_1$ is the ``gravity'' parameter. 

The parameterized initial conditions are given by $$h(\spatialvar
,\spatialvarTwo ,0) = 1 + \parameter_2 e^{ (\spatialvar  - 1)^2 +
(\spatialvarTwo  - 1)^2 }, \quad u(\spatialvar ,\spatialvarTwo ,0) =
v(\spatialvar ,\spatialvarTwo ,0) = 0,\quad (\spatialvar ,\spatialvarTwo ) \in
\physDomain,$$ which corresponds to a Gaussian pulse centered at $(1,1)$ with
a magnitude of $\parameter_2$. We consider a parameter domain
$(\parameter_1,\parameter_2)\in\parameterdomain = [2,9]\times [0.05,2]$.

To derive a parameterized dynamical system of the form \eqref{eq:fom} for the
FOM, we apply a third-order 
discontinuous-Galerkin spatial discretization, which partitions the domain
into a 
$16 \times 16$
grid of uniform square elements and employs tensor-product polynomials of order $2$
to represent the solution over each element, yielding a FOM of dimension
$\nfom = 1.23\times 10^4$. We employ slip-wall boundary conditions at the edge
of the domain.

To define the FOM O$\Delta$E of the form \eqref{eq:fom_discretized}, we employ
a time step 
$\Delta t = 1.25\times 10^{-2}$ with the
implicit Crank--Nicolson scheme, which is characterized by multistep
coefficients $\alpha_0 = 1$, $\alpha_1 = -1$, $\beta_0=\beta_1 = 1/2$ in
\eqref{eq:FOMODeltaERes}, which yields $\ntimesteps = 800$ time instances.
Figure~\ref{fig:swe_surf} shows elevation surfaces obtained from the solution
to the FOM O$\Delta$E at one randomly chosen parameter instance.

As in the first example,
we employ projection-based reduced-order
models to construct approximate solutions.
We again construct a POD 
trial-basis matrix $\basismat\in\RRstar{N\times \dimrom}$ 
by executing Algorithm \ref{alg:pod} in
Appendix \ref{appendix:pod} with inputs  $\parameterdomainROM =
\{3,6,9\} \times  \{
	0.05,0.1,0.2 \}$, $\nskip= 1$, reference state 
$\statevecIntercept(\param) = \boldsymbol 0$, and $\dimrom = 78$.
For the ROM, we employ Galerkin projection such that the ROM ODE corresponds
to \eqref{eq:ROM} with $\testmat = \basismat$. 
We employ the same time
discretization for the ROM as was employed for the FOM.


We model both the QoI error $\qoiErrorArg{\tindx}(\parametervec)$,
where the QoI corresponds to the 
water-surface height in the middle of the
domain,
and the normed state error $\stateErrorArg{\tindx}(\parametervec)$.


\begin{figure}
\begin{center}
\begin{subfigure}[t]{0.31\textwidth}
\includegraphics[width=1.\linewidth]{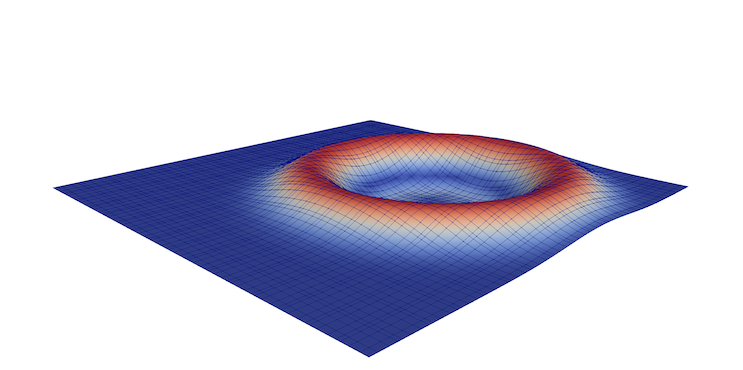}
\caption{$t=1$}
\end{subfigure}
\begin{subfigure}[t]{0.31\textwidth}
\includegraphics[width=1.\linewidth]{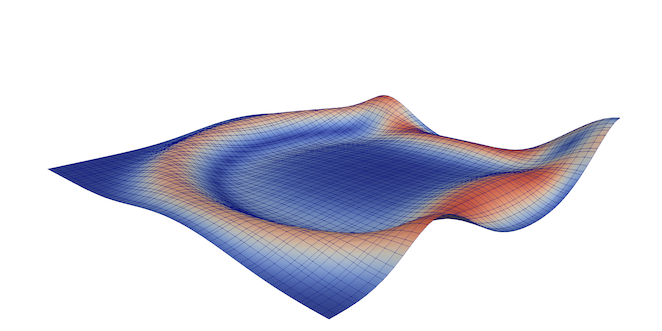}
\caption{$t=2$}
\end{subfigure}
\begin{subfigure}[t]{0.31\textwidth}
\includegraphics[width=1.\linewidth]{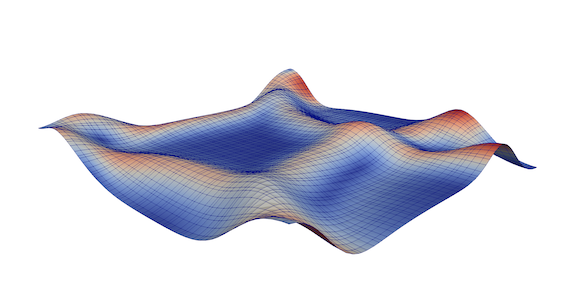}
\caption{$t=3$}
\end{subfigure}
\end{center}
	\caption{\textit{Shallow-water equations.} Elevation surfaces of $h$ for a
	randomly chosen parameter instance. Solutions are colored by velocity magnitude.}
\label{fig:swe_surf}
\end{figure}

\subsubsection{Coarse time grid, training, validation, and test sets}
We now describe how we execute Steps
\ref{step:timeIndexSet}--\ref{step:trainValTest} and \ref{train:noiseextract}
introduced Section
\ref{sec:trainingModelSel}. For Step \ref{step:timeIndexSet},
we set 
$\timeIndexSet =
\{4n\}_{n=1}^{200}$; this coarse grid is again 
selected such that the error responses on the training set are well represented. 
For Step \ref{step:trainValTest}, we employ set sizes of
$\card{\parameterdomainTrain} = 40$, 
$\card{\parameterdomainVal}=10$, and
$\card{\parameterdomainTest} = 20$.
Again, to assess the impact of the training set size on the performance of the
deterministic regression-function model,
we consider four smaller training and validation sets
with $(\card{\parameterdomainTrain},\card{\parameterdomainVal})
\in\{(8,2),(16,4), (24,6), (32,8)\}$, which are constructed by sampling the
original sets $\parameterdomainTrain$ and $\parameterdomainVal$.
In Step \ref{train:noiseextract}, we set
$\card{\paramtrainsetNoise}=10$.

Lastly, we note that we do not consider any feature-engineering methods that
employ \ref{feat:r}, as the dimension of the residual $\nfom = 1.23\times
10^4$ is prohibitively large in this example.


\subsubsection{Regression results}
Figures~\ref{fig:swe_converge}--\ref{fig:swe_hist} summarize the
performance of the deterministic regression-function models and stochastic
noise models.
First, Figure~\ref{fig:swe_converge} reports the dependence of the performance of each 
deterministic regression-function model (with its best performing
feature-engineering method) on the size of the training set
$\card{\parameterdomainTrain}$.
Figures~\ref{fig:swe_converge_a}--\ref{fig:swe_converge_b} considers all
feature-engineering methods, while Figures~\ref{fig:swe_converge_c}--\ref{fig:swe_converge_d} limit consideration to feature-engineering
methods that omit time from the set of features. When time is included as a
feature, LSTM and RNN yield the best overall performance, followed
closely by ANN; ANN slightly outperforms RNN in predicting the normed state 
error when time is included as a feature. ARX and LARX do not perform well for predicting the normed
state error $\stateError$, but do perform very well in predicting the QoI error
$\qoiError$. We additionally observe that both methods for training ARX yield similar trends. ARX (RT) 
outperforms ARX (NRT) in predicting the normed state error, and ARX (NRT) outperforms ARX (RT) in predicting the QoI error. 
Next, we again observe that 
ANN-I performs worse than the standard ANN. We also observe that, in general, ANN-I (RT) performs slightly better than ANN-I (NRT) and yields similar results to ANN. 
Next, in comparing Figures~\ref{fig:swe_converge_a} and~\ref{fig:swe_converge_b} to 
Figures~\ref{fig:swe_converge_c} and~\ref{fig:swe_converge_d}, it is observed that 
the non-recursive regression methods are more reliant upon time to yield 
good performance than the recursive regression methods. This
indicates that the non-recursive methods  will likely not
	generalize well for problems characterized by phase shifts or for
	prediction beyond the training time interval.
Lastly, the
time-local GP benchmark is the worst performing method, and is particularly
poor at predicting the QoI errors. This is thought to be due to the fact 
that the parameters are low quality features. 

\begin{figure}
\begin{center}
\begin{subfigure}[t]{0.49\textwidth}
\includegraphics[width=1.\linewidth]{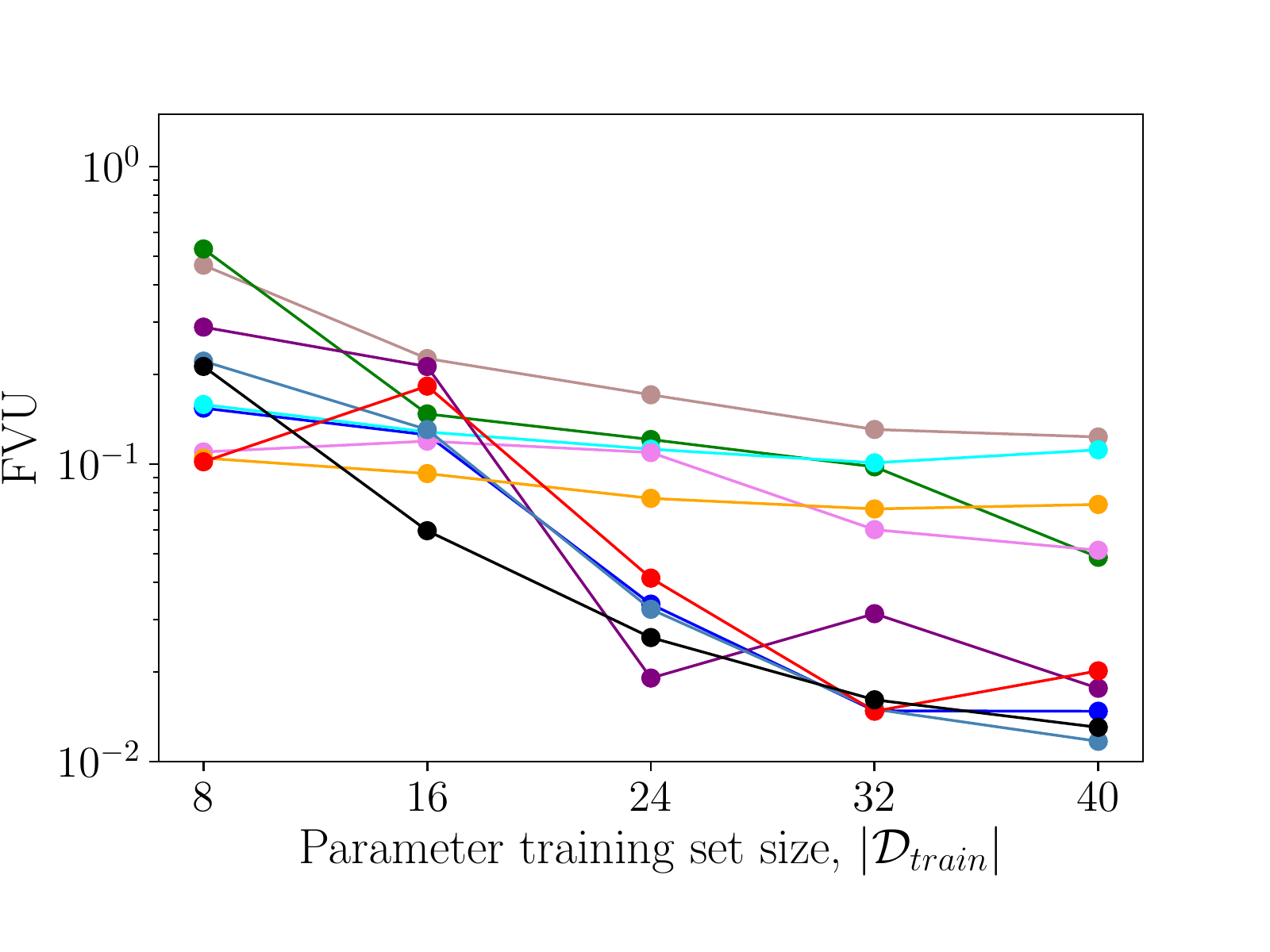}
\caption{Normed state error $\stateError$ (All feature-engineering methods)}
\label{fig:swe_converge_a}
\end{subfigure}
\begin{subfigure}[t]{0.49\textwidth}
\includegraphics[width=1.\linewidth]{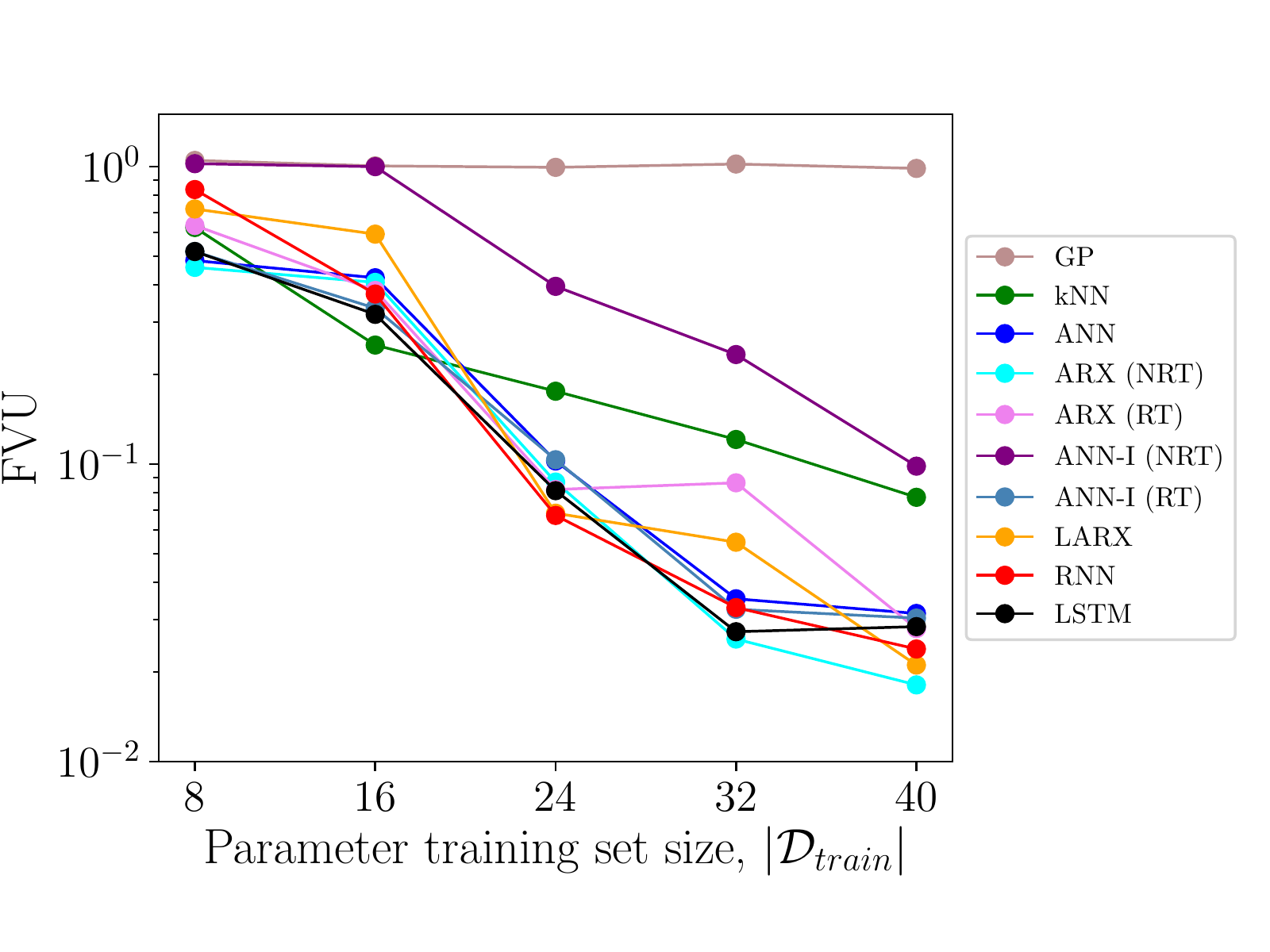}
\caption{QoI error $\qoiError$ (All feature-engineering methods)}
\label{fig:swe_converge_b}
\end{subfigure}
\begin{subfigure}[t]{0.49\textwidth}
\includegraphics[width=1.\linewidth]{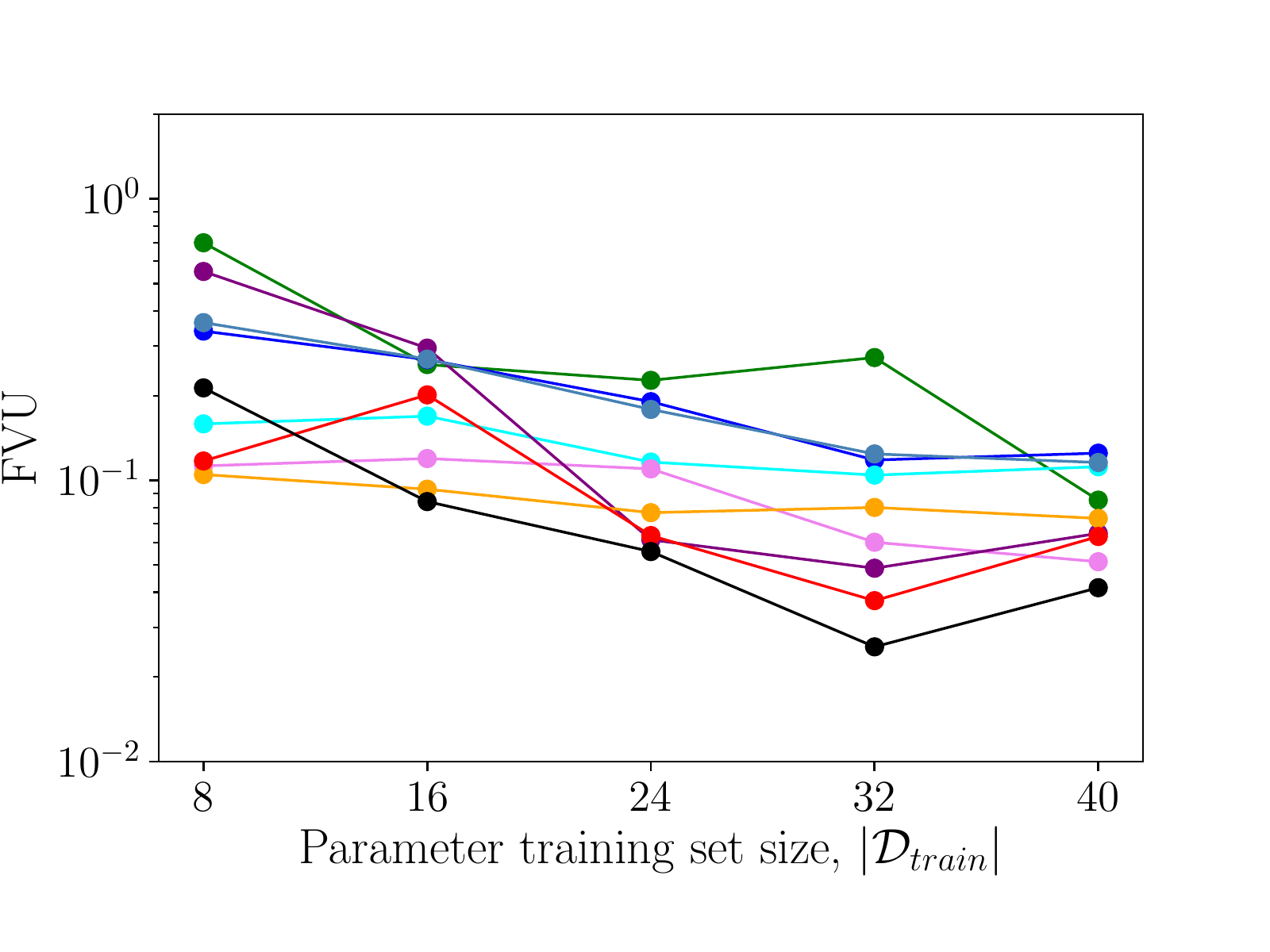}
\caption{Normed state error $\stateError$ (Feature-engineering methods without time)}
\label{fig:swe_converge_c}
\end{subfigure}
\begin{subfigure}[t]{0.49\textwidth}
\includegraphics[width=1.\linewidth]{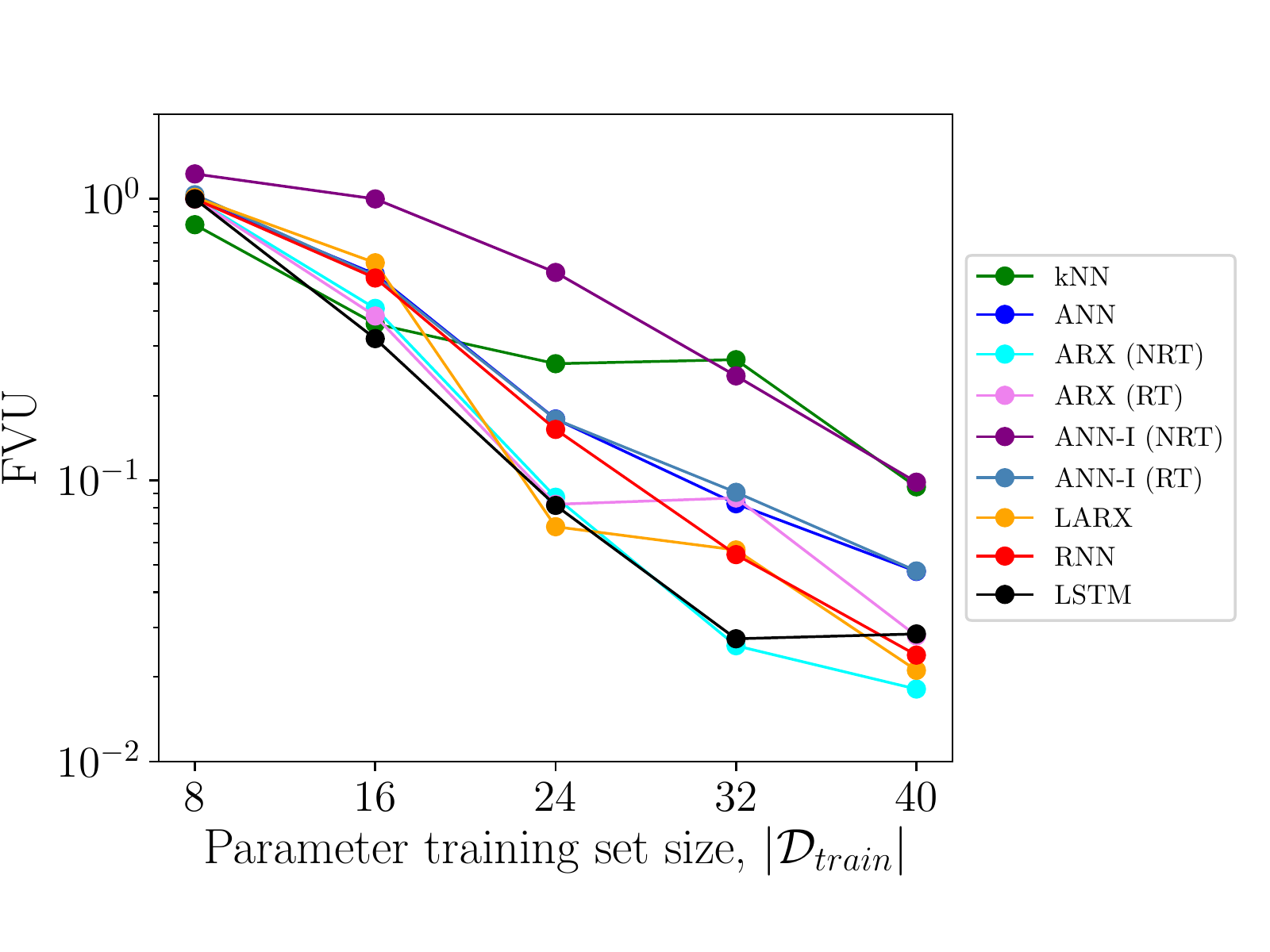}
\caption{QoI error $\qoiError$  (Feature-engineering methods without time)}
\label{fig:swe_converge_d}
\end{subfigure}
\end{center}
\caption{\textit{Shallow-water equations.} 
	Fraction of variance
	unexplained for all considered deterministic regression-function models,
	with each employing their best respective feature-engineering method.}
\label{fig:swe_converge}
\end{figure}

Table~\ref{tab:swe_summary} summarizes the percentage of cases that a
deterministic regression-function model produces results with the lowest FVU,
where each case is defined by one of the 11 feature-engineering methods
(recall that the residual, i.e., \ref{feat:r} is not considered in this
example) and
one of the 5 training-set sizes, yielding 55 total cases. We again present results for the cases where (1)
all feature methods are considered and (2) only feature methods that don't
include time are considered.  Again, the results indicate that LSTM yields the
best overall performance, followed by LARX and RNN. Recursive methods yield
the best performance for predicting the normed state error in $81.8\%$ of
cases when all feature engineering methods are considered, and in $86.7\%$ of
cases when only feature engineering methods that do not include time are
considered. For the QoI prediction, recursive methods yield the best
performance in $85.5\%$ and $86.7\%$ of cases, respectively. As before, this
shows that including time as a feature can (slightly) improve the relative performance of
non-recursive regression methods. Also as observed previously, including time as an 
input feature is observed to be more beneficial in predicting the normed state error than the QoI error.
\begin{table}[ht]
\small{
\begin{center}
\setlength{\tabcolsep}{4pt}
\begin{tabular}{|c| c | c c c c c c c c c | c|}
	\hline
error	&
\begin{tabular}{c }
	include feat.\\
	eng.\ w/ time?
\end{tabular}
	&kNN & ANN & \begin{tabular}{c} ARX \\ (NRT) \end{tabular} & \begin{tabular}{c} ARX \\ (RT) \end{tabular} &  \begin{tabular}{c} ANN-I \\ (NRT)\end{tabular}  & \begin{tabular}{c} ANN-I \\ (RT)\end{tabular} & LARX &  RNN & LSTM & 
	\begin{tabular}{c}	
	Recursive\\ total
	\end{tabular}\\ \hline
		$\stateError$ &Yes & 9.1 & 9.1& 0.0 & 5.5 & 9.1  & 9.1   & 5.5 & 7.3 & 45.5 & 81.8 \\  
		$\stateError$ &No  & 13.3 & 0.0 & 0.0 & 10.0 & 6.7 & 0.0 & 10.0  & 3.3 & 56.7  & 86.7 \\ 
		$\qoiError$   &Yes & 14.5 & 0.0 & 12.7 & 5.5 & 5.5 & 7.3 & 20.0  & 16.4  & 18.2 & 85.5  \\ 
		$\qoiError$   &No  & 13.3 & 0.0 & 16.7 & 3.3 & 3.3 & 3.3 & 16.7  & 13.3  & 30.0 & 86.7 \\ \hline
\end{tabular}
\end{center}
}
\caption{\textit{Shallow-water equations.} Percentage of cases having
	the lowest predicted FVU for each regression method. Results are summarized
	over all values $\card{\paramtrainindex} \in\{ 10,20,30,40,50\}$ and all
	feature-engineering methods. Recursive total reports the sum of all
	regression methods
	in \ref{cat:recursiveLatentPred} and 
	\ref{cat:recursive}.}
\label{tab:swe_summary}
\end{table}


Figure~\ref{fig:swe_grid} summarizes the performance of each combination of
regression method and feature-engineering method for the cases
$\card{\parameterdomainTrain} = 8$ (top) and 
$\card{\parameterdomainTrain} = 40$ (bottom).
For the case $\card{\paramtrainset} = 40$,
the recursive regression methods
generally display better overall performance than the non-recursive regression
methods in predicting the normed state error $\stateError$. The low-dimensional
feature engineering methods are additionally seen to lead to better overall
performance than the high-dimensional feature engineering methods in this
case, as using a small number of features typically leads to lower-capacity
models that generalize with limited
training data (see \ref{cons:numberFeatures}). For the QoI
error, all regression methods perform poorly for the case $\card{\paramtrainset} =
10$, with the lowest observed FVU value around $\text{FVU}
= 0.5$. For the case $\card{\paramtrainset} = 40$, all methods perform
substantially better; here, it is clear that residual-based features
are essential for obtaining good performance. LSTM, RNN, ANN, and ANN-I (RT) generally yield the best performance for
the normed state error $\stateError$; however, ANN  and ANN-I (RT) perform well only when
time is included as a feature. For the QoI prediction, LSTM, RNN, LARX,
and ARX (NRT) generally perform best.


\begin{figure}
\begin{center}
\begin{subfigure}[t]{0.49\textwidth}
\includegraphics[width=1.\linewidth]{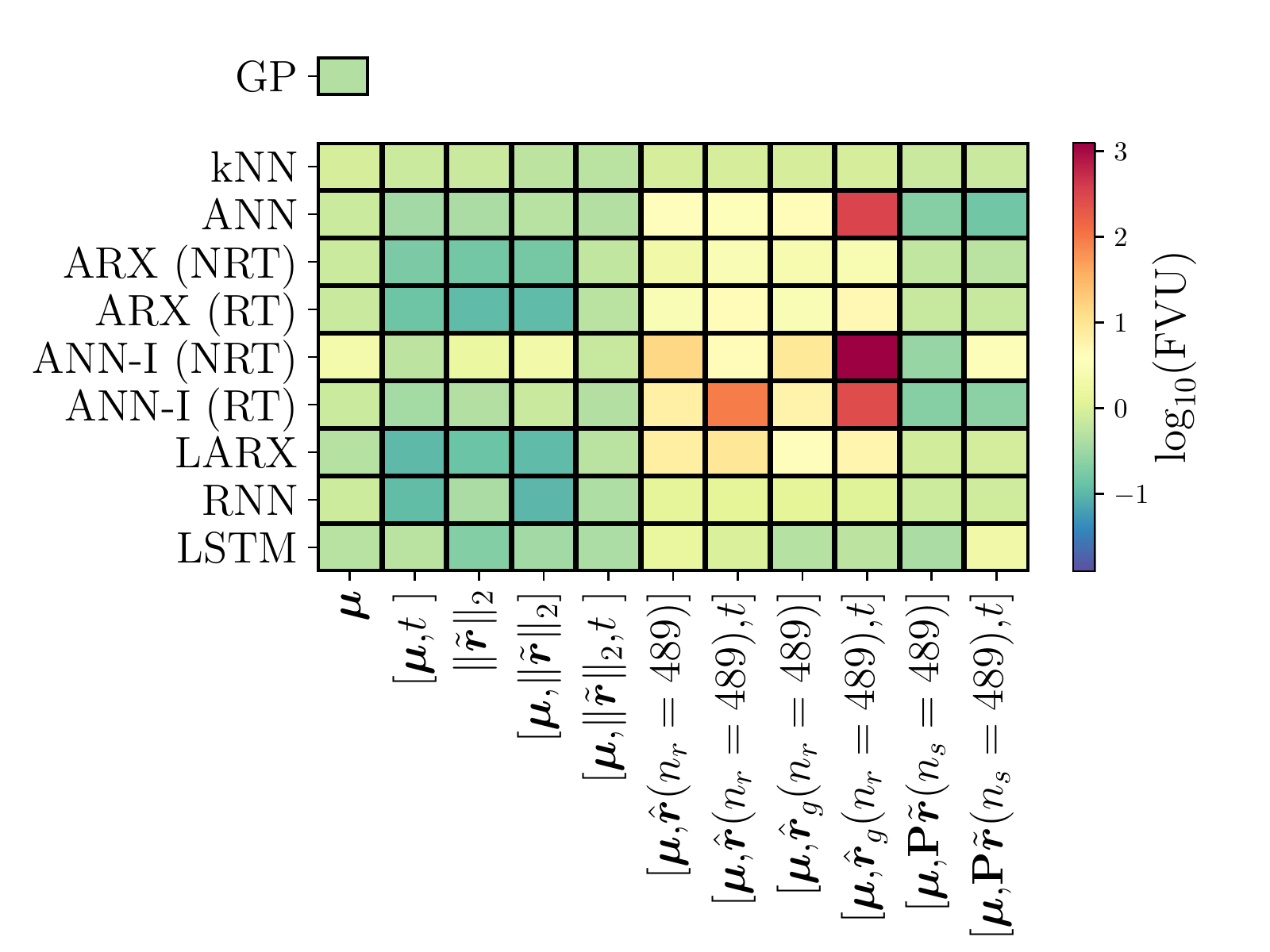}
	\caption{Normed state error $\stateError, \card{\paramtrainset} = 8$}
\end{subfigure}
\begin{subfigure}[t]{0.49\textwidth}
\includegraphics[width=1.\linewidth]{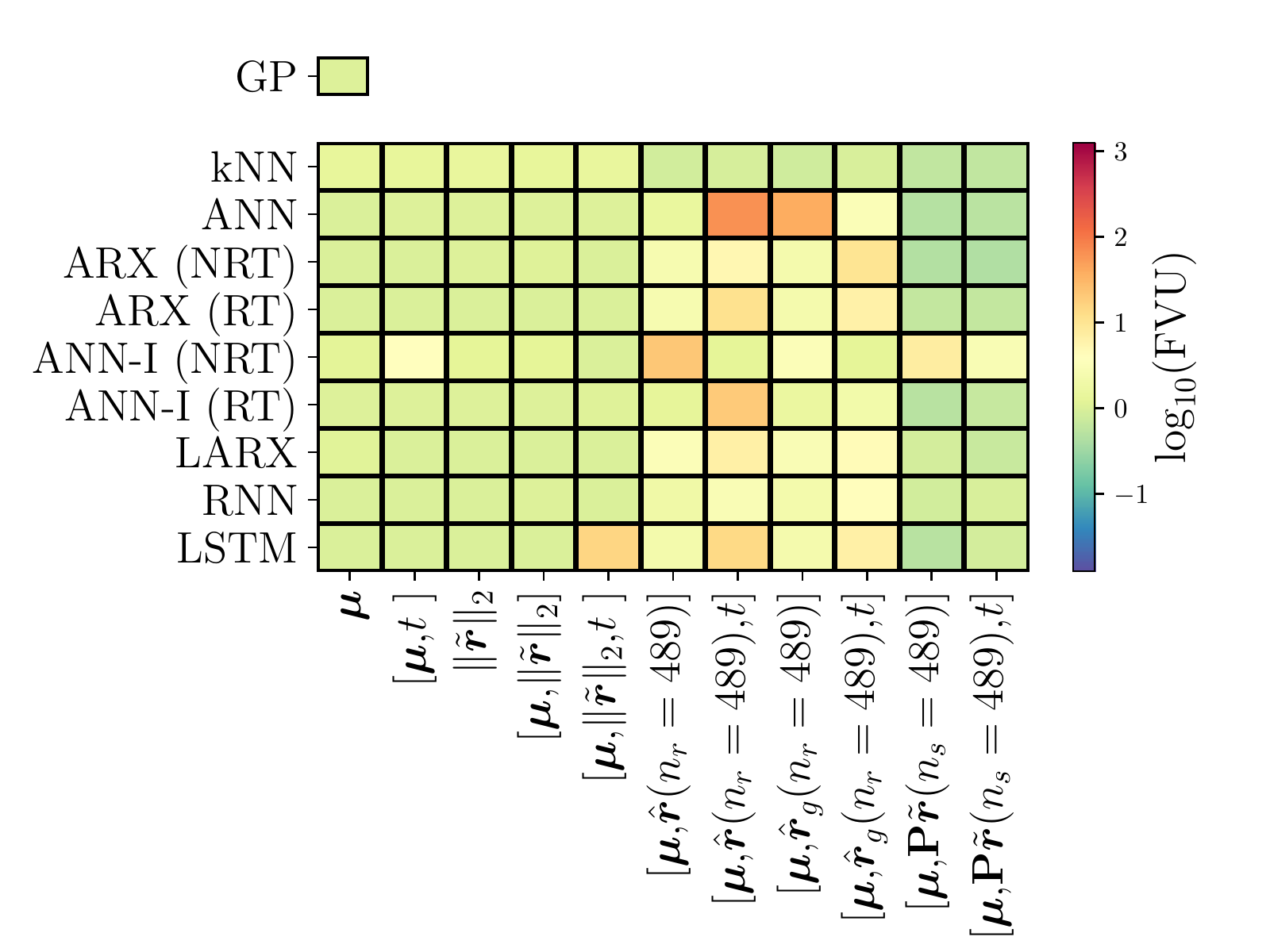}
	\caption{QoI error $\qoiError, \card{\paramtrainset} = 8$}
\end{subfigure}
\begin{subfigure}[t]{0.49\textwidth}
\includegraphics[width=1.\linewidth]{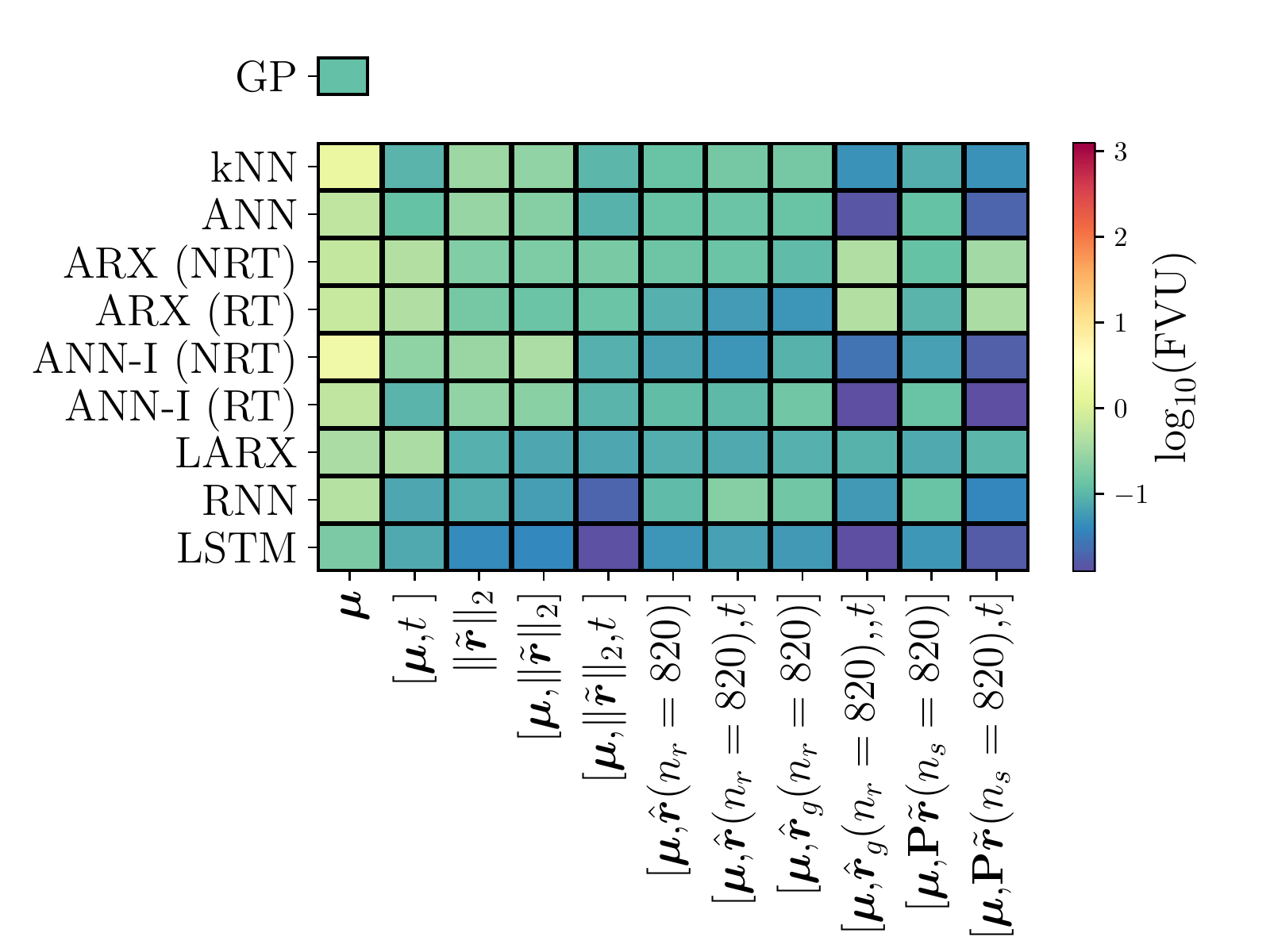}
	\caption{Normed state error $\stateError, \card{\paramtrainset} = 40$}
\end{subfigure}
\begin{subfigure}[t]{0.49\textwidth}
\includegraphics[width=1.\linewidth]{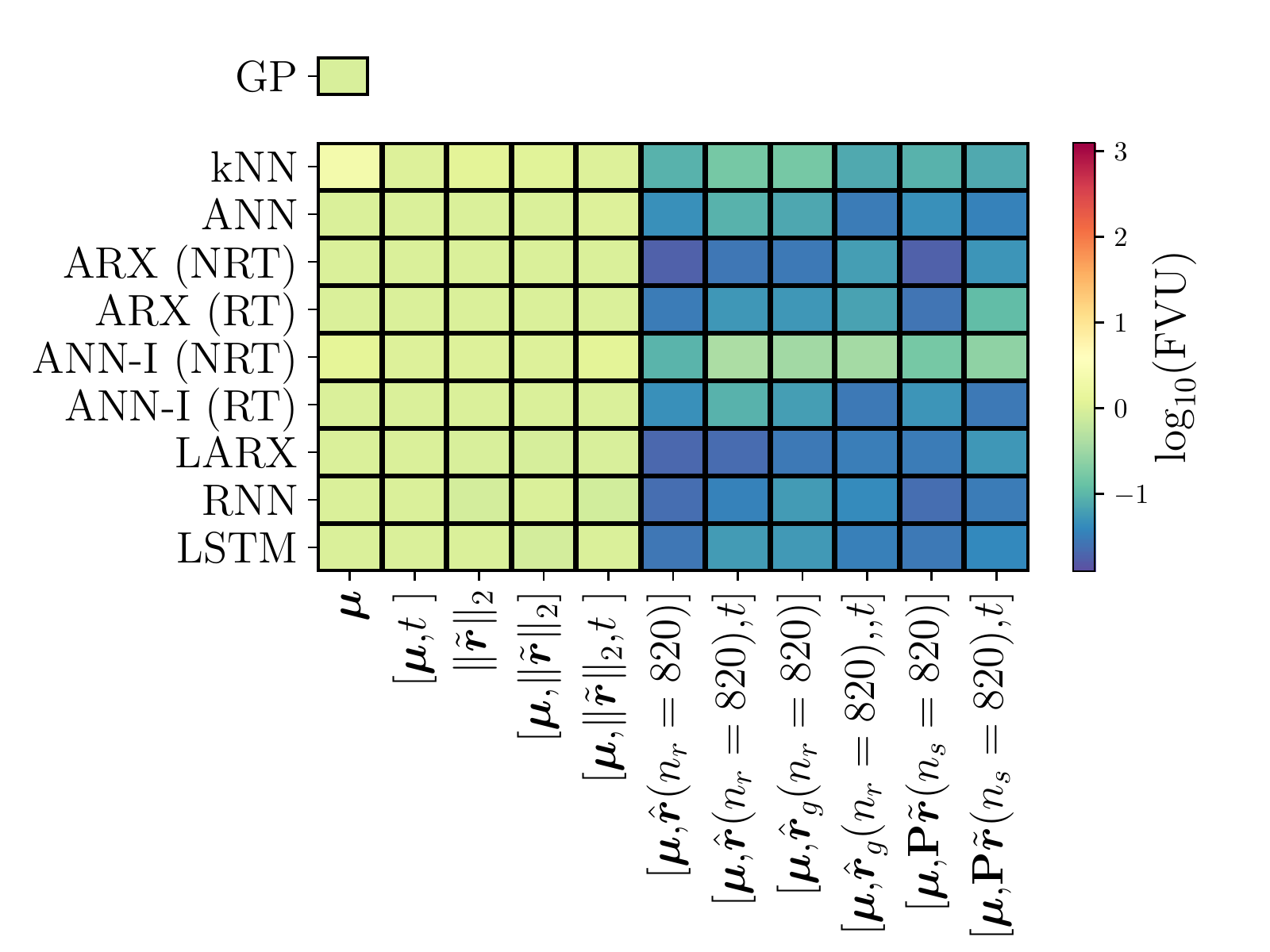}
	\caption{QoI error $\qoiError, \card{\paramtrainset} = 40$}
\end{subfigure}
\end{center}
\caption{\textit{Shallow-water equations.} Summary of the performance of
	each combination of regression method and feature-engineering method for the
	cases $\card{\paramtrainindex} = 8$ (top) and 
	$\card{\paramtrainindex} = 40$ (bottom).}
\label{fig:swe_grid}
\end{figure}


Figure~\ref{fig:swe_yhat_vs_t} shows the predictions made by the different
considered regression methods (for their best performing respective
feature-engineering method) as a function of time for a randomly drawn element
of $\parameterdomainTest$ for case $\card{\paramtrainset} = 40$.
The figure shows that all approaches except GP provide reasonably accurate
predictions for both the normed state and QoI errors.

Figure~\ref{fig:swe_hist} and Table~\ref{tab:swe_hist} summarize the
performance of the stochastic noise models, where results are shown for the
LSTM regression method employing its best performing
feature-engineering method for the case $\card{\paramtrainindex} = 40$.
Similar to the previous case, the Laplacian noise model yields the best
performance.


\begin{figure}
\begin{center}
\begin{subfigure}[t]{0.49\textwidth}
\includegraphics[width=1.\linewidth]{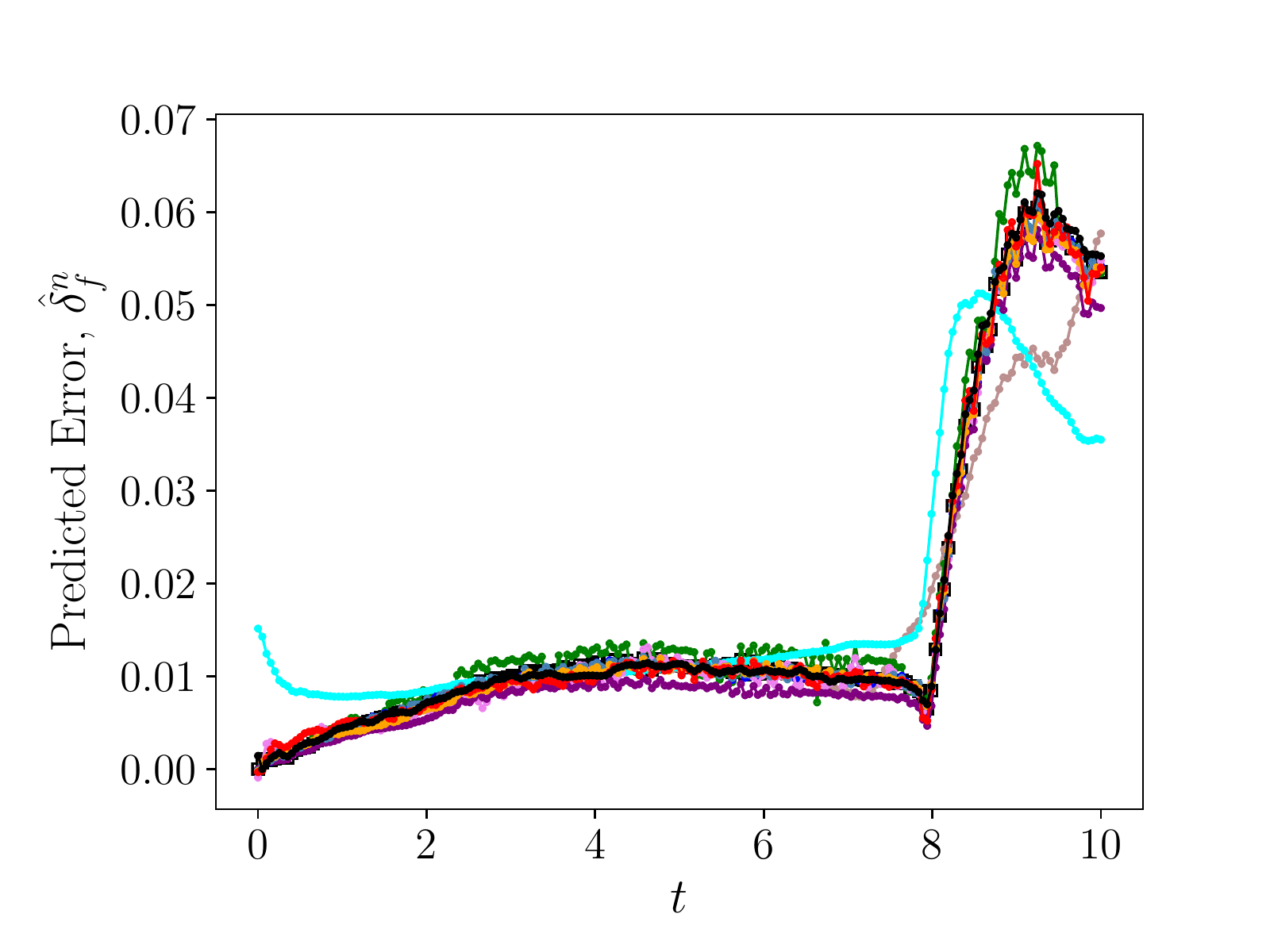}
\caption{Normed state error $\stateError$}
\end{subfigure}
\begin{subfigure}[t]{0.49\textwidth}
\includegraphics[width=1.\linewidth]{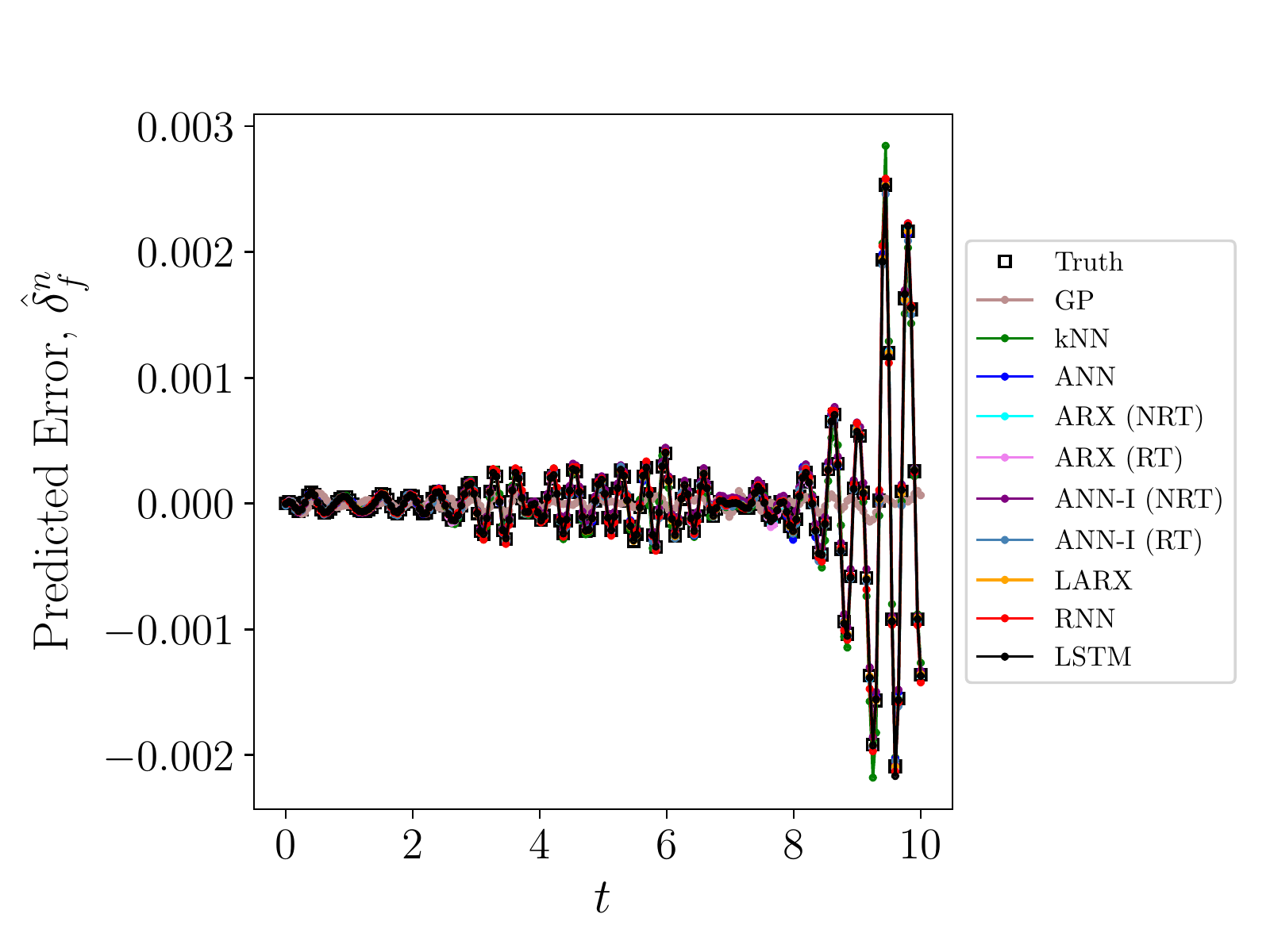}
\caption{QoI error $\qoiError$}
\end{subfigure}
\end{center}
\caption{\textit{Shallow-water equations.} Error response $\stateError$ (left) and $\qoiError$ (right)
	predicted by each regression method as a function of time. Results are shown
	for the best performing feature engineering method on the
	$\card{\paramtrainindex} = 40$ case.
}
\label{fig:swe_yhat_vs_t}
\end{figure}

\begin{figure}
\begin{center}
\begin{subfigure}[t]{0.49\textwidth}
\includegraphics[width=1.\linewidth]{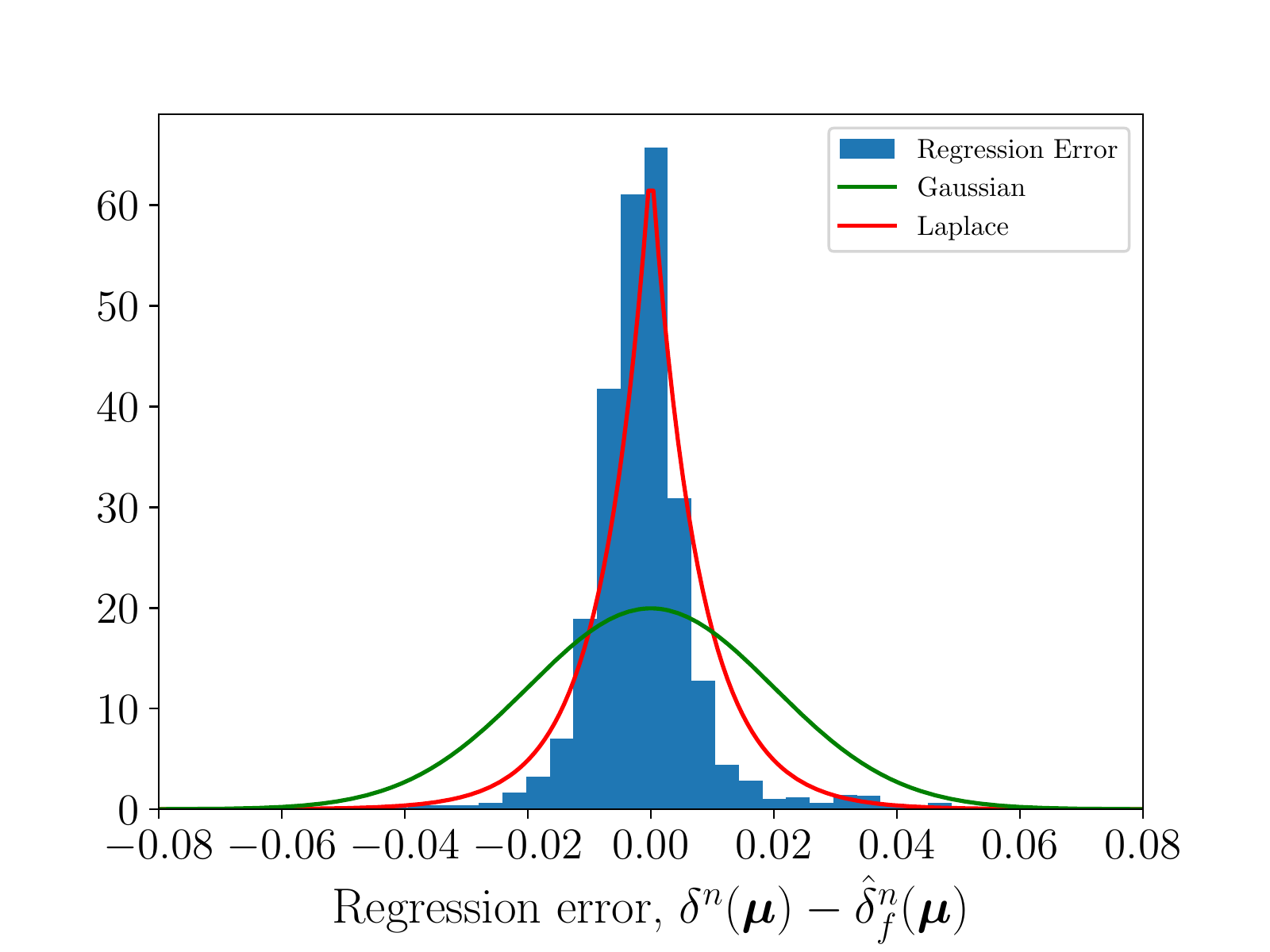}
\caption{Normed State Error}
\end{subfigure}
\begin{subfigure}[t]{0.49\textwidth}
\includegraphics[width=1.\linewidth]{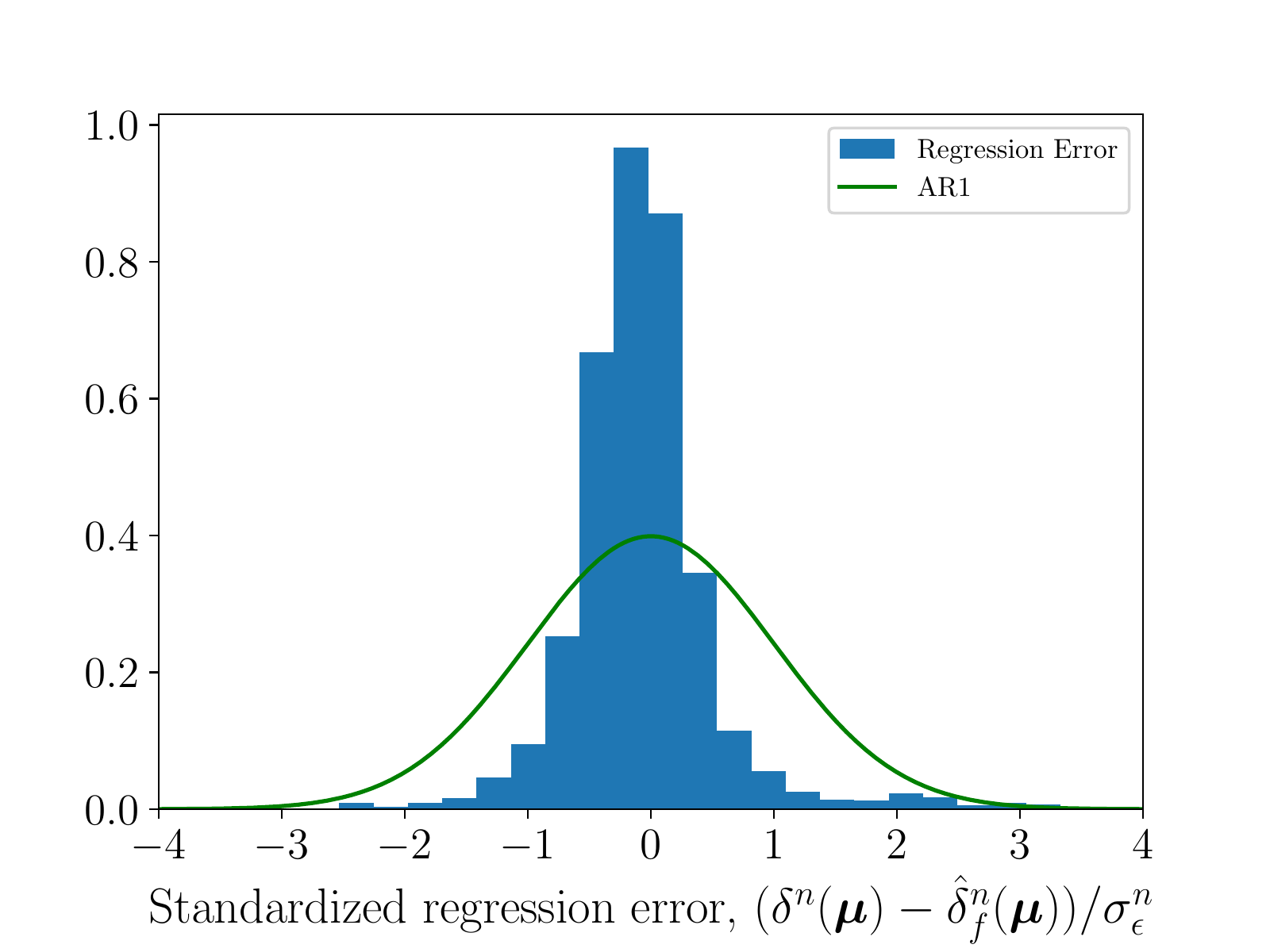}
\caption{Standardized Normed State Error}
\end{subfigure}

\begin{subfigure}[t]{0.49\textwidth}
\includegraphics[width=1.\linewidth]{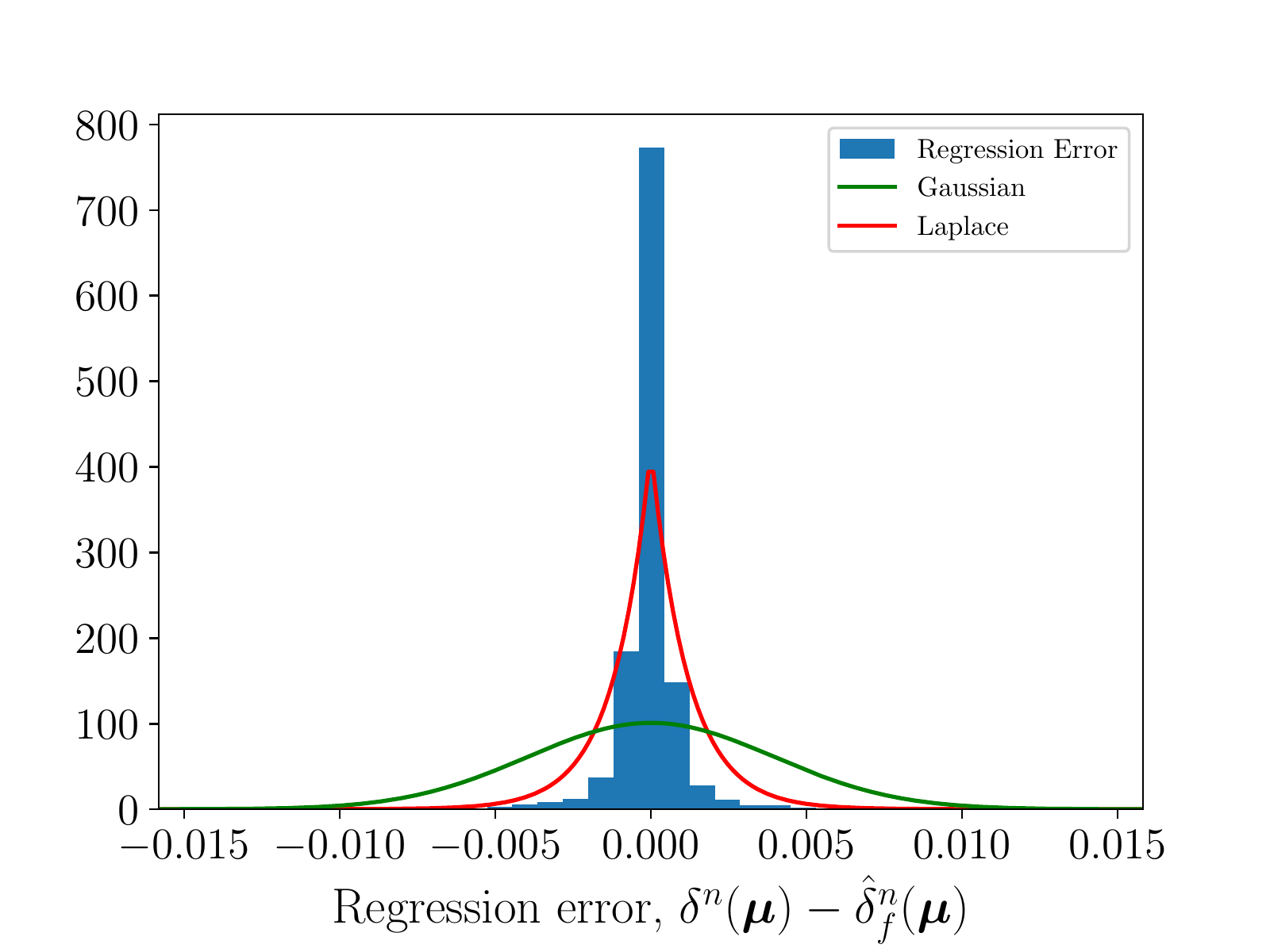}
\caption{QoI Error}
\end{subfigure}
\begin{subfigure}[t]{0.49\textwidth}
\includegraphics[width=1.\linewidth]{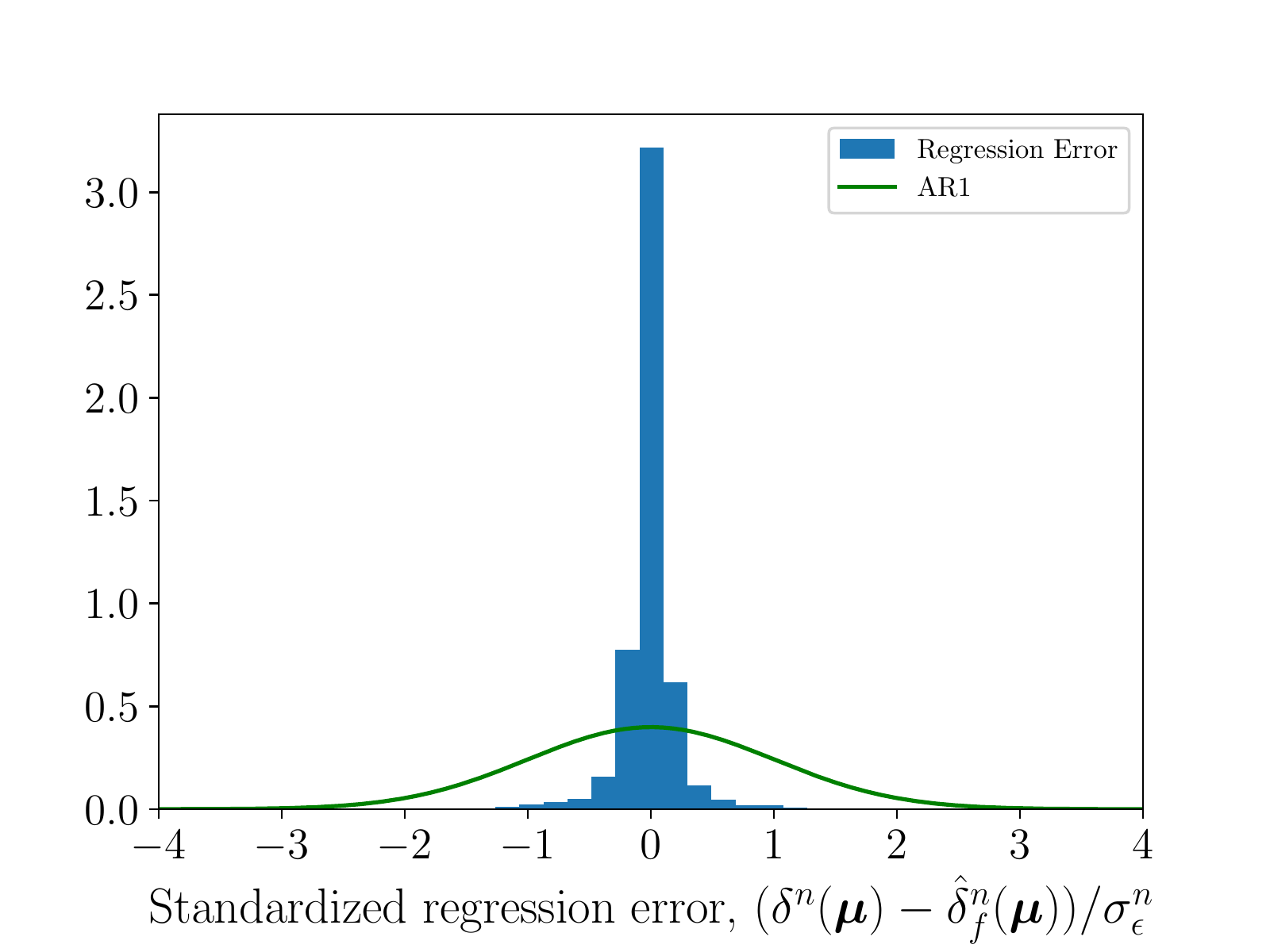}
\caption{Standardized QoI Error}
\end{subfigure}

\end{center}
\caption{\textit{Shallow-water equations.} Histogram of regression errors on the noise test set, $\testSetNoise$, for the normed state error (top) and QoI error (bottom). On the left, we show absolute errors $(\errorGennmuArg{n} - \errorApproxMeanGennmunArg{n})$, $\parametervec \in \paramtestsetNoise$, $n \in \timeIndexSet$. On the right, we show standardized errors, $(\errorGennmuArg{n} - \errorApproxMeanGennmunArg{n})/\sigma_{\epsilon}^n$, $\parametervec \in \paramtestsetNoise$, $n \in \timeIndexSet$. The true regression errors are compared to the distributions predicted by the Gaussian, Laplacian, and AR1 error models. Regression errors are shown 
for the LSTM deterministic regression function model with the best performing feature-engineering method on the $\card{\paramtrainindex} = 40$ case.}
\label{fig:swe_hist}
\end{figure}

\begin{table}[ht]
\begin{center}
\begin{tabular}{c | c c|  c c | c c}
    \hline
    & \multicolumn{2}{c|}{Gaussian}  & \multicolumn{2}{c|}{Laplacian} &
		\multicolumn{2}{c}{AR1} \\
    Prediction interval &  $\stateError$ & $\qoiError$ &  $\stateError$ & $\qoiError$ &  $\stateError$ & $\qoiError$   \\ \hline
 $\omega( 0.68 )$ &  0.956 & 0.971  & 0.801 & 0.848 & 0.927 & 0.972\\
 $\omega( 0.95 )$ & 0.991 & 0.995 & 0.966 & 0.960  & 0.984 & 0.990 \\
 $\omega( 0.99 )$ & 0.998 & 0.996 & 0.987 & 0.988 & 0.987 & 0.991 \\
 \hline
 K-S Statistic & 0.265 & 0.323 & 0.125 & 0.146 & 0.231 & 0.350\\ \hline
\end{tabular}
\end{center}
\caption{\textit{Shallow-water equations.} Prediction
	intervals for the various stochastic noise models.
	Results correspond to those generated by the
	best LSTM regression function model  
	with its  best performing feature-engineering method on the
	$\card{\paramtrainindex} = 40$ case, with samples collected over $\testSetNoise$.
	}
\label{tab:swe_hist}
\end{table}

\subsection{Example 3: Inviscid Burgers' equation with coarse-mesh low-fidelity model}
The final example considers the one-dimensional inviscid Burgers' equation
and approximate solutions generated by a 
low-fidelity model provided by a coarse spatial discretization.
The governing PDE is 
\begin{equation}\label{eq:burgers}
\frac{\partial u}{\partial t} + \frac{1}{2} \frac{\partial u^2}{\partial
	\spatialvar } = \parameter_1 e^{\parameter_2 \spatialvar },\quad
u(0,t) = \parameter_3, \quad u(\spatialvar ,0) = \parameter_4
	 \text{ for }\spatialvar\in (0,100]
\end{equation}
on the domain $\spatialvar  \in \physDomain = [0,100]$ with $t \in [0,T]$ and
$T=40$, and $u: \physDomain \times [0,T] \rightarrow \RR{}$.
We employ a parameter domain of
$(\parameter_1,\parameter_2,\parameter_3,\parameter_4)\in\parameterdomain =
[0.005,0.05]\times[0.005,0.05]\times[3,5]\times [0.5,2.5]
$.

To derive a parameterized dynamical system of the form \eqref{eq:fom} for the
FOM, we apply a finite-volume spatial discretization to the governing
equations \eqref{eq:burgers}, which uses an upwind flux at the cell interfaces
and a uniform control-volume width of  $\Delta \spatialvar  = 0.1$, yielding a
FOM with $\nfom=1000$ degrees of freedom. 

To define the FOM O$\Delta$E of the form \eqref{eq:fom_discretized}, we employ
a time step 
$\Delta t = 0.05$ with the
implicit Euler scheme, which is characterized by multistep
coefficients $\alpha_0 = 1$, $\alpha_1 = -1$, $\beta_0=1$ in
\eqref{eq:FOMODeltaERes}, which yields $\ntimesteps =800$ time instances.
 
To construct approximate solutions, we employ coarse-mesh low-fidelity models
as described in Section \ref{sec:LFM}. In particular, to define the LFM ODE of
the form \eqref{eq:approx_model}, we employ the same finite-volume scheme as
that employed by the FOM, but use a uniform control-volume width of  $\Delta
\spatialvar  = 2$ such that $\dimLF= 50$. The
prolongation operator $\prolongate$ consists of a linear interpolation of the
coarse-grid solution onto the fine grid.

We model the QoI error $\qoiErrorArg{\tindx}(\parametervec)$, where the QoI
corresponds to the solution at the midpoint of the domain such that the
associated functional is $ \qoiFunc(\statevec;t,\parametervec)\mapsto
\kroneckerVecArg{501}^T\statevec $.

\subsubsection{Coarse time grid, training, validation, and test sets}
We now describe how we execute Steps
\ref{step:timeIndexSet}--\ref{step:trainValTest} and \ref{train:noiseextract}
described in Section \ref{sec:trainingModelSel}. For Step
\ref{step:timeIndexSet}, we set $\timeIndexSet = \{8n\}_{n=1}^{100}$.
For Step \ref{step:trainValTest}, we employ the same approach as in the previous
section. Namely, we employ set sizes of
$\card{\parameterdomainTrain} = 40$, $\card{\parameterdomainVal}=10$, and
$\card{\parameterdomainTest} = 50$.  To assess the impact of the training set
size on the performance of the deterministic regression-function model, we
consider four smaller training and validation sets with
$(\card{\parameterdomainTrain},\card{\parameterdomainVal}) \in\{(8,2),(16,4),
(24,6), (32,8)\}$, which are constructed by sampling the original sets
$\parameterdomainTrain$ and $\parameterdomainVal$.  In Step
\ref{train:noiseextract}, we set $\card{\paramtrainsetNoise}=20$.

\subsubsection{Regression results}

Figures~\ref{fig:brian_converge} through~\ref{fig:brian_hist} summarize the
performance of the deterministic regression-function models and stochastic
noise models.  First, Figure~\ref{fig:brian_converge} reports the dependence
of the performance of each deterministic regression-function model (with its
best performing feature-engineering method) on the size of the training set
$\card{\parameterdomainTrain}$.  Figure~\ref{fig:brian_converge_1} considers
all feature-engineering methods, while Figure~\ref{fig:brian_converge_2} limits
consideration to feature-engineering methods that omit time from the set of
features.  We observe that LSTM and RNN again yield the best performance,
although their improvement over ANN is less significant than what was observed
in the previous two examples.  We observe that all other methods---including
the time-local GP benchmark---perform substantially worse than these three
techniques.


\begin{figure}
\begin{center}
\begin{subfigure}[t]{0.49\textwidth}
\includegraphics[width=1.\linewidth]{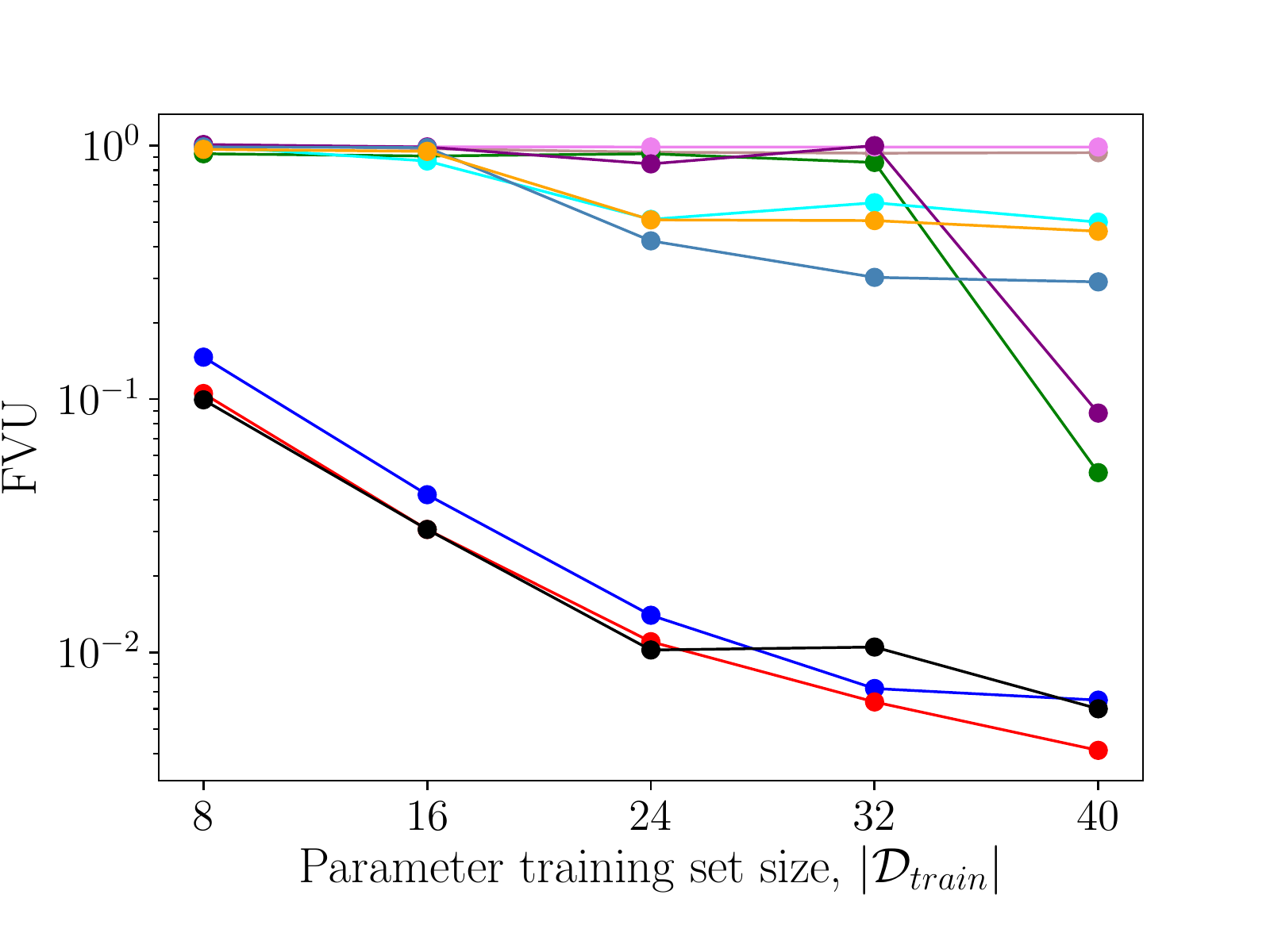}
\caption{QoI error $\qoiError$ (All feature-engineering methods)}
\label{fig:brian_converge_1}
\end{subfigure}
\begin{subfigure}[t]{0.49\textwidth}
\includegraphics[width=1.\linewidth]{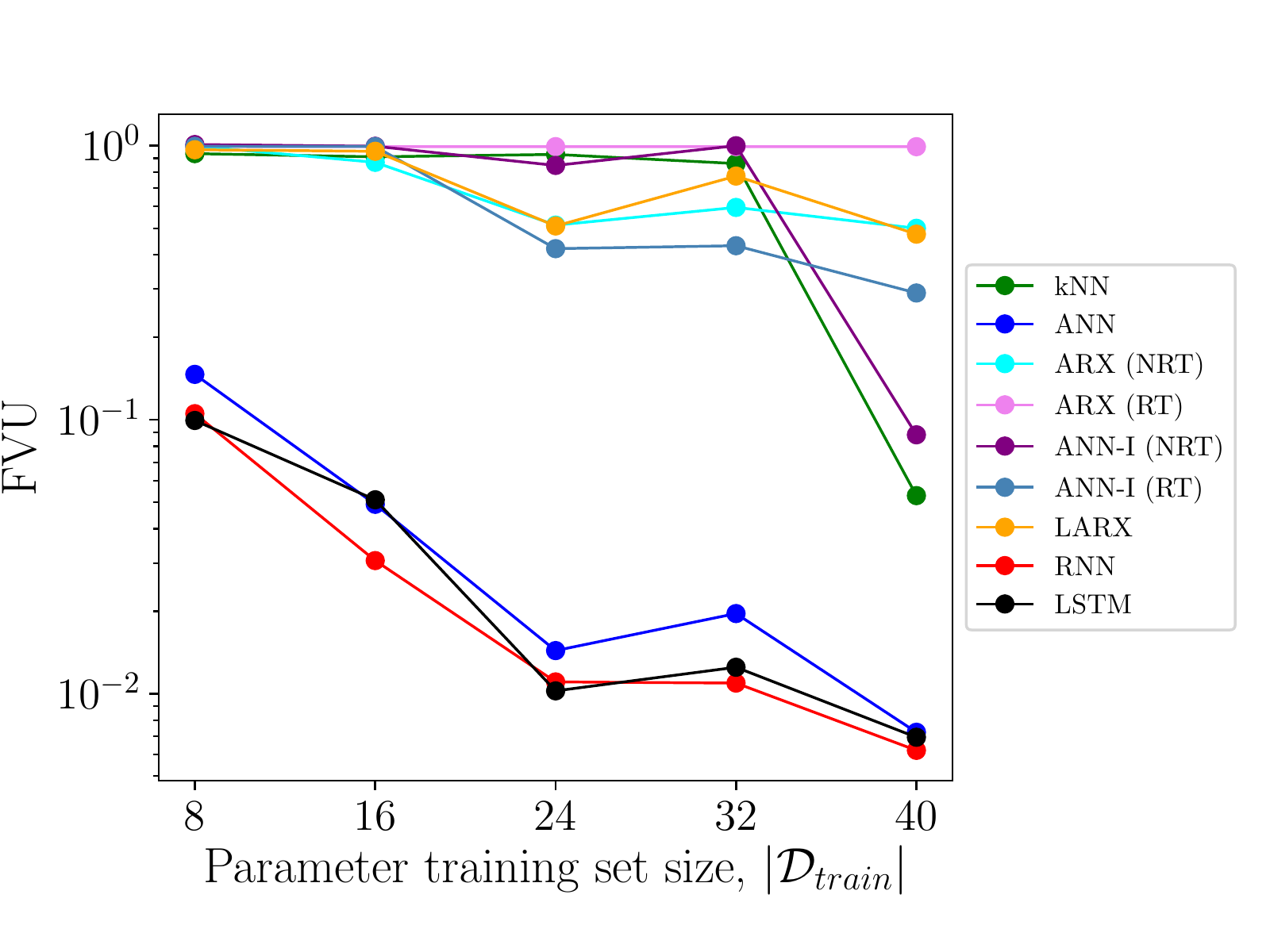}
\caption{QoI error $\qoiError$ (Feature-engineering methods without time)}
\label{fig:brian_converge_2}
\end{subfigure}
\end{center}
	\caption{\textit{Inviscid Burgers' equation}. Fraction of variance
	unexplained for all considered deterministic regression-function models,
	with each employing their best respective feature-engineering method.}
\label{fig:brian_converge}
\end{figure}

Next, Table~\ref{tab:brian_summary} summarizes the percentage of cases that a
deterministic regression-function model produces results with the lowest FVU,
where each case is defined by one of the 13 feature-engineering methods and
one of the 5 training-set sizes, yielding 65 total cases.  Results are
presented for the cases where (1) all feature methods are considered and (2)
only feature methods that don't include time are considered. The table shows
that LSTM is the best overall performing regression-function model, leading to
the lowest FVUs in over half of cases. RNN yields the next best
performance, followed by ANN.  We again observe that recursive regression
methods yield the best performance in the vast majority of cases; we also
again observe that including time can improve the relative performance of
non-recursive regression methods.

\begin{table}[ht]
\small{
\begin{center}
\setlength{\tabcolsep}{4pt}
\begin{tabular}{|c| c | c c c c c c c c c | c|}
	\hline
error	&
\begin{tabular}{c }
	include feat.\\
	eng.\ w/ time?
\end{tabular}
	&kNN & ANN & \begin{tabular}{c} ARX \\ (NRT) \end{tabular} & \begin{tabular}{c} ARX \\ (RT) \end{tabular} & \begin{tabular}{c} ANN-I \\ (NRT)\end{tabular} & \begin{tabular}{c} ANN-I \\ (RT)\end{tabular} & LARX &  RNN & LSTM & 
	\begin{tabular}{c}	
	Recursive\\ total
	\end{tabular}\\ \hline
		$\qoiError$   &Yes & 0.0 & 12.3 &0.0   & 0.0   & 0.0 & 0.0 & 1.5 & 35.4 & 50.8 & 87.7 \\ 
		$\qoiError$   &No  & 0.0 & 5.7 &0.0   & 0.0   & 0.0 & 0.0 & 2.9 & 34.3 & 57.1 & 94.3  \\ \hline
\end{tabular}
\end{center}
}
\caption{\textit{Inviscid Burgers' equation}. Percentage of cases having
	the lowest predicted FVU for each regression method. Results are summarized
	over all values $\card{\paramtrainindex} \in\{ 8,16,24,32,40\}$ and all
	feature-engineering methods. Recursive total reports the sum of all
	regression methods
	in \ref{cat:recursiveLatentPred} and 
	\ref{cat:recursive}. 
	}
\label{tab:brian_summary}
\end{table}


Figure~\ref{fig:brian_grid} summarizes the performance of each combination of
regression method and feature-engineering method for the case
$\card{\parameterdomainTrain} = 40$.  We once again observe that
feature-engineering methods employing elements of the residual lead to
the best performance. We also observe that the LSTM and RNN methods yield the best
performance, and substantially outperform the time-local GP method; indeed,
they generate FVU values around 0.1 when they use only the parameters
$\parametervec$ as features. These methods also generate FVU values lower than
0.01 when using only 46 residual samples in the set of features.

Figure~\ref{fig:brian_yhat_vs_t_qoi} shows the predictions made by the
different considered regression methods (for their best performing respective
feature-engineering method) as a function of time for a randomly drawn element
of $\parameterdomainTest$ for cases $\card{\paramtrainset} = 8$ and
$\card{\paramtrainset} = 40$.  Here, we observe that LSTM, RNN, and ANN
provide the most accurate predictions; indeed, their predictions are already
very accurate for only $\card{\paramtrainset} = 8$. The time-local GP benchmark and ARX
models perform the worst. 

Lastly, Figure~\ref{fig:brian_hist} and Table~\ref{tab:brian_hist} summarize
the performance of each stochastic noise model for the LSTM 
regression method employing its best performing
feature-engineering method.  Similar to the previous cases, the Laplacian
stochastic noise model provides the best stochastic noise model.

\begin{figure}
\begin{center}
\begin{subfigure}[t]{0.49\textwidth}
\includegraphics[width=1.\linewidth]{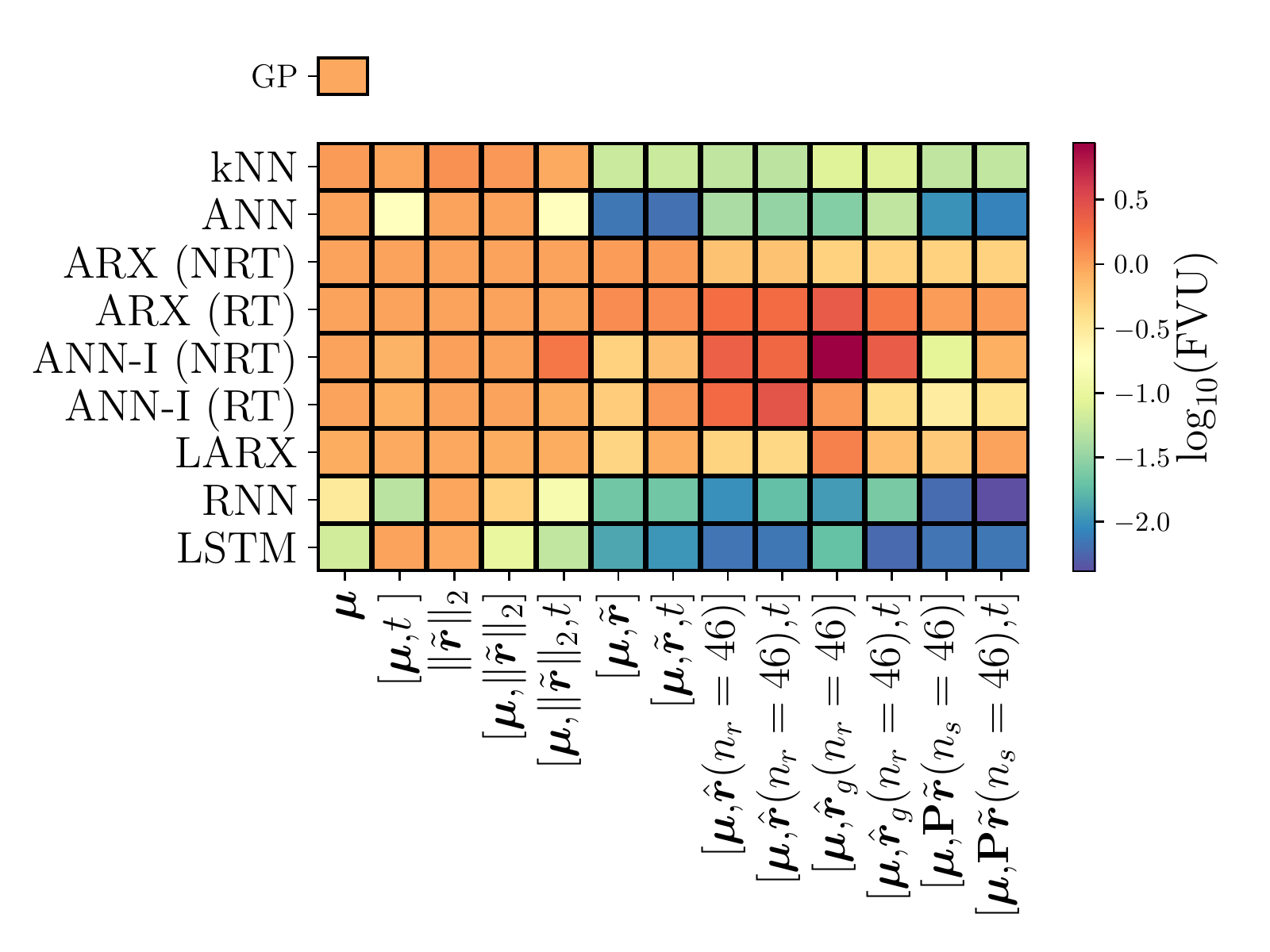}
	\caption{QoI error $\qoiError$ }
\end{subfigure}

\end{center}
\caption{\textit{Inviscid Burgers' equation}.
	Summary of the performance of
	each combination of regression method and feature-engineering method for the
	case
	$\card{\paramtrainindex} = 40$.}
\label{fig:brian_grid}
\end{figure}

\begin{figure}
\begin{center}
\begin{subfigure}[t]{0.49\textwidth}
\includegraphics[width=1.\linewidth]{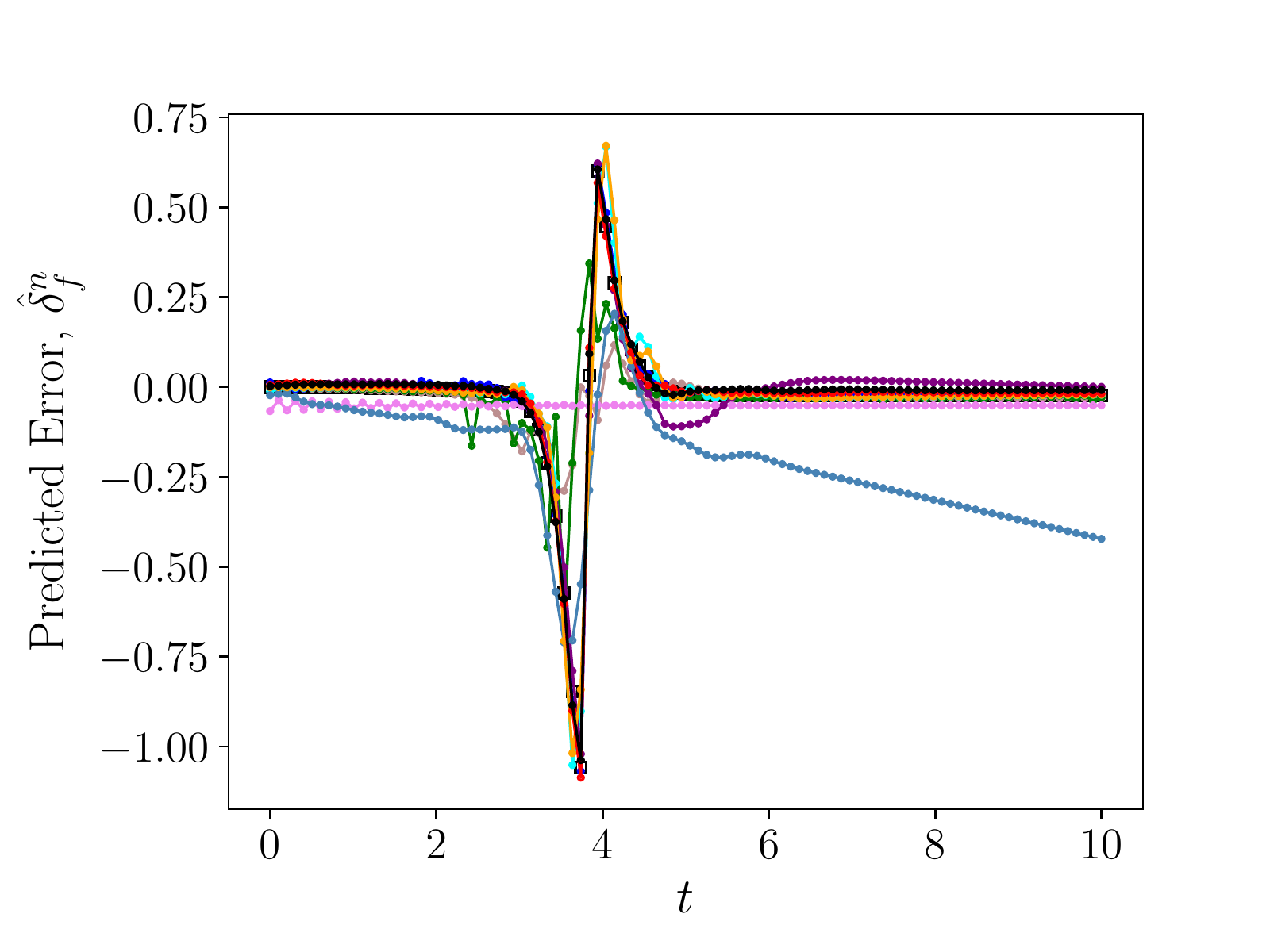}
	\caption{$\card{\paramtrainindex} = 8$}
\end{subfigure}
\begin{subfigure}[t]{0.49\textwidth}
\includegraphics[width=1.\linewidth]{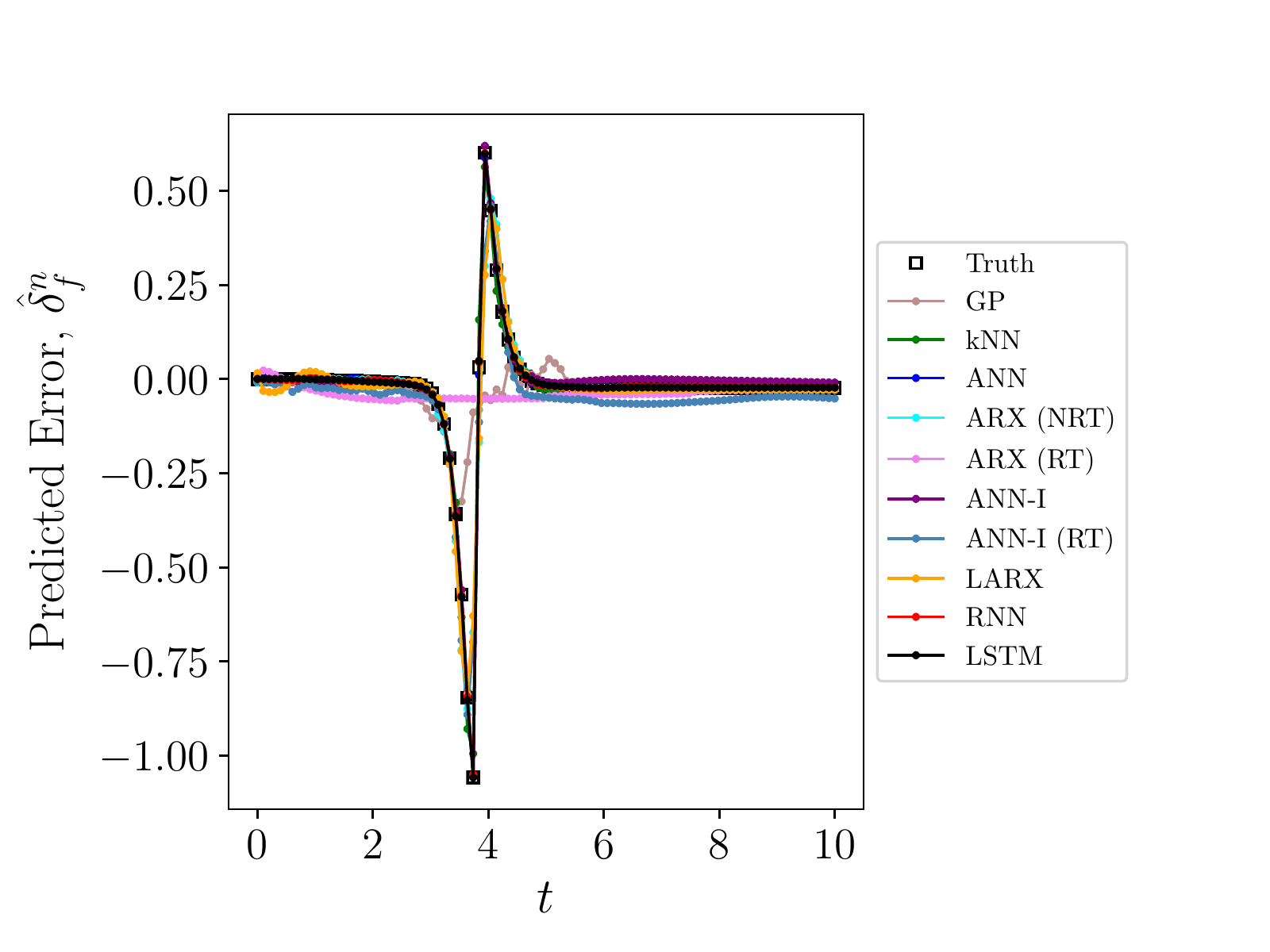}
	\caption{$\card{\paramtrainindex} = 40$}
\end{subfigure}
\end{center}
	\caption{\textit{Inviscid Burgers' equation}. 
	Error response $\qoiError$ 
	predicted by each regression method as a function of time. Results are shown
	for the best performing feature engineering method on the
	$\card{\paramtrainindex} = 8$ case (left) and the 
	$\card{\paramtrainindex} = 40$ case (right).
}
\label{fig:brian_yhat_vs_t_qoi}
\end{figure}

\begin{figure}
\begin{center}
\begin{subfigure}[t]{0.49\textwidth}
\includegraphics[width=1.\linewidth]{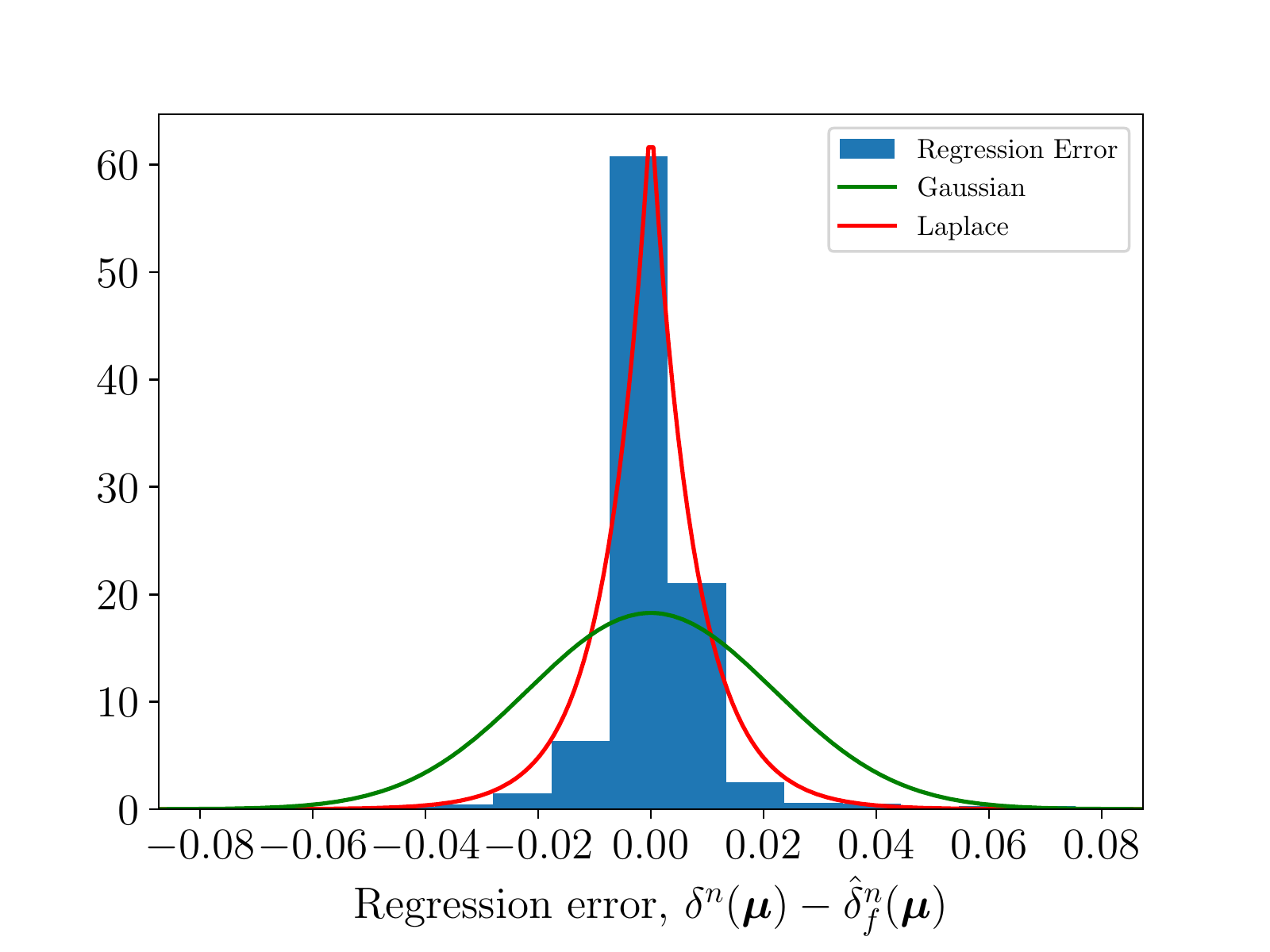}
\caption{QoI Error}
\end{subfigure}
\begin{subfigure}[t]{0.49\textwidth}
\includegraphics[width=1.\linewidth]{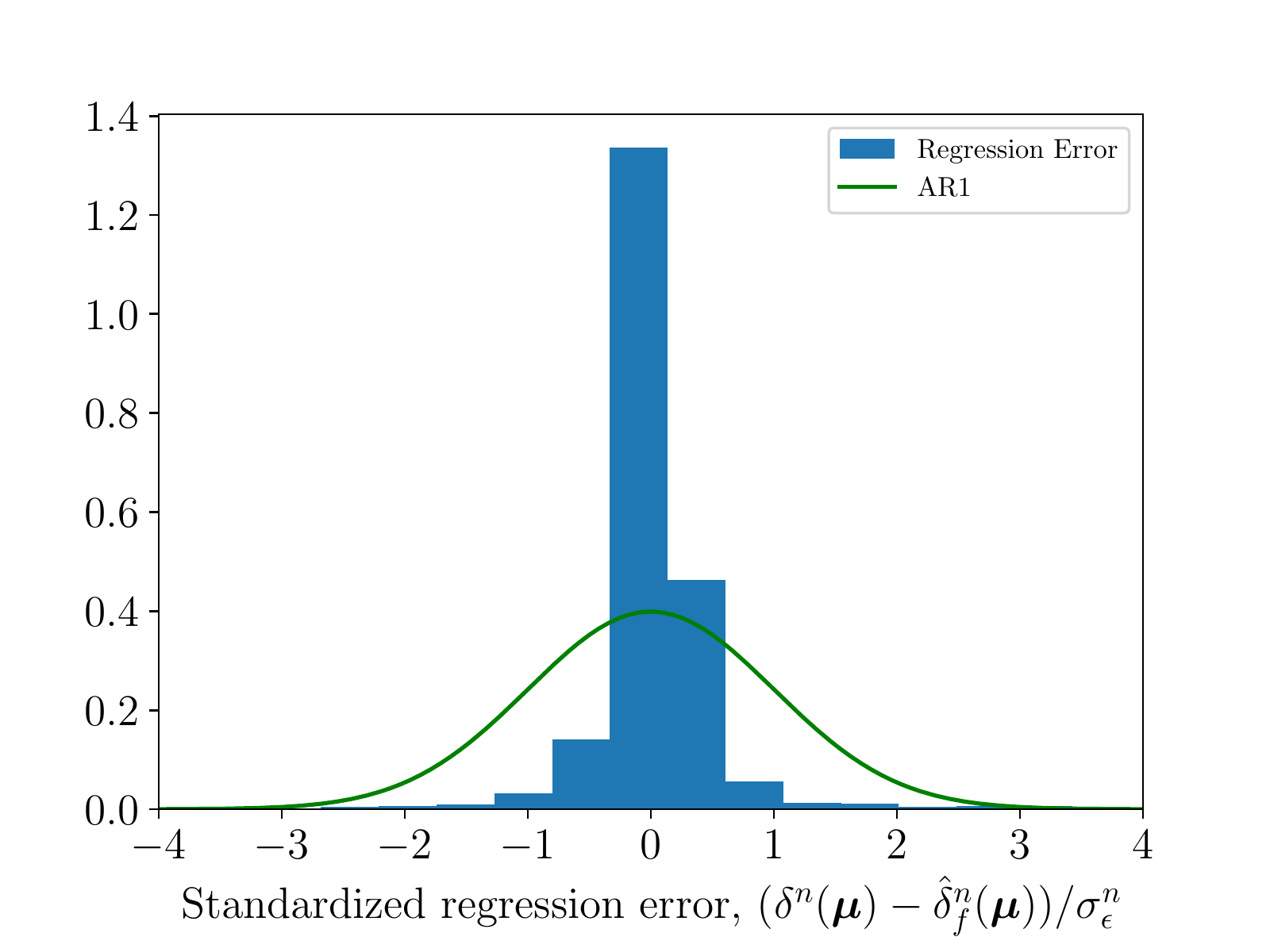}
\caption{Standardized QoI Error}
\end{subfigure}

\end{center}

\caption{\textit{Inviscid Burgers' equation.} Histogram of regression errors on the noise test set, $\testSetNoise$, for the QoI error. On the left, we show absolute errors $(\errorGennmuArg{n} - \errorApproxMeanGennmunArg{n})$, $\parametervec \in \paramtestsetNoise$, $n \in \timeIndexSet$. On the right, we show standardized errors, $(\errorGennmuArg{n} - \errorApproxMeanGennmunArg{n})/\sigma_{\epsilon}^n$, $\parametervec \in \paramtestsetNoise$, $n \in \timeIndexSet$. The true regression errors are compared to the distributions predicted by the Gaussian, Laplacian, and AR1 error models. Regression errors are shown 
for the LSTM deterministic regression function model with the best performing feature-engineering method on the $\card{\paramtrainindex} = 40$ case.
	}
\label{fig:brian_hist}
\end{figure}

\begin{table}[ht]
\begin{center}
	\begin{tabular}{c| c |c |c }
    \hline
    & Gaussian  & Laplacian & AR1 \\
    Prediction interval &  $\qoiError$ &    $\qoiError$ &  $\qoiError$   \\ \hline
 $\omega( 0.68 )$ &  0.943 & 0.810 & 0.933 \\  
 $\omega( 0.95 )$ & 0.970 & 0.946 & 0.959 \\
 $\omega( 0.99 )$ & 0.976 & 0.975 & 0.965\\
    \hline
 K-S Statistic & 0.273 & 0.111 & 0.274\\ \hline
\end{tabular}
\end{center}
\caption{
	\textit{Inviscid Burgers' equation}.
Prediction
	intervals for the various stochastic noise models.
Results correspond to those generated by the
	best performing 
	regression method (i.e., LSTM) with its  best performing feature-engineering method on the
	$\card{\paramtrainindex} = 40$ case, with samples collected over $\testSet$.
	}
\label{tab:brian_hist}
\end{table}

\section{Conclusions}\label{sec:conclude}

This work proposed the \methodName\ (\methodAcronym) method for constructing
error models for approximate solutions to parameterized dynamical systems. The
\methodAcronym\ method extends the machine learning error modeling
(MLEM)~\cite{freno_carlberg_ml} method to dynamical systems. The key contribution of
this work as compared to previous works on error modeling for dynamical
systems is the use of residual-based features and recursive regression methods
(e.g., recurrent neural networks, long short-term memory networks) that employ
latent variables. The use of these latent variables enables these regression
models to model dependence on temporally non-local quantities, which is
essential for modeling errors in this context (as illustrated by classical
approaches for error quantification).

Similar to the MLEM method, the \methodAcronym\ method entails four steps:
feature engineering, data generation, training a deterministic
regression-function model that maps from features to the conditional
expectation of the error, and training a stochastic 
noise model. For feature engineering, we considered a variety of techniques
based on those proposed in Ref.~\cite{freno_carlberg_ml}, along with
several new techniques that include the time variable as a feature. The
majority of these feature-engineering methods employ residual-based quantities
inspired by classic error-quantification approaches. Next, we considered
a hierarchy of regression methods that employ both non-recurrent and recurrent
regression functions as well as linear and nonlinear latent-variable dynamics. These
techniques include classic time-series modeling techniques (e.g., ARX), as
well as time-series modeling techniques emerging from the deep-learning
community (e.g., LSTM).  Lastly, we examined three different stochastic noise
models: Gaussian, Laplacian, and autoregressive Gaussian.

Across three benchmark problems and two different types of approximate
solutions, the numerical experiments illustrated the best overall performance
was obtained using the \methodAcronym\ framework with an LSTM
regression-function model. We attribute this favorable performance to LSTM's
ability to capture dependencies on long-term non-local quantities through its
use of a ``cell state'', as well as its amenability to training.   While the best-performing
feature-engineering method varied across the considered problems, generally
feature-engineering methods that employ residual-based features yielded the
best performance. Lastly, we observed the Laplacian stochastic noise model to
provide the best statistical model for the prediction noise. In
particular, we emphasize that 
the \methodAcronym\ method with an LSTM regression-function model and
residual-based features significantly outperformed both (1) the classical
model-discrepancy method \cite{ken_ohag}, which constructs an error model by applying
Gaussian-process regression in parameter space, and (2) error models
constructed using non-recursive regression models as was proposed in
Ref.~\cite{trehan_ml_error}. 

Future work involves the application of the \methodAcronym\ framework on truly
large-scale problems, as well as in more difficult scenarios, e.g., predicting
errors beyond the training time interval.

\section{Acknowledgements}\label{sec:acknowledge}
E.\ Parish acknowledges an appointment to the Sandia National Laboratories' 
John von Neumann Postdoctoral Research Fellowship in Computational Science. The 
authors thank Brian Freno for providing the data that was used in the third numerical experiment.
This work was partially sponsored by Sandia's Advanced Simulation and Computing (ASC)
Verification and Validation (V\&V) Project/Task \#103723/05.30.02. This paper
describes objective technical results and analysis. Any subjective views or
opinions that might be expressed in the paper do not necessarily represent the
views of the U.S.\ Department of Energy or the United States Government.
Sandia National Laboratories is a multimission laboratory managed and operated
by National Technology and Engineering Solutions of Sandia, LLC., a wholly
owned subsidiary of Honeywell International, Inc., for the U.S.\  Department of
Energy's National Nuclear Security Administration under contract
DE-NA-0003525.

\newpage
\begin{appendices}

\section{Proper orthogonal decomposition}\label{appendix:pod}
	Algorithm \ref{alg:pod} provides the algorithm for computing the POD basis
	used in this work. We note that the basis dimension $K$ can be determined
	from the decay of the singular values; for simplicity, we treat it as an
	algorithm input.
\begin{algorithm}
	\caption{Proper orthogonal decomposition (POD)}
\label{alg:pod}
	\textbf{Input}: Training parameter instances
	$\parameterdomainROM\equiv\{\parametervec_\text{ROM}^i\}\subseteq\parameterdomain$; number of time
	steps between collected snapshots $\nskip\leq\ntimesteps$; 
	reference state $\statevecIntercept(\param)$;
	desired basis
	dimension $K$.\\
	\textbf{Output}: POD basis $\basismat\equiv[\basismatVecArg{1}\ \cdots\
	\basismatVecArg{K}]\in\RRstar{N\times \dimrom}$. \\
	\textbf{Steps}:
\begin{enumerate}
	\item Solve the FOM O$\Delta$E \eqref{eq:fom_discretized} for
		$\parametervec\in\parameterdomainROM$ and collect the snapshot matrix
\begin{equation}\label{eq:snapshotDef}
	\snapshotArg{\parametervec,\nskip}\defeq\left[\statevec^{\nskip}(\parametervec) - 
\statevecIntercept(\parametervec)\ \cdots\ 
	\statevec^{\text{floor}(\ntimesteps/\nskip)\nskip}(\parametervec) - 
\statevecIntercept(\parametervec)\right].
\end{equation} 
\item Compute the (thin) singular value decomposition
\begin{equation} 
	\left[\snapshotArg{\parametervec_\text{ROM}^1,\nskip}\ \cdots\
	\snapshotArg{\parametervec_\text{ROM}^{\ntrainROM},\nskip}\right] = \svdLeft
	\svdMid \svdRight,
\end{equation} 
where
$\svdLeft\equiv\left[\svdLeftVecArg{1}\ \cdots\
\svdLeftVecArg{\ntrainROM\ntimesteps/\nskip}\right]$.
\item Truncate the left singular vectors such that 
$\basismatVecArg{i} = \svdLeftVecArg{i}$, $i=1,\ldots,K$.
\end{enumerate}

\end{algorithm}

\section{Q-sampling}\label{appendix:qsampling}
Feature engineering methods explored in this work that employ either the
	gappy principal components (\ref{feat:params_rgpca}) or
	sampled residual (\ref{feat:params_pr}) require the specification of
	the sampling matrix $\sampMat$. In this work, we compute this sampling matrix 
	via q-sampling~\cite{qdeim_drmac}. 
	Algorithm~\ref{alg:qdeim} provides the associated algorithm.
\begin{algorithm}
\caption{Algorithm for generation of sampling matrix through q-sampling.}
\label{alg:qdeim}
	\textbf{Input}: Residual training set $\residualTrainSet \equiv  \{ \resApproxArg{\tindx} \,|\,
\parametervec \in \parameterdomainTrain \equiv \{\parametervec_i \}_{i=1}^{\card{\paramtrainset} } 
, \tindx \in \timeIndexSetTrain \}$
\newline
	\textbf{Output}: Sampling matrix $\sampMat \in \{0,1\}^{\nsample\times N}$\;
\newline
Steps:
\begin{enumerate}
    \item Compute the residual snapshot matrix:
      $$\residualSnapshotMatrix \defeq \begin{bmatrix} \resApproxArgParam{\timeIndexMapping{1}}{\parametervec_1} & \cdots & 
       \resApproxArgParam{\timeIndexMapping{\ntimestepsCoarse}}{\parametervec_1} &\cdots & \resApproxArgParam{\timeIndexMapping{1}}{\card{\parametervec_{\paramtrainindex}}} & \cdots & \resApproxArgParam{\timeIndexMapping{\ntimestepsCoarse}}{\parametervec_{\card{\paramtrainindex}}}\end{bmatrix}.$$
    \item Compute the column mean of $\residualSnapshotMatrix$
    $$\overline{\resApprox}
\defeq \frac{1}{\card{\residualTrainSet} }
\sum_{\resApprox\in\residualTrainSet}\resApprox$$
    \item Compute the (thin) singular value decomposition:
    \begin{equation*}
     \residualSnapshotMatrix - \overline{\resApprox} \mathbf{e}^T  = \svdLeft \svdMid \svdRight,
    \end{equation*}
    where $\mathbf{e} \in \{1\}^{\card{\residualTrainSet}}$. 
    \item Compute the QR factorization of $\svdLeft^T$ with column pivoting
    \begin{equation*}
			\svdLeft^T [\sampMat^*]^T= \boldsymbol Q \boldsymbol R,
    \end{equation*}
    where $\sampMat^* \equiv \begin{bmatrix} \sampMatvec{1} & \cdots & \sampMatvec{\card{\residualTrainSet}} \end{bmatrix}^T$ , $\sampMatvec{i} \in \{0,1\}^{N}$.
    \item From the permutation matrix $\sampMat^*$, select the first $n_s \le \card{\residualTrainSet}$
			rows to form the sampling matrix
      $$ \sampMat  \defeq \begin{bmatrix} \sampMatvec{1} & \cdots & \sampMatvec{n_s} \end{bmatrix}^T.$$
\end{enumerate}

\end{algorithm}

\section{Gaussian-process regression}\label{appendix:gp}
The numerical experiments consider time-local Gaussian-process regression
models as a benchmark error-modeling method.  This method is analogous to
employing the ``model-discrepancy'' method of Kennedy and
O'Hagan~\cite{ken_ohag} at each time instance.  This approach constructs a
separate GP model to predict the error for each time instnace of interest. As
such, the approach can only be used to make predictions at time instances for
which training data exist, i.e., prediction beyond the training time interval
is not possible.  We note that the use of time-local GP regression was
considered in Ref.~\cite{pagani_enkfrom}, wherein GP models were constructed
in parameter space over time windows. 

We begin by setting
$\parameterdomainTrain\equiv\{\parametervecTrainArg{i}\}_{i=1}^{\nparamTrain}$,
where $\nparamTrain\defeq\card{\parameterdomainTrain}$,
and consider an arbitrary test point $\parametervecTest\in\parameterdomain$.
We define the vector of training responses and the \textit{model} of the test response
at time instance $n\in\timeIndexSet$ as
$$\errorGennAllArg{\tindx}(\parametervecTest) \defeq \begin{bmatrix}
\errorGennTrainArg{\tindx} \\
\errorApproxGennmuArg{\tindx}\end{bmatrix},$$
where
$\errorGennTrainArg{\tindx}\defeq[\errorGennArg{\tindx}(\parametervecTrainArg{1})\
\cdots\ \errorGennArg{\tindx}(\parametervecTrainArg{\nparamTrain})]^T$.
We then assume the data are drawn from the prior distribution
$$\errorGennAllArg{\tindx}(\parametervecTest) \sim \MC{N}(\boldsymbol 0,\gpCovarianceAll + \gpNoise \mathbf{I}),$$
where 
$\gpNoise\in\RR{}_+$ is an additive noise
hyper-parameter and
$\gpCovarianceAll$ is a positive definite
$(\nparamTrain+1)\times(\nparamTrain+1)$
covariance matrix 
of the form
$$
\gpCovarianceAll = \begin{bmatrix}
\gpCovarianceTrain & \gpCovarianceTest^T\\
\gpCovarianceTest & 1
\end{bmatrix},
$$
where we employ radial-basis-function kernels such that
\begin{align}
\begin{split}
&[\gpCovarianceTrain]_{ij} \defeq \exp\frac{
	\norm{\parametervecTrainArg{i} - \parametervecTrainArg{j}}^2
	}{\gpLength^2}, \qquad i,j = 1,\hdots,\nparamTrain
,\\
&
[\gpCovarianceTest]_{1i} = \exp\frac{ \norm{\parametervecTest
- \parametervecTrainArg{i}}^2 }{\gpLength^2}, \quad i
=1,\ldots,\nparamTrain,
\end{split}
\end{align}
Given the training data, the posterior distribution of the model of the test response is
$$
\errorApproxGennmuArg{\tindx}\sim
\normalDistribution \bigg(\gpCovarianceTest( \gpCovarianceTrain + \gpNoise \mathbf{I})^{-1} \errorGennTrainArg{\tindx} ,1 + \gpNoise - \gpCovarianceTest (\gpCovarianceTrain + \gpNoise \mathbf{I})^{-1} ( \gpCovarianceTest )^T \bigg),$$

Comparing with Eq.~\eqref{eq:error_split}, we see that this model corresponds
to the following quantities in the present \methodAcronym\ framework:
\begin{align}
\begin{split}
&\errorApproxMeanGennmunArg{\tindx} = \gpCovarianceTest( \gpCovarianceTrain +
\gpNoise \mathbf{I})^{-1} \errorGennTrainArg{\tindx}\\
&\errorApproxRandGennmunArg{\tindx}  \sim \normalDistribution \bigg(0,1 + \gpNoise - \gpCovarianceTest (\gpCovarianceTrain + \gpNoise \mathbf{I})^{-1} (\gpCovarianceTest)^T \bigg).
\end{split}
\end{align}

This model requires specifying parameters $\gpNoise$ and $\gpLength$. We
compute the
length-scale parameter $\gpLength$ by maximizing the log-marginal
likelihood of the posterior distribution (see Ref.~\cite{gp_book}), while we
treat the noise
parameter $\gpNoise$ as a hyper-parameter that is selected using a validation
procedure with a grid search.

\section{Keras algorithms}
We implement the ANN, ARX, ANN-I, LARX, RNN, and LSTM models using Keras. This appendix
provides the Python code that was used to construct these networks. For ARX and ANN-I, 
only networks corresponding to the RT training formulation are presented.
\subsection{Artificial neural network (ANN)}
\begin{lstlisting}[language=Python,caption=Python code for the initialization
of ANN.,label={listing:NN}]
def build_nn(X,y,hyperparams):
  neurons,depth,reg = hyperparams[0],np.int(hyperparams[1]),hyperparams[2]
  model = Sequential()
  model.add(Dense(neurons,input_shape=(None,np.shape(X)[-1]),activation='relu',use_bias=True,kernel_regularizer=regularizers.l2(10.**reg)))
  for i in range(1,depth):
    model.add(Dense(neurons,activation='relu',use_bias=True,
      kernel_regularizer=regularizers.l2(10**reg)))
  model.add(Dense(np.shape(y)[-1],activation='linear',use_bias=True))
  model.compile(loss='mean_squared_error', optimizer='Adam')
  return model
\end{lstlisting}

\subsection{Autoregressive w/ Exogenous Inputs (ARX)}
\begin{lstlisting}[language=Python,caption=Python code for the initialization of ARX., label={listing:ARX}]
def build_model_arx(X,y,hyperparams):
  neurons,reg = int(hyperparams[0]),np.int(hyperparams[1])
  model = Sequential()
  model.add(SimpleRNN(np.shape(y)[-1],input_shape=(None,np.shape(X)[-1]),return_sequences=True,activation='linear',kernel_regularizer=regularizers.l2(10**reg),recurrent_regularizer=regularizers.l2(10**reg)) )
  model.compile(loss='mean_squared_error', optimizer='Adam')
  return model
\end{lstlisting}

\subsection{Integrated Artificial Neural Network (ANN-I)}
\begin{lstlisting}[language=Python,caption=Python code for the initialization of ANN-I., label={listing:NNI}]
def build_model_nni(X,y,hyperparams):
  neurons,depth,reg = hyperparams[0],np.int(hyperparams[1]),hyperparams[2]
  model = Sequential()
  model.add(Dense(neurons,input_shape=(None,np.shape(X)[-1]),activation='relu',use_bias=True,kernel_regularizer=regularizers.l2(10.**reg)))
  for i in range(1,depth):
    model.add(Dense(neurons,activation='relu',use_bias=True,kernel_regularizer=regularizers.l2(10**reg)))
  model.add(Dense(np.shape(y)[-1],activation='linear',use_bias=True))
  model.add(SimpleRNN(np.shape(y)[-1],return_sequences=True ,activation='linear',use_bias=False,recurrent_initializer='ones',kernel_initializer='ones',name='RNN'))
  model.get_layer('RNN').trainable = False
  model.compile(loss='mean_squared_error', optimizer='Adam')
  return model
\end{lstlisting}

\subsection{Latent Auto-Regressive w/ Exogenous Inputs (LARX)}
\begin{lstlisting}[language=Python,caption=Python code for the initialization
of LARX.,label={listing:LARX}]
def build_model_rnn_linear(X,y,hyperparams):
  neurons,reg = int(hyperparams[0]),hyperparams[1]
  model = Sequential()
  model.add(SimpleRNN(neurons,input_shape=(None,np.shape(X)[-1]),return_sequences=True,activation='linear',kernel_regularizer=regularizers.l2(10**reg),recurrent_regularizer=regularizers.l2(10**reg)) )
  model.add(Dense(np.shape(y)[-1],activation='linear',use_bias=True))
  model.compile(loss='mean_squared_error', optimizer='Adam')
  return model
\end{lstlisting}

\subsection{Recurrent neural network (RNN)}
\begin{lstlisting}[language=Python,caption=Python code for the initialization
of the RNN.,label={listing:RNN}]
def build_rnn(X,y,hyperparams):
  neurons,depth,reg = int(hyperparams[0]),np.int(hyperparams[1]),hyperparams[2]
  model = Sequential()
  model.add(SimpleRNN(neurons,input_shape=(None,np.shape(X)[-1]), 
   return_sequences=True , kernel_regularizer=regularizers.l2(10**reg), 
   recurrent_regularizer=regularizers.l2(10**reg)) )
  for i in range(1,depth):
    model.add(SimpleRNN(neurons,
     return_sequences=True,
     kernel_regularizer=regularizers.l2(10**reg) ,
     recurrent_regularizer=regularizers.l2(10**reg)) )
  model.add(Dense(np.shape(y)[-1],activation='linear',use_bias=True))
  model.compile(loss='mean_squared_error', optimizer='Adam')
  return model
\end{lstlisting}

\subsection{Long short-term memory network (LSTM)}
\begin{lstlisting}[language=Python,caption=Python code for the initialization of LSTM.,label={listing:LSTM}]
def build_lstm(X,y,hyperparams):
  neurons,depth,reg = int(hyperparams[0]),np.int(hyperparams[1]),hyperparams[2]
  model = Sequential()
  model.add(LSTM(neurons,input_shape=(None,np.shape(X)[-1]), return_sequences=True , kernel_regularizer=regularizers.l2(10**reg) , recurrent_regularizer=regularizers.l2(10**reg)) )
  for i in range(1,depth):
    model.add(LSTM(neurons,return_sequences=True,kernel_regularizer=regularizers.l2(10**reg),recurrent_regularizer=regularizers.l2(10**reg)) )
  model.add(Dense(np.shape(y)[-1],activation='linear',use_bias=True))
  model.compile(loss='mean_squared_error', optimizer='Adam')
  return model
\end{lstlisting}

\end{appendices}

\bibliographystyle{siam}
\bibliography{refs}
\end{document}